\documentclass{article}
\usepackage{amsmath,amsfonts,amsthm,amssymb,amscd,color,xcolor,mathrsfs,tikz}
 \usepackage{hyperref}
\colorlet{darkblue}{blue!50!black}

\hypersetup{
    colorlinks,%
    citecolor=darkblue,%
    filecolor=red,%
    linkcolor=darkblue,%
    urlcolor=magenta,%
    pdfnewwindow=true,%
    pdfstartview={FitH}
}

\renewcommand{\Re}{\mathop{\rm Re}\nolimits}

\newcommand{\RInt}{\mathop{\rm ri}\nolimits}

\newcommand{\p}{\partial}
\newcommand{\e}{\varepsilon}

\newcommand{\C}{{\mathbb C}}

\newcommand{\R}{{\mathbb R}}
\newcommand{\Z}{{\mathbb Z}}
\newcommand{\IP}{{\mathbb P}}
\newcommand{\pP}{{\mathbb P}}

\newcommand{\E}{{\mathbb E}}

\newcommand{\N}{{\mathbb N}}
\newcommand{\la}{\lambda}

\newcommand{\ty}{\infty}

\newcommand{\om}{\omega}

\newcommand{\de}{\delta}

\newcommand{\pppp}{{\mathfrak p}}
\newcommand{\wwww}{{\mathfrak w}}
\newcommand{\mmmm}{{\mathfrak m}}

\newcommand{\bbar}{\boldsymbol{|}}

 \newcommand{\HHH}{{\boldsymbol{H}}}

\newcommand{\uu}{{\boldsymbol{ {u}}}}

 \newcommand{\Ss}{{\boldsymbol {S}}}
\newcommand{\eeta}{{\boldsymbol{{\eta}}}}

\newcommand{\zzeta}{{\boldsymbol{{\zeta}}}}

\newcommand{\III}{{\boldsymbol{I}}}

\newcommand{\LLambda}{{\boldsymbol\Lambda}}

\newcommand{\ttau}{{\boldsymbol\tau}}

\newcommand{\BB}{{\cal B}}
\newcommand{\CC}{{\cal C}}
\newcommand{\DD}{{\cal D}}

\newcommand{\FF}{{\cal F}}

\newcommand{\HH}{{\cal H}}
\newcommand{\KK}{{\cal K}}
\newcommand{\LL}{{\cal L}}
\newcommand{\MM}{{\cal M}}

\newcommand{\PP}{{\cal P}}

\newcommand{\RR}{{\cal R}}

\newcommand{\VV}{{\cal V}}

\newcommand{\lag}{\langle}
\newcommand{\rag}{\rangle}

\newcommand{\dd}{{\textup d}}

\newcommand{\PPPP}{{\mathfrak P}}
\newcommand{\SSSS}{{\mathfrak S}}

\newcommand{\BBBB}{{\mathfrak B}}

\newcommand{\SSS}{{\mathscr S}}

\newcommand{\uuu}{{\boldsymbol{\mathit u}}}
\newcommand{\nnn}{{\boldsymbol{\mathit n}}}
\newcommand{\vvv}{{\boldsymbol{\mathit v}}}

\newcommand{\zzz}{{\boldsymbol{\mathit z}}}

\newcommand{\lspan}{\mathop{\rm span}\nolimits}

\newcommand{\var}{\mathop{\rm Var}\nolimits}
\newcommand{\supp}{\mathop{\rm supp}\nolimits}
\newcommand{\esssup}{\mathop{\rm ess\ sup}\nolimits}
\newcommand{\diver}{\mathop{\rm div}\nolimits}

\newcommand{\Osc}{\mathop{\rm Osc}\nolimits}

\theoremstyle{plain}

\newtheorem*{mtheorem}{Main Theorem}
\newtheorem{theorem}{Theorem}[section]
\newtheorem{lemma}[theorem]{Lemma}
\newtheorem{proposition}[theorem]{Proposition}
\newtheorem{corollary}[theorem]{Corollary}
\theoremstyle{definition}
\newtheorem{definition}[theorem]{Definition}

\theoremstyle{remark}
\newtheorem{remark}[theorem]{Remark}

\numberwithin{equation}{section}

\begin{document}
\author{V.~Jak\v si\'c\footnote{Department of Mathematics and Statistics,
McGill University, 805 Sherbrooke Street West, Montreal, QC, H3A 2K6
Canada; e-mail: \href{mailto:Jaksic@math.mcgill.ca}{Jaksic@math.mcgill.ca}}
\and V. Nersesyan\footnote{Laboratoire de Mat\'ematiques, UMR CNRS 8100, Universit\'e de Versailles-Saint-Quentin-en-Yvelines, F-78035 Versailles, France;  e-mail: \href{mailto:Vahagn.Nersesyan@math.uvsq.fr}{Vahagn.Nersesyan@math.uvsq.fr}}
\and C.-A.~Pillet\footnote{Aix Marseille Universit\'e, CNRS, CPT, UMR 7332, Case 907, 13288 Marseille, France; 
Univerist\'e de Toulon, CNRS, CPT, UMR 7332, 83957 La Garde,  France; e-mail: \href{mailto:Pillet@univ-tln.fr}{Pillet@univ-tln.fr}}
\and A.~Shirikyan\footnote{Department of Mathematics, University of Cergy--Pontoise, CNRS UMR 8088, 2 avenue Adolphe Chauvin, 95302 Cergy--Pontoise, France; e-mail: \href{mailto:Armen.Shirikyan@u-cergy.fr}{Armen.Shirikyan@u-cergy.fr}}$\ ^{,*}$}

\title{Large deviations and mixing for dissipative PDE's with unbounded random kicks}
\date{}
%\date{\today}
\maketitle

\begin{abstract}
We study the problem of exponential mixing and large deviations for  discrete-time Markov processes associated with a class of random dynamical systems. Under some dissipativity and regularisation hypotheses for the underlying deterministic dynamics and a non-degeneracy condition for the driving random force, we discuss the existence and uniqueness of a stationary measure and its exponential stability in the Kantorovich--Wasserstein metric. We next turn to the large deviation principle and establish its validity for the occupation measures of the Markov processes in question. The obtained results  extend those established in~\cite{JNPS-2012} for the case of the bounded noise and  can be applied to the 2D Navier--Stokes system in a bounded domain and to the complex Ginzburg--Landau equation. 
 

\end{abstract}

\tableofcontents

\setcounter{section}{-1}

\section{Introduction}
\label{s0} 
The goal of the present paper is to study the problem of large deviations from a stationary distribution for a class of PDE's perturbed by a  smooth random force. This question was already investigated by the authors in~\cite{JNPS-2012} for PDE's with bounded perturbations in which case the validity of the large deviation principle (LDP) was established for initial data belonging to the support of the unique stationary measure. In this paper, we extend that result to the situation in which the random noise is unbounded. Let us mention straightaway that, as is well known from the case of a locally compact phase space (e.g., see the discussion in the introduction of~\cite{DV-1976}), the generalisation of the LDP to the unbounded case involves some new phenomena, and the mere verification of exponential tightness is far from being sufficient. 

To describe the main result of this paper, we consider a bounded domain $D\subset\R^2$ with $C^2$-smooth boundary~$\p D$ and the Navier--Stokes system in~$D$ supplemented with the Dirichlet boundary condition: 
\begin{equation} \label{0.1}
\dot u+\langle u,\nabla \rangle u-\nu\Delta u+\nabla p=h(x)+\eta(t,x), \quad\diver u=0, \quad u\bigr|_{\p D}=0. 
\end{equation}
Here $u=(u_1,u_2)$ and $p$ are unknown velocity field and pressure of the fluid, 
$$
\langle u,\nabla \rangle=u_1\p_1+u_2\p_2, \quad \Delta=\p_1^2+\p_2^2,
$$
and~$h$ and~$\eta$ are deterministic and random external forces, respectively. Let us introduce the phase space
\begin{equation} \label{0.3}
H=\bigl\{u\in L^2(D,\R^2):\diver u=0\mbox{ in $D$, $\langle u,\nnn\rangle=0$ on $\p D$}\bigr\},
\end{equation}
where~$\nnn$ stands for the outward unit normal vector to~$\p D$, and endow it with the usual $L^2$ norm. We assume that $h\in H$ and that~$\eta$ is a kick process of the form 
\begin{equation} \label{0.2}
\eta(t,x)=\sum_{k=1}^\infty \eta_k(x)\delta(t-k),
\end{equation}
where $\delta(t)$ stands for the Dirac measure on~$\R$ concentrated at zero and~$\{\eta_k\}$ is a sequence of i.i.d.\ random variables in~$H$. Problem~\eqref{0.1}, \eqref{0.2} generates a discrete-time Markov process in~$H$. Namely, for any $u_0\in H$, there is a unique solution $u(t,x)$ for~\eqref{0.1}, \eqref{0.2} that is right-continuous in time and  satisfies the initial condition
\begin{equation} \label{0.4}
u(0,x)=u_0(x),
\end{equation}
and the restrictions of all solutions to the non-negative integers~$\Z_+$ form a discrete-time Markov process, which will be denoted by~$(u_k,\IP_u)$. The ergodic properties of this process are by now well understood; see~\cite{FM-1995,KS-cmp2000,EMS-2001,BKL-2002} for the first results on the ergodic behaviour of the random flow generated by~\eqref{0.1} and the book~\cite{KS-book} for further references. In particular, there is a unique stationary distribution~$\mu$, which attracts exponentially fast all other trajectories, and the strong law of large numbers holds for a broad class of H\"older-continuous functionals $f:H\to \R$ which may grow at infinity. The latter means that
$$
\IP_u\biggl\{\frac1k\sum_{n=0}^{k-1}f(u_n)\to \langle f,\mu\rangle\biggr\}=1\quad\mbox{for any $u\in H$},
$$
where $ \langle f,\mu\rangle$ denotes the integral over~$H$ of the function~$f$ with respect to~$\mu$. 
Now let $\PP(H)$ be the set of all probability measures on~$H$. Our aim is to study the asymptotic behaviour, as $k\to\infty$, of the probabilities
$$
\varPsi_k(\lambda,f,\Gamma):=\IP_\lambda\biggl\{\frac1k\sum_{n=0}^{k-1}f(u_n)\in \langle f,\mu\rangle+\Gamma\biggr\}, 
\quad \lambda\in\PP(H),
$$
where $\IP_\lambda$ stands for the probability law corresponding to the initial distribution~$\lambda$, and  $\Gamma\subset \R$ is a Borel subset. The following theorem, which is a particular case of the main result of this paper, gives a complete description of the large-time behaviour of~$\varPsi_k(\lambda,f,\Gamma)$ on the logarithmic scale and establishes the level-1 LDP for the Markov process~$(u_k,\IP_u)$. We refer the reader to Section~\ref{s1.3} for more general results in an abstract setting and to Section~\ref{s2} for their application to the 2D Navier--Stokes system and the complex Ginzburg--Landau equation. 

\begin{mtheorem}
Let us assume that the i.i.d.\ random variables~$\eta_k$ are distributed according to a non-degenerate Gaussian law on~$H$ concentrated on the Sobolev space of order~$2$ and let $f:H\to\R$  be a continuous function that is bounded on balls of~$H$ and is such that
\begin{equation} \label{0.5}
\frac{|f(v)|}{\|v\|^2}\to0\quad\mbox{as $\|v\|\to\infty$}. 
\end{equation}
Then there is a function $I_f:\R\to[0,+\infty]$ with compact level sets such that, for any initial measure $\lambda\in\PP(H)$ satisfying the condition
\begin{equation} \label{0.6}
\int_He^{\delta \|v\|^2}\lambda(\dd v)<\infty\quad\mbox{for some $\delta>0$},
\end{equation}
we have
\begin{equation} \label{0.7}
-\inf_{y\in\dot \Gamma} I_f(y)\le \liminf_{k\to\infty}\frac1k\log\varPsi_k(\lambda,f,\Gamma)
\le  \limsup_{k\to\infty}\frac1k\log\varPsi_k(\lambda,f,\Gamma)\le -\inf_{y\in\overline{\Gamma}}I_f(y),
\end{equation}
where $\Gamma\subset\R$ is an arbitrary Borel set, and $\dot \Gamma$ {\rm(}respectively, $\overline{\Gamma}${\rm)} stands for its interior {\rm(}closure{\rm)} in~$H$. In particular, inequalities~\eqref{0.7} are true for the stationary solution, as well as for solutions issued from any deterministic point $v\in H$.  
\end{mtheorem}

Let us emphasise that the Gaussian structure of the noise does not play any role, and the result remains valid for a large class of decomposable measures (see Condition~(C) in Section~\ref{s1.1}). Moreover, the LDP holds also for occupation measures of solutions (level-2) and of full trajectories (level-3). 

Before presenting the main ideas of the proof of the above theorem, we briefly mention some earlier results related to the present work. The theory of large deviations from a stationary measure for Markov processes was initiated by Donsker and Varadhan~\cite{DV-1975,DV-1976}, who carried out a comprehensive study of the problem for strong Feller processes with a compact and later a general phase space. Their results were developed and extended by many others; e.g., see the papers~\cite{deacosta-1990,jain-1990,BDT-1995,wu-2000} and the references in~\cite{DS1989}. These works concern the case in which the Markov process in question possesses the strong Feller property. In the context of randomly forced PDE's, the problem of large deviation from a stationary measure was studied in~\cite{gourcy-2007a,gourcy-2007b,JNPS-2013} for various types of rough noise and in~\cite{JNPS-2012} for a smooth kick noise. The present work continues the investigation started in the latter  paper and extends its result to the case of unbounded perturbations. 

We now discuss some ideas of the proof and describe the main novelty of this paper. As in the case of a bounded perturbation, 
the proofs are based on Kiffer's criterion~\cite{kifer-1990} and the key points are the existence of the pressure function and the uniqueness of equilibrium state. These two properties are closely related to the large-time asymptotics of a generalised Markov semigroup $\PPPP_k^V:C_b(H)\to C_b(H)$ defined by
\begin{equation} \label{0.8}
(\PPPP_k^Vf)(u)=\lim_{k\to\infty}\frac1k\log\E_u\bigl\{\exp\bigl(V(u_1)+\cdots+V(u_k)\bigr)f(u_k)\bigl\},\quad f\in C_b(H),
\end{equation}
where~$C_b(H)$ stands for the space of bounded continuous real-valued functions on~$H$,  $\E_u$ denotes the mean value corresponding to the trajectory issued from~$u$,  and~$V\in C_b(H)$ is an arbitrary function. In the case when the random variables~$\eta_k$ are bounded, it was proved in~\cite{JNPS-2012} that a sufficient condition for the existence and stability of a maximal eigenvector for~$\{\PPPP_k\}$ is given by the uniform irreducibility and uniform Feller properties (introduced in~\cite{KS-cmp2000}). In the Markovian situation, stratified versions of these properties are also sufficient; see~\cite{KS-mpag2001}. The situation is different in the case of~\eqref{0.8}, for which one needs to require in addition the following {\it growth condition\/} (which seems to be new in this context): there is an integer $m\ge1$ and positive numbers~$R_0$ and~$C$ such that
\begin{equation} \label{0.9}
\sup_{u\in H}\bigl\{\wwww_m(u)^{-1}(\PPPP_k^V\wwww_m)(u)\bigr\}\le C\sup_{\|u\|\le R_0}(\PPPP_k^V{\mathbf1})(u)
\quad\mbox{for all $k\ge0$},
\end{equation}
where $\mathbf1$ is the function identically equal to~$1$ and $\wwww_m(u)=(1+\|u\|)^m$. The verification of this and uniform Feller properties are the key points of this work. They are based on the hyper-exponential recurrence and a coupling argument, respectively. We refer the reader to Section~\ref{s1.5} for more details on the proof of our main result.  

\medskip
The paper is organised as follows. In Section~\ref{s1}, we describe the abstract model we deal with, formulate our main result in full generality, and outline its proof. Applications of the main theorem to concrete examples of randomly forced PDE's are discussed in Section~\ref{s2}. In Sections~\ref{s4} and~\ref{s5}, we present two auxiliary results. The first of them is an extension of Kifer's criterion to an unbounded phase space, and the second is a general result on large-time asymptotics of generalised Markov semigroups. The details of the proof of the main result are given in Sections~\ref{s8}--\ref{s7}, and the appendix gathers some auxiliary assertions used in the main text.

\subsection*{Acknowledgments} 
The research of VJ was supported by NSERC.
The research of VN was supported by the ANR grants EMAQS (No. ANR 2011 BS01 017 01) and STOSYMAP (No.~ANR 2011 BS01 015 01).  The research of AS was carried out within the MME-DII Center of Excellence (ANR-11-LABX-0023-01) and supported by the ANR grant STOSYMAP and RSF research project 14-49-00079. This paper was finalised when AS was visiting the Mathematics and Statistics Department of the University of McGill, and he thanks the institution for hospitality and excellent working conditions. 

\subsection*{Notation}
Given a Polish space $X$, we denote by $B_X(a,R)$ (respectively, $\mathring{B}_X(a,R)$) the closed (open) ball in~$X$ of radius~$R$ centred at~$a$. In the case when~$X$ is a Banach space and~$a=0$, we write $B_X(R)$ (respectively, $\mathring{B}_X(R)$). We denote by~$\delta_a$ the Dirac mass concentrated at the point $a$ and  by~$\DD(\xi)$ the law of a random variable~$\xi$. Given  a function $V:X\to\R$, we write $\Osc_X(V)=\sup_X V-\inf_X V$. We denote by~$\Z_+$ (respectively, $\Z_-$)  the set of non-negative (non-positive) integers and by~$\R_+$ be the set of non-negative real numbers.

Let $J\subset\R$ be a closed interval and $D\subset\R^d$ be a domain. We shall use the following spaces:

\smallskip
\noindent
$L^\infty(X)$ is the space of bounded measurable functions $f:X\to\R$ endowed with the norm $\|f\|_\infty=\sup_X|f|$. 

\smallskip
\noindent
$C_b(X)$ is the space of continuous functions $f\in L^\infty(X)$.

\smallskip
\noindent
$C(X)$ is the space of continuous functions $f:X\to\R$. $C_+(X)$  is the set of non-negative functions in $C(X)$.

\smallskip
\noindent
$\MM_+(X)$ is the cone of non-negative finite measures on the space~$X$ endowed with the Borel $\sigma$-algebra~$\BB(X)$. For $\mu\in\MM_+(X)$
and an integrable function $f:X\to\R$ we set
$$
\lag f,\mu\rag=\int_X f(u)\mu(\dd u),\qquad
|f|_\mu=\int_X |f(u)|\mu(\dd u).
$$

\smallskip
\noindent
$\PP(X)$ is the set of probability Borel measures on~$X$. We endow it with the Kantorovich--Wasserstein metric denoted by~$\|\cdot\|_L^*$; e.g., see~(1.14) in~\cite{KS-book}. 

\smallskip
\noindent
$H^s(D)$ is the Sobolev space of order~$s\ge0$ on the domain~$D$. We denote by~$\|\cdot\|_s$ the usual  Sobolev norm. In the case $s=0$, we write $L^2(D)$ and~$\|\cdot\|$, respectively. The spaces of scalar and vector functions are denoted by the same symbol, and we often write~$H^s$ and~$L^2$ when the context implies in which domains the spaces are considered. 

\smallskip
\noindent
$H_0^s(D)$ is the closure in~$H^s(D)$ of the space of smooth functions with compact support in~$D$.  

\smallskip
\noindent
$C_b^k(D)$ is the space of $k$ times continuously differentiable functions $f:\overline D\to\R$ that are bounded together with the derivatives of order~$\le k$. In the case $k=0$, we write $C_b(D)$. 

\smallskip
\noindent
$L^p(J,X)$ is the space of measurable functions $f$ from~$J$ to a Banach space~$X$ such that 
$$
\|f\|_{L^p(J,X)}=\biggl(\int_J\|f(t)\|_X^p\dd t\biggr)^{1/p}<\infty.
$$
The middle term should be replaced by $\esssup_{t\in J}\|f(t)\|_X$ in the case $p=\infty$. 

\section{Main results}
\label{s1}
In this section, we present the model studied in the paper and formulate our results. We begin with the property of exponential mixing in the Kantorovich--Wasserstein distance for a class of discrete-time Markov processes. This type of results are by now well known in the literature for the case of the 2D Navier--Stokes system, and Theorem~\ref{t1.1} presented below is essentially a reformulation of earlier achievements in the abstract framework applicable to other PDE's with random perturbations. We next turn to our main result on large deviations from the stationary distribution for the occupation measures. As a consequence, we obtain the LDP for a class of observables with moderate growth at infinity. 

\subsection{Description of the model}
\label{s1.1}
Let~$H$ be a separable Hilbert space with a norm~$\|\cdot\|$ and let $S:H\to H$ be a continuous mapping. We consider the random dynamical system
\begin{equation} \label{1.1}
u_k=S(u_{k-1})+\eta_k, \quad k\ge1,
\end{equation}
where~$\{\eta_k,k\ge1\}$ is a sequence of i.i.d.\ random variables in~$H$. Equation~\eqref{1.1} is supplemented with the initial condition
\begin{equation} \label{1.2}
u_0=u\in H. 
\end{equation}
We denote by~$(u_k,\IP_u)$ the family of Markov processes associated with~\eqref{1.1}, \eqref{1.2}, by~$P_k(u,\Gamma)$ its transition function at time~$k$, and by $\PPPP_k:C_b(H)\to C_b(H)$ and $\PPPP_k^*:\PP(H)\to\PP(H)$ the corresponding Markov semigroups. In what follows, we assume that~$S$ satisfies the two conditions below.

\medskip
{\bf (A) Dissipativity.} 
{\sl There are positive numbers $\alpha$, $\beta$, $C$, and $q<1$ and a continuous function $\varPhi:H\to\R_+$ such that 
\begin{gather} 
1+\|u\|^\alpha\le \varPhi(u)\le C(1+\|u\|)^\beta\quad\mbox{for $u\in H$}, \label{1.03}\\
\varPhi(S(u)+v)\le q\,\varPhi(u)+C\,\varPhi(v)\quad\mbox{for $u,v\in H$}. \label{1.04}
\end{gather}}

The following hypothesis implies, in particular, that~$S$ is compact and gives some quantitative information about the possibility of approximation of the elements in the image of~$S$ by finite-dimensional subspaces. 

\medskip
{\bf (B) Compactness.} 
{\sl There is a continuous function $\pppp:H\to\R_+$ bounded on any ball, an increasing  sequence~$\{\gamma_N\}\subset(0,+\infty)$ going to~$+\infty$, and an increasing sequence of finite-dimensio\-nal subspaces~$H_N$ whose union is dense in~$H$ such that $H_0=\{0\}$ and, for any $u,v\in H$ and $N\ge0$, we have
\begin{equation} \label{1.4}
\bigl\|(I-{\mathsf P}_N)\bigl(S(u)-S(v)\bigr)\bigr\|
\le\gamma_N^{-1}\exp\bigl\{\pppp(u)+\pppp(v)\bigr\}\|u-v\|,
\end{equation}
where $I$ is the identity operator in~$H$ and ${\mathsf P}_N:H\to H$ denotes the orthogonal projection onto~$H_N$.}

\smallskip
Let us note that if this condition is satisfied for one sequence~$\{H_N\}$, then it holds for any other increasing sequence~$\{H_N'\}$ of finite-dimensional subspaces, and the function~$\pppp$ entering~\eqref{1.4} can be taken to be the same. Indeed, denoting by~${\mathsf P}_N'$ the orthogonal projection to~$H_N'$, we note that
$$
\bigl\|(I-{\mathsf P}_N'){\mathsf P}_Mw\bigr\|\le\e_{MN}\|w\|
\quad\mbox{for any $w\in H$},
$$
where $\e_{MN}\to0$ as $N\to\infty$ for any fixed~$M$. Setting $w=S(u)-S(v)$ and using inequality~\eqref{1.4} with $N=0$ and $N=M$, we derive
\begin{align*}
\bigl\|(I-{\mathsf P}_N')\bigl(S(u)-S(v)\bigr)\bigr\|
&\le \bigl\|(I-{\mathsf P}_N'){\mathsf P}_Mw\bigr\|
+ \bigl\|(I-{\mathsf P}_N')(I-{\mathsf P}_M)w\bigr\|\notag\\
&\le \e_{MN}\|w\|+\bigl\|(I-{\mathsf P}_M)w\bigr\|\notag\\
&\le \bigl(\gamma_0^{-1}\e_{MN}+\gamma_M^{-1}\bigr)
\exp\bigl\{\pppp(u)+\pppp(v)\bigr\}\|u-v\|. 
\end{align*}
The factor in the brackets on the right-hand side of this inequality can be made arbitrarily small by an appropriate choice of~$M$ and~$N$. We thus obtain inequality~\eqref{1.4} with~${\mathsf P}_N$ replaced by ${\mathsf P}_N'$ and a different sequence~$\{\gamma_N\}$. 

\smallskip
We now formulate the hypothesis imposed on the random variables~$\{\eta_k\}$.

\medskip
{\bf (C) Structure of the noise.}
{\sl The random variables~$\eta_k$ can be written as
\begin{equation} \label{1.5}
\eta_k=\sum_{j=1}^\infty b_j \xi_{jk}e_j,
\end{equation}
where $\{e_j\}$ is an orthonormal basis in~$H$, $\{\xi_{jk}\}$ are independent random variables whose laws possess $C^1$-smooth positive densities~$\rho_j$ against the Lebesgue measure such that $\var(\rho_j)\le1$ for all~$j\ge1$, and~$\{b_j\}\subset\R_+$ are numbers satisfying the condition
\begin{equation} \label{1.6}
\BBBB:=\sum_{j=1}^\infty \gamma_{j-1}|b_j|<\infty.
\end{equation}
Moreover, there is $\delta>0$ such that 
\begin{equation} \label{1.8}
\mmmm_\delta(\LL):=\int_He^{\delta (\varPhi(u)+\pppp(u))}\,\LL(\dd u)<\infty,
\end{equation}
where $\LL=\DD(\eta_1)$.
}

\smallskip
Let us note that if Condition~(C) is satisfied, then the random variables~$\eta_k$ are concentrated on a space compactly embedded into~$H$. Indeed, in view of the remark following Condition~(B), without loss of generality we can assume that the subspace~$H_N$ entering Condition~(B) coincides with the vector span of $e_1,\dots,e_N$, where~$\{e_j\}$ is the orthonormal basis in~\eqref{1.5}. Define~$U\subset H$ as the space of vectors $u\in H$ such that
$$
\|{\mathsf Q}_Nu\|\le C\gamma_N^{-1}\quad\mbox{for any $N\ge0$},
$$
where ${\mathsf Q}_N=I-{\mathsf P}_N$, and $C>0$ is a number not depending on~$N$. We endow~$U$ with the norm 
\begin{equation} \label{1.007}
\|u\|_U=\sup_{N\ge0}\bigl(\gamma_N\|{\mathsf Q}_N u\|\bigr).
\end{equation}
It is straightforward to check that~$U$ is a Banach space compactly embedded into~$H$ and that $\|e_j\|_U=\gamma_{j-1}$ for any $j\ge1$. It follows from~\eqref{1.8}, \eqref{1.5}, and~\eqref{1.6} that 
\begin{equation} \label{1.08}
\E\|\eta_k\|_U\le C\BBBB.
\end{equation}

We also note that if Condition~(B) is satisfied for some sequence~$\{\gamma_N\}$, then it holds also for any other sequence of positive numbers~$\{\gamma_N'\}$ such that $\gamma_N'\le C\gamma_N$ for some~$C>0$. Thus, condition~\eqref{1.6} imposed on the sequence~$\{b_j\}$ can be relaxed to the following one: there is a sequence of positive numbers~$\{\gamma_j\}$ going to~$+\infty$ such that series~\eqref{1.6} converges. 

\subsection{Exponential mixing}
\label{s1.2}
To formulate our result on exponential mixing, we shall need the concept of a stabilisable functional. Let $\pppp:H\to\R_+$ be a  continuous function. We shall say that~$\pppp$ is {\it stabilisable\/} for the Markov family~$(u_k,\IP_u)$ if there is an increasing continuous function $Q:\R_+\to\R_+$ and positive numbers~$c$ and~$\delta$ such that
\begin{equation} \label{1.7}
\E_u\exp\bigl\{\delta\bigl(\pppp(u_1)+\cdots+\pppp(u_{k})\bigr)\bigr\}\le Q(\|u\|)e^{c k}\quad\mbox{for $k\ge1$, $u\in H$}. 
\end{equation}
The following theorem is a generalisation of some earlier results on mixing properties of the Navier--Stokes system proved in~\cite{KS-mpag2001,shirikyan-jmfm2004}; see also~\cite{FM-1995,KS-cmp2000,EMS-2001,BKL-2002,HM-2006} and the references in~\cite{KS-book,debussche-2013} for other related results. 

\begin{theorem} \label{t1.1}
Let Conditions~{\rm(A)}, {\rm(B)}, and~{\rm(C)} stated in Section~\ref{s1.1} be fulfilled. Assume, in addition, that $S(0)\in U$, and the functional~$\pppp$  entering~{\rm(B)} is stabilisable, with a function~$Q$ satisfying the inequality 
\begin{equation} \label{1.09}
Q(\|u\|)\le e^{\rho\, \varPhi(u)}\quad\mbox{for all $u\in H$}, 
\end{equation}
where $\rho>0$ does not depend on~$u$. Then there is an integer $N\ge1$ such that the Markov family~$(u_k,\IP_u)$ associated with~\eqref{1.1} has a unique stationary measure $\mu\in\PP(H)$, provided that 
\begin{equation} \label{1.9}
b_j\ne0\quad\mbox{for $j=1,\dots,N$}. 
\end{equation}
Moreover,  there are positive constants~$\gamma$ and~$C$ such that
\begin{equation} \label{1.11}
\|\PPPP_k^*\lambda-\mu\|_L^*
\le C\,e^{-\gamma k}\int_H \varPhi(u)\,\lambda(\dd u)
\quad\mbox{for any $\lambda\in\PP(H)$, $k\ge0$}. 
\end{equation}
\end{theorem}

The proof of this result is an extension of  the scheme used in~\cite{shirikyan-jmfm2004} and Chapter~3 of~\cite{KS-book} for the case of the Navier--Stokes system perturbed by an unbounded kick force to the abstract setting of this paper.  This extension is straightforward  and we  omit  the details.

\subsection{Large deviations}
\label{s1.3}
To formulate our result on LDP, we first recall some definitions from the theory of large deviations. Let~$X$ be a Polish space and let $\{\zeta_k\}$ be a sequence of random probability measures on~$X$ defined on a measurable space $(\Omega,\FF)$. In other words, for any $k\ge1$, we are given  a measurable mapping $\zeta_k:\Omega\to\PP(X)$. Recall that a mapping $I : \PP(X) \to [0, +\ty]$ is called a {\it rate function\/} if it is lower semicontinuous, and a rate function~$I$ is said to be {\it good\/} if the set~$\{\sigma\in\PP(X) : I(\sigma) \le \alpha \}$ is compact for any~$\alpha \in[0, +\infty)$. For a  set~$\Gamma\subset\PP(X)$ we write~$I(\Gamma)=\inf_{\sigma\in\Gamma}I(\sigma)$. 

\begin{definition}
Suppose that a family of probability measures $\{\IP_\lambda, \lambda\in\Lambda\}$ is given on the measurable space $(\Omega,\FF)$. We say that the sequence $\{\zeta_k\}$ {\it satisfies the LDP with a good rate function~$I$, uniformly with respect to~$\lambda\in\Lambda$\/}, if the following two properties hold. 
\begin{description}
\item[Upper bound.] 
For any closed subset~$F\subset\PP(X)$, we have
$$
\limsup_{k\to\infty} \frac1k\log \sup_{\lambda\in\Lambda} \pP_\lambda\{\zeta_k\in F\}\le -I(F).
$$
\item[Lower bound.] 
For any open subset~$G\subset\PP(X)$, we have
$$
\liminf_{k\to\infty} \frac1k\log \inf_{\lambda\in\Lambda}\pP_\lambda\{\zeta_k\in G\}\ge -I(G).
$$
\end{description}
\end{definition}   

We now turn to the Markov process~$(u_k,\IP_u)$ associated with~\eqref{1.1}. Given a probability measure $\lambda\in\PP(H)$, we denote by~$\IP_\lambda$ the induced  measure
$$
\IP_\lambda(\Gamma)=\int_X\IP_u(\Gamma)\lambda(\dd u),\quad \Gamma\in\FF,
$$
 on~$(\Omega,\FF)$. For positive numbers $\delta$ and $M$, let~$\Lambda(\delta,M)$ be the set of measures $\lambda\in\PP(H)$ that satisfy the inequality
$$
\int_He^{\delta\varPhi(v)}\lambda(\dd v)\le M. 
$$
For any integer~$\ell\ge1$, we set $\uu^\ell_n=[u_n,\ldots,u_{n+\ell-1}]$ and consider the family of {\it occupation measures\/} 
 \begin{equation} \label{1.18}
\zeta_k^\ell=\frac1k\sum_{n=0}^{k-1}\delta_{\uu_n^\ell},\quad k\ge1,
\end{equation}
defined on the family of probability spaces $(\Omega,\FF,\IP_\lambda)$, where $\lambda\in\PP(H)$. Thus, for any~$\lambda$, relation~\eqref{1.18} defines a sequence of random probability measures on~$H^\ell$. Recall that the space~$U$ was introduced at the end of Section~\ref{s1.1}. The following theorem is the main result of this paper. 

\begin{theorem} \label{t1.2}
In addition to the hypotheses of Theorem~\ref{t1.1}, let us assume the numbers~$b_j$ are all non-zero. Then, for any $\delta>0$ and $M>0$, the sequence~$\{\zeta_k^\ell,k\ge1\}$ defined on $(\Omega,\FF,\IP_\lambda)$ satisfies the LDP, uniformly with respect to $\lambda\in\Lambda(\delta,M)$, with a good rate function~$I_\ell: \PP(H^\ell)\to [0,+\ty]$ not depending on~$\lambda$. Moreover, $I_\ell$ is given by
$$
I_\ell(\sigma)=\sup_{V\in C_b(H^\ell)}  \bigl(\lag V,\sigma\rag-Q_\ell(V) \bigr),
\quad \sigma\in\PP(H^\ell),
$$
where~$Q_\ell:C_b(H^\ell)\to\R$ is a~$1$-Lipschitz convex function such that~$Q_\ell(C)=C$ for any~$C\in\R$. 
\end{theorem}

The scheme of the proof of this theorem is outlined in Section~\ref{s1.5}, the details are presented in Sections~\ref{s4}--\ref{s8}, and applications to PDE's with random perturbation are discussed in Section~\ref{s2}. Here we derive two corollaries from Theorem~\ref{t1.2}. The first of them concerns the LDP for the time averages of unbounded functionals. 

\begin{corollary} \label{c1.3}
Under the hypotheses of Theorem~\ref{t1.2}, let $f:H^\ell\to \R$ be a continuous functional bounded on any ball such that 
\begin{equation}\label{1.19}
\frac{|f(v^1,\dots,v^\ell)|}{\pppp(v^1)+\cdots+\pppp(v^\ell)} \to 0 \quad \text{as $\|v^1\|+\cdots+\|v^\ell\|\to\ty$.}
\end{equation}
Then, for any $\delta>0$ and $M > 0$, the $\IP_\lambda$-laws of the random variables
$$
\zeta_k^{\ell,f}=\frac1k\sum_{n=0}^{k-1} {f(\uu_n^\ell)},\quad k\ge1,
$$ 
satisfy the LDP, uniformly with respect to $\lambda\in\Lambda(\delta,M)$, with a good rate function~$I_\ell^{ f}: \R\to[0,+\infty]$ not depending on~$\lambda$. Moreover,  $I_\ell^{f}$ is given by
$$
I_\ell^{f}(\sigma)= \inf\{I_{\ell}(\nu): \nu\in \PP(H^\ell),  \lag f,\nu\rag=\sigma\},
\quad \sigma\in \R,
$$
where the infimum over an empty set is equal to $+\ty$.
\end{corollary}

We shall not give a proof of this corollary, because similar contraction-principle-type results were established earlier in~\cite{gourcy-2007b,JNPS-2013}. The only difference is that here  we claim that the LDP holds uniformly with respect to the initial measure~$\lambda$. However, the proofs in the above-mentioned works give also the uniformity, provided that the LDP for the occupation measures holds uniformly. 

To formulate the second corollary, we denote by~$\HHH=H^{\Z_+}$ the direct product of countably many copies of~$H$ and, given a solution~$\{u_k\}$ for~\eqref{1.1}, define the occupation measure for full trajectories by the relation
$$
\zeta_k^\infty=\frac1k\sum_{n=0}^{k-1}\delta_{\uuu_n^\infty},\quad k\ge1,
$$
where we set~$\uuu_n^\infty=[u_k,k\ge n]$. The following result is an immediate consequence of Theorem~\ref{t1.2} and the Dawson--G\"artner theorem (see Section~4.6 in~\cite{ADOZ00}). Its proof can be carried out in exactly the same way as for the case of bounded kicks (see Section~1.5 in~\cite{JNPS-2012}).

\begin{corollary} \label{c1.4}
Let the conditions of Theorem~\ref{t1.2} be fulfilled.  Then, for any $\delta>0$ and~$M>0$, the family~$\{\zeta_k^\infty,k\ge1\}$ satisfies the LDP, uniformly in $\lambda\in\Lambda(\delta,M)$, with a good rate function~$\III:\PP(\HHH)\to[0,+\infty]$ not depending on~$\lambda$.
\end{corollary}

\subsection{Scheme of the proof of Theorem~\ref{t1.2}}
\label{s1.5}
We begin with some definitions. Given a random dynamical system (RDS for short) of the form~\eqref{1.1} and an integer $\ell\ge1$, we define the {\it Markov $\ell$-process\/} associated with~\eqref{1.1} in the following way: the phase space is the direct product of $\ell$ copies of~$H$, and the time-$1$ transition function is given by
\begin{equation} \label{1.12}
P_1^\ell(\vvv,\Gamma_1\times\cdots\times\Gamma_\ell)=
\delta_{v^2}(\Gamma_1)\cdots\delta_{v^{\ell}}(\Gamma_{\ell-1})
P_1(v^{\ell}, \Gamma_\ell), 
\end{equation}
where $\vvv=[v^1,\dots,v^{\ell}]\in H^\ell$ and $\Gamma_j\in\BB(H)$ for $j=1,\dots,\ell$. In other words, the Markov $\ell$-process associated with~\eqref{1.1} is the Markov process corresponding to the RDS in $H^\ell$ defined by 
\begin{equation} \label{1.015}
\uuu_k=\Ss(\uuu_{k-1})+\eeta_k,
\end{equation}
where $\uuu_k=[u_k^1, \dots,u_k^{\ell}]$, $\eeta_k=[0,\dots,0,\eta_{k}]$, and $\Ss:H^\ell\to H^\ell$ is the mapping given by
$$
\Ss(\vvv)=\bigl[v^2,\dots,v^\ell,S(v^\ell)\bigr], 
\quad \vvv=[v^1,\dots,v^\ell]\in H^\ell.
$$
Let us note that if $v\in H$ and $\vvv\in H^\ell$ are such that $v^\ell=v$, then the trajectories~$\{u_k\}$ and~$\{\uuu_k\}=\{[u_k^1,\dots,u_k^\ell]\}$ of the RDS~\eqref{1.1} and~\eqref{1.015} issued from~$v$ and~$\vvv$, respectively, satisfy the relation
\begin{equation} \label{1.017}
u_k^j=u_{k-\ell+j}\quad\mbox{for $k\ge\ell-j$, $1\le j\le\ell$.}
\end{equation}
Denote by~$\zzeta_k$ the occupation measures for~\eqref{1.015}:
 \begin{equation} \label{1.016}
\zzeta_k=\frac1k\sum_{n=1}^{k}\delta_{\uu_n},\quad k\ge1.
\end{equation}
Given $\lambda\in\PP(H)$, we denote by~$\lambda^{(\ell)}\in\PP(H^\ell)$ the tensor product of $\ell-1$ copies of~$\delta_0$ with~$\lambda$; that is, $\lambda^{(\ell)}=\delta_0\otimes\cdots\otimes\delta_0\otimes\lambda$. 
It is straightforward to check that, for any $\delta>0$ and $M>0$, the random measures~$\zzeta_k$ and~$\zeta_k^\ell$ corresponding to the initial laws~$\lambda^{(\ell)}$ and~$\lambda$, respectively, are exponentially equivalent uniformly with respect to $\lambda\in\Lambda(\delta,M)$ (cf.\ Lemma~6.2 in~\cite{JNPS-2012}). We shall thus study the LDP for the occupation measures~\eqref{1.016}.

Given positive numbers~$\delta$ and~$M$, we denote $\LLambda(\delta,M)$ the set of measures $\lambda\in\PP(H^\ell)$ satisfying the condition
$$
\int_{H^\ell}\exp\bigl\{\delta\bigl(\varPhi(v_1)+\cdots+\varPhi(v_\ell)\bigr)\bigr\}\lambda(\dd v_1,\dots,\dd v_\ell)\le M.
$$
Suppose we can prove the following two properties:
\begin{description}
\item[\bf  Property~1: Existence of a limit.]  
For any~$V\in C_b(H^\ell)$ and~$\lambda\in \LLambda(\delta,M)$, the limit (called {\it pressure function\/})
\begin{equation} \label{1.21}
Q_\ell(V)=\lim_{k\to+\ty} \frac{1}{k}
\log\E_\lambda\exp\biggl(\,\sum_{n=1}^kV(\uuu_n)\biggr)
\end{equation}
exists and does not depend on the initial measure. Moreover, this limit is uniform with respect to $\lambda\in \LLambda(\delta,M)$ for any $\delta>0$ and $M>0$.
\end{description}
  \begin{description}
\item[\bf Property~2: Uniqueness of the equilibrium state.]
There is  a   vector space  $\VV\subset C_b( H^\ell)$ such that, for any compact set $K\subset H^\ell$, the family of restrictions to~$K$ of the functions in~$\VV$   is dense in~$C(K)$, and for any~$V\in \VV $ there is at most one~$\sigma_V\in \PP(H^\ell)$ satisfying the relation
\begin{equation}\label{6.2}
Q_\ell(V)=  \lag V, \sigma_V\rag-I_\ell(\sigma_V),
\end{equation} 
where $I_\ell(\sigma)$ denotes the Legendre transform of~$Q_\ell$.
\end{description}
In this case, a generalisation of a result established by Kifer in the case of a compact space shows that the LDP holds for~$\{\zzeta_k\}$, uniformly in $\lambda\in\LLambda(\delta,M)$, provided that the RDS~\eqref{1.015} possesses a property of exponential tightness; see Theorem~\ref{T:2.3} for the exact formulation. We thus need to prove the above two properties. 

To this end, given a function $V\in C_b(H^\ell)$, we introduce a generalised Markov semigroup by the formula
\begin{equation} \label{6.3}
\PPPP_k^V f(\uuu):=\E_\uuu f(\uuu_k)\exp \biggl(\,\sum_{n=1}^kV(\uuu_n)\biggr), 
\quad f\in C_b(H^\ell),
\end{equation}
and denote by $\PPPP_k^{V*} :\MM_+(H^\ell)\to \MM_+(H^\ell)$ its dual semigroup. Under some hypotheses on the kernel, we describe the asymptotic behaviour of generalised Markov semigroups in Theorem~\ref{T:3.1}. We then construct explicitly a vector space~$\VV\subset C_b(H^\ell)$ such that the hypotheses of Theorem~\ref{T:3.1} are satisfied for~$\{\PPPP_k^V\}$ with any $V\in\VV$. It follows that there is a function $h_V\in C_+( H^\ell)$, a
measure $\mu_V\in\PP(H^\ell)$ and a number $\lambda_V>0$ such that
$\PPPP_1^V h_V=\la_V h_V$, and for any $R>0$, $f\in  C_b( H^\ell)$, and $\nu\in \PP( H^\ell)$ we have
\begin{align}
\lambda_V^{-k}\PPPP_k^V f&\to\lag f,\mu_V\rag h_V
\quad\mbox{in~$C(B_{H^\ell}(R)) $ as~$k\to\infty$}, \label{1.25}\\
\lambda^{-k}_V\PPPP_k^{V*}\nu&\to\lag h_V,\nu\rag\mu_V
 \quad\mbox{in~$\MM_+(H^\ell)$ as~$k\to\infty$}. \label{1.26}
\end{align}
Taking $f=1$ in~\eqref{1.25},  we obtain Property~1 for~$V\in\VV$ and $\lambda=\delta_\uuu$, and an approximation argument enables one to prove it for any~$V\in C_b(H^\ell)$, uniformly with respect to $\lambda\in\LLambda(\delta,M)$.
 
To establish Property~2, we first show that any equilibrium state~$\sigma_V$ is a stationary measure for the following Markov semigroup:
\begin{equation} \label{1.27}
\SSS^V_k g:=\la_V^{-k} h_V^{-1} \PPPP_k^V (gh_V), \quad g\in C_b(H^\ell).
\end{equation}
We then deduce the uniqueness of stationary measure for~$\SSS^V_k$ from convergence~\eqref{1.26}, showing that $\sigma_V(\dd \vvv)= h_V(\vvv) \mu_V (\dd \vvv)$. 

The key point in the above analysis is the proof of the uniform Feller property (see Theorem~\ref{T:3.1}). It is based on a coupling argument and is carried out in Section~\ref{s7}. The other details of the proof of Theorem~\ref{t1.2} are presented Section~\ref{s8}.

\section{Applications}
\label{s2}
In this section, we apply our main results to some concrete examples of PDE's with random perturbations. We shall confine ourselves to the 2D Navier--Stokes system and the complex Ginzburg--Landau equation without nonlinear dissipation; however, the results apply equally well to other models, such as nonlinear reaction-diffusion system and the Ginzburg--Landau equation without linear dispersion; see~\cite{JNPS-2013,KN-2013}. 

\subsection{Navier--Stokes system}
\label{s2.1}
Let us consider the 2D Navier--Stokes system in a bounded domain, subject to a random kick force. After projecting it to the space~$H$ of divergence-free square integrable vector fields with zero normal component (see~\eqref{0.3}), we obtain the evolution equation
\begin{equation} \label{02.1}
\dot u+\nu Lu+B(u)=h+\eta(t).
\end{equation}
Here $\nu>0$ is the viscosity, $L=-\Pi\Delta$ is the Stokes operator, $B(u)=\Pi(\langle u,\nabla\rangle u)$ stands for the bilinear term, $\Pi$ is the orthogonal projection in~$L^2$ onto the closed subspace~$H$, $h\in H$ is a deterministic function, and~$\eta$ is a random process of the form~\eqref{0.2}, in which~$\{\eta_k\}$ is a sequence of i.i.d.\ random variables in~$H$. The well-posedness of~\eqref{02.1} is well known; e.g., see Section~6 in~\cite[Chapter~1]{lions1969} or Section~2.3 in~\cite{KS-book}. Normalising the trajectories to be right-continuous and setting $u_k=u(k,x)$, for any initial state $u\in H$ we obtain a sequence $\{u_k\}$ satisfying~\eqref{1.1}. 

\smallskip
We wish to prove that Theorems~\ref{t1.1} and~\ref{t1.2} are applicable\footnote{The fact that Theorem~\ref{t1.1} holds  for the Navier--Stokes system is well known; see Section~3.4 in~\cite{KS-book}.} to the Markov process associated with~\eqref{02.1}. Let us denote by~$S_t:H\to H$ the nonlinear semigroup generated by Eq.~\eqref{02.1}  with $\eta\equiv0$ and let $S=S_1$. 

\begin{proposition} \label{t2.1}
Conditions~{\rm(A)} and~{\rm(B)} stated in Section~\ref{s1.1} are satisfied for~$S$ with the following choice of parameters:
\begin{align*}
\varPhi(u)&=1+\|u\|^2,  & \pppp(u)&=C\int_0^1\bigl(\|S_t(u)\|_1^2+1\bigr)\dd t,\\ 
\gamma_N&=\alpha_{N+1}^{1/2}, & H_N&=\lspan\{e_1,\dots,e_N\},
\end{align*}
where $C$ is a positive number, $\|\cdot\|_1$ stands for the $H^1$-norm, and~$\{e_j\}$ is a complete set of normalised eigenfunction for the Stokes operator, with a non-decreasing order of the corresponding eigenvalues~$\{\alpha_j\}$. Moreover, $S(0)$ belongs to the space~$U$ defined after Condition~{\rm(C)}, and if $\E\exp(\varkappa\|\eta_1\|^2)<\infty$ for some $\varkappa>0$, then the functional~$\pppp$ is stabilisable with a function~$Q$ satisfying~\eqref{1.09}. 
\end{proposition}

Once this result is established, we can claim that  the conclusions of Theorem~\ref{t1.1} and~\ref{t1.2} are valid for the 2D Navier--Stokes system~\eqref{02.1}, provided that the random variables~$\eta_k$ satisfy Condition~(C) and all coefficients~$b_j$ are nonzero.\footnote{Note that the orthonormal basis~$\{e_j\}$ entering~\eqref{1.5} does not need to coincide with the system of eigenvectors for the Stokes operator, because the validity of Condition~(B) for one orthonormal basis implies its validity for any other.} In particular, it is easy to check that those results are true if the law of~$\eta_k$ is a non-degenerate Gaussian measure concentrated on the Sobolev space of order~$2$.  Furthermore, in view of Corollary~\ref{c1.3} and the inequality $\|u\|^2\le C(\pppp(u)+1)$ which is true for any $u\in H$ (see Exercise~2.1.23 in~\cite{KS-book}), the uniform LDP is valid for the energy functional $f(u)=\|u\|$. Considering the Navier--Stokes system on higher Sobolev spaces, we can establish the LDP for other physically relevant functionals, such as the enstrophy and correlation tensors; cf.\ Section~1.3 in~\cite{JNPS-2012}. However, we do not give any detail here, because a similar situation is considered in the next subsection for the technically more complicated case of the complex Ginzburg--Landau equation. 

\begin{proof}[Proof of Proposition~\ref{t2.1}]
The fact that $S:H\to H$ is continuous is well known; e.g., see Sections~1.6 in~\cite{BV1992}. Inequality~\eqref{1.03} with $\alpha=\beta=2$ is trivial for the above choice of~$\varPhi$, and~\eqref{1.04} is the dissipativity property of the Navier--Stokes dynamics; see Theorem~6.1 in~\cite[Chapter~1]{BV1992}. 

Let us check Condition~(B). The continuity of the resolving operator for the Navier--Stokes system from~$H$ to $L^2(0,1;H^1)$ implies that~$\pppp$ is a continuous function on~$H$. We introduce the space $V=H\cap H_0^1(D,\R^2)$ and endow it with the usual $H^1$-norm~$\|\cdot\|_1$. It is well known that (e.g., see Proposition~2.1.25 in~\cite{KS-book})
\begin{equation} \label{2.02}
\|S(u)-S(v)\|_1\le \exp\bigl(\pppp(u)+\pppp(v)\bigr)\,\|u-v\|\quad\mbox{for $u,v\in H$}. 
\end{equation}
Applying the Poincar\'e inequality, we obtain~\eqref{1.4}. 

We now prove that the space~$V$ is embedded into~$U$. Recalling~\eqref{1.007}, we write
$$
\|u\|_U^2=\sup_{N\ge0}\bigl(\alpha_{N+1}\|{\mathsf Q}_Nu\|^2\bigr)\le \sup_{N\ge0}\sum_{j=N+1}^\infty \alpha_j\langle u,e_j\rangle^2
= \|u\|_1^2. 
$$
Since~$S$ maps continuously~$H$ into~$V$, we conclude that $S(0)\in U$. Finally, the stabilisibility of~$\pppp$ is established in Step~2 of the proof of Proposition~2.3.8 in~\cite{KS-book}.
\end{proof}

\subsection{Ginzburg--Landau equations}
\label{s2.2}
Let $D\subset\R^d$ ($d\le4$) be a bounded domain with $C^2$-smooth boundary~$\p D$. We consider the complex Ginzburg--Landau equation with a cubic nonlinearity:
\begin{equation} \label{2.03}
\dot u-(\nu+i)\Delta u+ia|u|^2u=h(x)+\eta(t,x), \quad x\in D. 
\end{equation}
Here $u(t,x)$ is an unknown complex-valued function, $\nu$ and~$a$ are positive parameters, $h\in H^1(D,\C)$ is a deterministic function, and~$\eta$ is a random force of the form~\eqref{0.2}, in which~$\{\eta_k\}$ is a sequence of i.i.d.\ random variables in~$H_0^1(D,\C)$. 
Equation~\eqref{2.03} is supplemented with the Dirichlet boundary condition,
\begin{equation} \label{2.04}
u\bigr|_{\p D}=0,
\end{equation}
and an initial condition at time $t=0$,
\begin{equation} \label{2.05}
u(0,x)=u_0(x).
\end{equation}
It is well known that, for any $u_0\in H_0^1(D,\C)$, the problem~\eqref{2.03}--\eqref{2.05} possesses a unique solution~$u(t,x)$ that belongs to the space $C(\R_+,H_0^1)\cap L_{\rm loc}^2(\R_+,H^2)$; see~\cite{cazenave2003} for the more complicated case of Schr\"odinger equation. Normalising solutions to be right-continuous and restricting them to the integer lattice, we obtain a sequence~$\{u_k\}$ satisfying~\eqref{1.1}, where $S$ denotes the time-$1$ shift along trajectories of~\eqref{2.03} with $\eta\equiv0$. Our aim is to prove that Theorems~\ref{t1.1} and~\ref{t1.2} are applicable in this situation. 

Let us denote by~$H$ the complex Sobolev space~$H_0^1(D)$ and regard it as a real Hilbert space endowed with the scalar product
$$
\langle u,v\rangle=\Re\int_D\nabla u\cdot\nabla\bar v\,\dd x
$$
and the corresponding norm~$\|\cdot\|_1$. Let $S_t:H\to H$ be the resolving semigroup for problem~\eqref{2.03}, \eqref{2.04} with $\eta\equiv0$ and let~$S=S_1$. We introduce a continuous functional $\HH:H\to\R$ by the relation
$$
\HH(u)=\int_D\Bigl(\frac12|\nabla u|^2+\frac{a}{4}|u|^4\Bigr)\dd x.
$$
Let us denote by~$e_j$ the eigenfunctions of the Dirichlet Laplacian in~$D$, indexed in the non-decreasing order of the corresponding eigenvalues~$\{\alpha_j\}$ and normalised by the condition~$\|\nabla e_j\|=1$ (i.e., $\|e_j\|=\frac{1}{\sqrt{\alpha_j}}$). Then the vectors $\{e_j,ie_j,j\ge1\}$ form an orthonormal basis in~$H$. 

\begin{proposition} \label{p2.2}
Conditions~{\rm(A)} and~{\rm(B)} are satisfied for~$S$ if we choose
\begin{align*}
\varPhi(u)&=1+\HH^2(u), &  \pppp(u)&=C\int_0^1\bigl(\|\nabla S_t(u)\|^4+1\bigr)\dd t,\\
\gamma_N&=\alpha_{N+1}^{\e}, & H_N&=\lspan\{e_j,ie_j,1\le j\le N\}. 
\end{align*}
where $C>0$ is sufficiently large and $\e>0$ can be taken arbitrarily small. Moreover, $S(0)$ belongs to the space~$U$ defined after Condition~{\rm(C)}, and if $\E\exp(\varkappa\HH^2(\eta_1))<\infty$ for some $\varkappa>0$, then the functional~$\pppp$ is stabilisable with a function~$Q$ satisfying~\eqref{1.09}. 
\end{proposition}

\begin{proof}
The continuity of~$S$ is a standard fact and we omit the  proof, referring the reader to Chapter~1 in~\cite{BV1992} for a proof of similar properties for various PDE's. Inequality~\eqref{1.03} with $\alpha=4$ and $\beta=8$ follows from the definition and the Sobolev embedding $H_0^1\subset L^4$. To prove~\eqref{1.04}, recall that if~$u(t)$ is a solution of~\eqref{2.03}, \eqref{2.04}, then
$$
\frac{\dd}{\dd t}\HH(u)+c\,\HH(u)\le C_1(\|h\|_{L^4}^4+1),
$$
where $c>0$, and we denote by~$C_j$ unessential positive numbers; e.g., see inequality~(1.33) in~\cite{JNPS-2012}. It follows that
\begin{equation} \label{2.06}
\frac{\dd}{\dd t}\HH^2(u)+c\,\HH^2(u)\le C_2(\|h\|_{L^4}^8+1).
\end{equation}
Applying the Gronwall inequality, we derive
$$
\HH^2(S(u_0))\le e^{-c}\,\HH^2(u_0)+C_3(\|h\|_{L^4}^8+1).
$$
Since $\HH(w+v)\le (1+\theta)\HH(w)+C_\theta\HH(v)$ for any $\theta>0$ and a sufficiently large $C_\theta>1$, we have
$$
\HH^2(S(u_0)+v)\le e^{-c}(1+\theta)\HH^2(u_0)+C_4(\|h\|_{L^4}^8+1)+C_\theta\HH(v),\quad u_0,v\in\ H. 
$$
Choosing~$\theta>0$ so small that $q:=e^{-c}(1+\theta)<1$, we obtain~\eqref{1.04}. 

\smallskip
Let us prove~\eqref{1.4}. As is established in~\cite{JNPS-2012} (see inequality~(1.39)), if~$u_1$ and~$u_2$ are two solutions, then the function $w={\mathsf Q}_N(u_1-u_2)$ satisfies the differential inequality 
$$
\p_t\|w\|_1^2\le -\bigl(\nu\alpha_{N+1}-C_5(\|u_1\|_1+\|u_2\|_1)^4\bigr)\|w\|_1^2. 
$$
Application of the Gronwall inequality results in
$$
\bigl\|{\mathsf Q}_N(u_1(1)-u_2(1))\bigr\|^2
\le \exp\biggl(-\nu\alpha_{N+1}+C_5\int_0^1(\|u_1\|_1+\|u_2\|_1)^4\dd t\biggr)\|u_1^0-u_2^0\|^2,
$$
where $u_i^0$ is the initial state of~$u_i$. This implies inequality~\eqref{1.4} in which one can take for~$\gamma_N$ an arbitrary sequence such that $\exp(-\nu\alpha_{N+1})\le C\gamma_N^{-1}$. In particular, if we choose $\gamma_N=\alpha_{N+1}^{\e}$ with $\e\in(0,1]$, then
\begin{align*}
\|u\|_U^2&=\sup_{N\ge0}\bigl(\alpha_{N+1}^{2\e}\|{\mathsf Q}_Nu\|_1^2\bigr)
=\sup_{N\ge0}\Bigl(\alpha_{N+1}^{2\e}\sum_{j=N+1}^\infty \alpha_j\bigl(\langle u,e_j\rangle^2+\langle u,ie_j\rangle^2\bigr)\Bigr) \\
&=\sup_{N\ge0}\sum_{j=N+1}^\infty \alpha_j^{1+2\e}\bigl(\langle u,e_j\rangle^2+\langle u,ie_j\rangle^2\bigr)=\|u\|_{1+\e}^2. 
\end{align*}
Thus, the space~$H_0^1\cap H^{1+2\e}$ is continuously embedded into~$U$, and the regularising property of the Ginzburg--Landau dynamics implies that $S(0)\in U$. 

\smallskip
It remains to prove the stabilisability of~$\pppp$. To this end, let us take any solution~$u(t)$ of problem~\eqref{2.03}--\eqref{2.05} in which~$\eta$ is given by~\eqref{0.2}. Denoting by~$u_l^-$ the left-hand limit of~$u$ at the point $t=l$ and integrating~\eqref{2.06} in $t\in(l-1,l)$, we derive
\begin{equation} \label{2.07}
\HH^2(u_l^-)+c\int_{l-1}^l\HH^2(u)\,\dd t\le \HH^2(u_{l-1})+C_4(\|h\|_{L^4}^8+1). 
\end{equation}
Taking the sum over $l=1,\dots,k+1$ and recalling the definition of~$\pppp$ and~$\varPhi$, we see that 
$$
\sum_{l=0}^k\pppp(u_l)\le C_6\sum_{l=0}^k\varPhi(u_l),
$$
where $C_6$ depends on $\|h\|_{L^4}$, but not on the solution. The required assertion follows now from Lemma~\ref{l3.1} (see~\eqref{3.01}). 
\end{proof}

We have thus checked Conditions~(A) and~(B) for~$S$ and the stabilisability of~$\pppp$ under a suitable hypothesis on the law of~$\eta_k$. Therefore, if~$\{\eta_k\}$ satisfies Condition~(C), with all coefficients~$b_j$ nonzero, then the conclusions of Theorems~\ref{t1.1} and~\ref{t1.2} are valid. Moreover, it follows from~\eqref{2.07} with $l=0$ that $\pppp(u)\le C\varPhi(u)$. Hence, inequality~\eqref{1.8} is equivalent to
\begin{equation} \label{2.08}
\int_H\exp\bigl(\delta\HH^2(u)\bigr)\LL(\dd u)<\infty\quad\mbox{for some $\delta>0$},
\end{equation}
and the hypotheses of Corollary~\ref{c1.3} are satisfied for the energy functional~$\HH(u)$. 

Finally, let us mention that condition~\eqref{2.08} is rather restrictive and is not satisfied for Gaussian measures on~$H$. On the other hand, if we replace the cubic term in~\eqref{2.03} by the weaker nonlinearity $a|u|^{2\sigma}u$, where $\sigma\in[0, 2/d]$ for $d=3,4$, then the nondegenerate Gaussian measures concentrated on smooth functions  satisfy all required conditions. The proof of this fact is an immediate 
consequence of  the arguments used in this section. 

\section{Kifer's criterion}
\label{s4}
Let~$X$ be a Polish space  and let $\PP(X)$ be the set of probability Borel measures on~$X$ endowed with the weak topology and the corresponding Borel $\sigma$-algebra. We consider a directed set~$\Theta$ and a family~$\{\zeta_\theta,\theta\in\Theta\}$ of random probability measures on~$X$ that are defined on some probability spaces~$(\Omega_\theta, \FF_\theta, \pP_\theta)$. In other words, for any $\theta\in\Theta$, we have a measurable mapping $\zeta_\theta:\Omega_\theta\to\PP(X)$. In what follows, we often omit the parameter~$\theta$ from the notation of the probability space and write simply~$(\Omega,\FF,\IP)$, and the corresponding expectation is denoted by~$\E$; this will not lead to a confusion. In the case when~$X$ is compact, Kifer~\cite{kifer-1990} obtained a sufficient condition ensuring the validity of the LDP for the family~$\{\zeta_\theta\}$. In this section, we extend Kifer's result to the case of a general Polish space~$X$ under some additional hypotheses on~$\{\zeta_\theta\}$. The possibility of such an extension was hinted in Remark 2.2 of~\cite{kifer-1990}. Since this extension of Kiffer's criterion plays a 
central role in our work, in this section we give its detailed proof.

\smallskip
As in~\cite{kifer-1990}, we assume that the following limit exists for any~$V\in C_b(X)$:
 \begin{equation}\label{2.1}
Q(V)=\lim_{\theta\in\Theta} \frac{1}{r(\theta)}\log\int_{\Omega} 
\exp\bigl(r(\theta)\lag V, \zeta_\theta^\om\rag\bigr)\dd \pP(\om),
\end{equation} 
where~$r:\Theta\to \R$ is a given positive function such that $\lim_{\theta\in\Theta} r(\theta)=+\ty$. 
Then~$Q:C_b(X)\to \R$ is a convex~$1$-Lipschitz function such that~$Q(V)\ge0$ for any~$V\in C_+(X)$ and~$Q(C)=C$ for any constant~$C\in\R$. Recall that the {\it Legendre transform\/} of~$Q$ is defined on the space~$\MM(X)$ by 
\begin{equation} \label{2.2}
I(\sigma)=\sup_{V\in C_b(X)}\bigl(\lag V, \sigma\rag-Q(V)\bigr)
\quad\mbox{for $\sigma \in {\cal P}(X)$}, 
\end{equation}
and~$I(\sigma)=+\infty$ otherwise. It is well known that~$I(\sigma)$ is a convex function lower semicontinuous in the weak topology, and the function~$Q$ can be reconstructed by the formula 
$$
Q(V)= \sup_{\sigma\in\PP(X)} \bigl(\lag V, \sigma\rag-I(\sigma)\bigr).
$$
Any measure~$\sigma_V\in\PP(X)$ satisfying the relation
$$
Q(V)= \lag V,\sigma_V\rag-I(\sigma_V)
$$
is called an {\it equilibrium state\/} for~$V$.  

\begin{definition} \label{d4.1}
The family~$\{\zeta_\theta\}$ is said to be {\it exponentially tight}    with speed~$r(\theta)$ if for any~$a>0$ there is a compact set~$\KK_a\subset\PP(X)$ such that   
\begin{equation}\label{2.3}
\limsup_{\theta\in \Theta}\frac{1}{r(\theta)}\log \pP\{\zeta_\theta \in \KK_a^c\}\le -a.
\end{equation}
\end{definition}

A sufficient condition for the exponential tightness for a family of random probability measures is given by the following lemma. 

\begin{lemma}\label{L:2.2}  
Let $\varPhi : X  \to [0,+\ty]$ be a function with compact level sets $A_\alpha :=\{u \in X : \varPhi(u)\le \alpha\}$ for all $\alpha>0$  such that 
\begin{equation}\label{2.4}
\E\exp\bigl(r(\theta) \lag \varPhi,\zeta_\theta\rag \bigr)\le 
Ce^{c r(\theta)}\quad \text{for $\theta\in \Theta$},
\end{equation} where $C$ and $c$ are positive numbers. Then, for any~$a\ge0$, the level set
\begin{equation}\label{2.5}
\KK_a:=\{\sigma\in \PP(X): \lag \varPhi,\sigma\rag\le a\}
\end{equation} 
is compact in $\PP(X)$,  and~$\{\zeta_\theta\}$ is  an  exponentially tight family in~$\PP(X)$.
\end{lemma}

\begin{proof} 
Since $\varPhi : X  \to [0,+\ty]$ is lower semicontinuous, the set $\{\varPhi>\alpha\}\subset X$ is open for any $\alpha\ge0$. Using the Fatou lemma and the portmanteau theorem, for any sequence $\{\mu_n\}\subset\PP(X)$ converging weakly to a limit $\mu\in\PP(X)$, we can write
\begin{align*}
\liminf_{n\to\infty}\langle\varPhi,\mu_n\rangle
&=\liminf_{n\to\infty}\int_0^\infty \mu_n(\{\varPhi>\alpha\})\,\dd \alpha 
\ge \int_0^\infty \liminf_{n\to\infty}\mu_n(\{\varPhi>\alpha\})\,\dd \alpha\\
&\ge \int_0^\infty \mu(\{\varPhi>\alpha\})\,\dd \alpha=\langle\varPhi,\mu\rangle,
\end{align*}
whence we conclude that the function $\lag \varPhi,\cdot\rag: \PP(X)\to \R$ is also lower semicontinuous. It follows that the level set~$\KK_a$ is closed for any~$a\ge0$. For any~$\sigma\in \KK_a$ and $\e>0$, we have
$$
\sigma(A_{a/\e}^c)\le \int_{A_{a/\e}^c} \frac{  \e \varPhi (u)}{a} \sigma(\dd u)\le \frac{\e }{a}\lag \varPhi,\sigma\rag\le\e. 
$$ Since  the set~$A_{a/\e}  $ is compact in~$X$, the  Prokhorov compactness criterion (see Theorem~11.5.4 in~\cite{dudley2002}) implies that~$\KK_a$ is compact.  The Chebyshev  inequality combined with~\eqref{2.4} now gives
\begin{equation}\label{2.8}
\limsup_{\theta\in \Theta}\frac{1}{r(\theta)}\log \pP\{\zeta_\theta \in \KK_a^c\}\le c-a.
\end{equation}
This proves that $\{\zeta_\theta\}$ is exponentially tight in~$\PP(X)$.
\end{proof}

Recal that $I(\Gamma)=\inf_{\sigma\in \Gamma} I (\sigma)$, $\Gamma\subset \PP(X)$. The following theorem is the main result  of this section. Its proof is based on the arguments of the paper~\cite{kifer-1990}  and uses some intermediate results from there. 

\begin{theorem}\label{T:2.3}
Suppose that limit~\eqref{2.1} exists for any~$V\in C_b(X)$, and that the conditions of Lemma~\ref{L:2.2}  are  satisfied.  Then the relation~\eqref{2.2} defines a good rate function~$I$, and the following upper bound holds  for any closed subset~$F\subset\PP(X):$ 
\begin{equation}\label{2.6}
\limsup_{\theta\in\Theta} \frac{1}{r(\theta)}\log\pP\{\zeta_\theta\in F\}\le -I(F).
\end{equation}
Furthermore, suppose that  there exists a vector space~$\VV\subset C_b(X)$ such that the restrictions of its functions to any compact set $K\subset X$ form a dense subspace in $C(K)$, and that, for any~$V\in \VV $, there is at most one equilibrium state. Then the following lower bound holds for any open subset~$G\subset\PP(X):$
\begin{equation}\label{2.7}
\liminf_{\theta\in\Theta} \frac1{r(\theta)}\log\pP\{\zeta_\theta\in G\}\ge -I(G).
\end{equation}
\end{theorem}

\begin{proof} 
{\it Step~1}.  Let us show that~$I$ is a good rate function. We need to prove that the set $T_a:=\{\sigma\in \PP(X): I(\sigma)\le a\}$ is compact for any $a\ge0$. Since~$I$ is lower semicontinuous, $T_a$~is closed. In view of Lemma~\ref{L:2.2}, it suffices to show that $T_a\subset \KK_b$ for some $b\geq 0$. 

Let $V_n\in C_b(X)$ be any sequence of non-negative functions converging  pointwise  to~$\varPhi $ and satisfying the inequality $V_n(u)\le \varPhi (u)$ for all $u\in X$. For example, one can take $V_n(u)=n\wedge \inf_{v\in X} (\varPhi (v)+n d(u,v))$ for any $n\ge1$ and $u\in X$. It follows from~\eqref{2.1} and~\eqref{2.4} that
$$
Q(V_n)\le\limsup_{\theta\in\Theta} \frac{1}{r(\theta)}\log \E \exp\left(r(\theta) \lag \varPhi,\zeta_\theta\rag)\right)\le c \quad\text{for all~$n\ge1$}. 
$$
Combining this with~\eqref{2.2}, we get 
$$
\lag V_n, \sigma \rag \le a+c \quad \text{for any~$\sigma\in T_a$ and~$n\ge1$.
}
$$ Using Fatou's lemma to pass to the limit as~$n\to\ty$, we get~$\lag \varPhi,\sigma\rag\le a+c$. This shows that~$T_a\subset \KK_{a+c}$. Thus, $I$~is a good rate function.

\medskip
{\it Step~2}. The proof of the upper bound~\eqref{2.6} is essentially the same as the one given by Kifer~\cite{kifer-1990} (who follows the argument in~\cite[Theorem~2.1]{Ac85}). Since~$\{\zeta_\theta\}$ is exponentially tight, in view of Lemma~1.2.8 in~\cite{ADOZ00}, it suffices to establish~\eqref{2.6} for compact sets $F \subset \PP(X)$. The case $I(F)\le 0$ is trivial, since the left-hand side of~\eqref{2.6} is non-positive. Let us assume that $0<I(F)<\ty$ and fix any $\e>0$. Given $W\in C_b(X)$, we introduce the open set
$$
\Gamma_\e(W)=\left\{\nu\in \PP(X): \lag W,\nu\rag-Q(W)>I(F)-\e\right\}.
$$
By definition of $I$, we have
$$
F\subset \{\nu\in \PP(X): I(\nu)> I(F)-\e\}=\bigcup_{W\in C_b(X)} \Gamma_\e(W). 
$$
Since $F$ is compact, there are finitely
many functions $W_1, \ldots, W_l\in C_b(X)$ such that 
$$
F\subset \bigcup_{i=1}^l \Gamma_\e(W_i).
$$
Combining this with the Chebyshev inequality, we get  
\begin{align*}
\pP\{\zeta_\theta\in F\}&\le \sum_{i=1}^l \pP\{\zeta_\theta\in \Gamma_\e(W_i)\} \nonumber\\
&= \sum_{i=1}^l \pP\{\lag W_i, \zeta_\theta\rag>Q(W_i)+I(F)-\e \} \nonumber\\
&\le \sum_{i=1}^l \exp\bigl(-r(\theta)(Q(W_i)+I(F)-\e)\bigr) \int_{\Omega_\theta} \exp\bigl(r(\theta)\lag W_i, \zeta_\theta\rag\bigr)\dd \pP.
\end{align*}
This implies that 
$$
\limsup_{\theta\in\Theta} \frac1{r(\theta)}\log \pP\{\zeta_\theta\in F\}\le -I(F)+\e.
$$
Since $\e>0$ is arbitrary, we obtain~\eqref{2.6}. 

We now assume that $I(F)=\ty$ and fix an arbitrary $N>0$. Repeating the above arguments and using the open sets
$$ 
\Gamma^N(W)=\left\{\nu\in \PP(X): \lag W,\nu\rag-Q(W)>N\right\}
$$
instead of $\Gamma_\e(W)$, we arrive at 
$$
\limsup_{\theta\in\Theta} \frac1{r(\theta)}\log \pP\{\zeta_\theta\in F\}\le -N.
$$
Since $N>0$ is arbitrary, we get~\eqref{2.6}.

\medskip
{\it Step~3}.
Let us turn to the proof of the  lower bound~\eqref{2.7}. As in~\cite{kifer-1990}, we   first prove it for an auxiliary family of $\R^n$-valued random variables~$\{\zeta_\theta^n\}$. More precisely, taking any functions~$V_1,\dots,V_n\in \VV$,  we define the mapping 
\begin{equation} \label{4.09}
f_n: \PP(X)\to \R^n, \quad  
f_n(\mu)=\bigl(\lag V_1,\mu\rag, \ldots, \lag V_n,\mu\rag\bigr),
\end{equation}
and denote $\zeta_\theta^n:= f_n(\zeta_\theta)$. The following result is a generalisation of Lemma~2.1 in~\cite{kifer-1990}; its proof is given at the end of this section.

\begin{proposition} \label{L:2.4} 
Under the conditions of Theorem~\ref{T:2.3}, for any closed set $M\subset\R^n$, 
\begin{equation}\label{2.9}
\limsup_{\theta\in\Theta}\frac1{r(\theta)}\log\pP\{ \zeta_\theta^n \in M\}\le -J_n(M),
\end{equation}
and for any open set~$U\subset \R^n$,
  \begin{equation}\label{2.10}
\liminf_{\theta\in\Theta} \frac1{r(\theta)}\log\pP\{  \zeta_\theta^n \in U\}\ge -J_n(U),
\end{equation}
where~$J_n(\Gamma)=\inf_{\sigma\in f_n^{-1}(\Gamma)} I(\sigma)$ for $\Gamma\subset \R^n.$ 
\end{proposition}

\medskip
{\it Step~4}. 
We can now complete the proof of~\eqref{2.7}. For any sequence $\{V_k\}\subset \VV$ such that $\|V_k\|_\ty=1$, define the following function on $\PP(X)\times \PP(X)$:
\begin{equation} \label{4.011}
d(\mu,\nu):=\sum_{k=1}^\ty 2^{-k}|\lag V_k,\mu\rag-\lag V_k,\nu\rag|.
\end{equation}
We claim that, given a compact subset $\KK\subset\PP(X)$, one can choose a sequence $\{V_k\}\subset \VV$ such that the restriction of~$d$ to $\KK\times\KK$ is a metric on~$\KK$ compatible with the weak topology. Indeed, $d$~is a non-negative symmetric function satisfying the triangle inequality. Therefore we only need to ensure that~$d$ separates points and the convergence for~$d$ is equivalent to the weak convergence. 

In view of the Prokhorov compactness criterion, there is an increasing sequence of compact subsets~$K_n\subset X$ such that $\mu(K_n^c)<\frac1n$ for any $\mu\in\KK$. By assumption, the restriction of the functions in~$\VV$ to~$K_n$ is dense in $C(K_n)$ for any $n\ge1$. Since~$C(K_n)$ is separable, we can find a sequence $\{V_k\}\subset \VV$ such that $\|V_k\|_\infty=1$ for any~$k\ge1$, and that the restriction to~$K_n$ of the functions in the vector span of~$\{V_k\}$ is dense for any $n\ge1$. It is straightforward to check that metric~\eqref{4.011} with this choice of~$\{V_k\}$ separates points and generates weak convergence. 
 
\smallskip
To prove~\eqref{2.7}, we consider an open set $G\subset\PP(X)$. The case $I(G)=\ty$ is trivial, so let us assume that $I(G)<\ty$ and fix any $\e>0$. Then there is $\nu_\e\in G$ such that 
\begin{equation}\label{2.20}
I(\nu_\e)\le I(G)+\e.
\end{equation}
Furthermore, for $a=I(G)+c+1+\e$, the compact set $\KK_a\subset\PP(X)$ is such that~\eqref{2.8} holds and $\nu_\e\in\KK_a$. Let~$\{V_k\}\subset\VV$ be a sequence of functions of norm~$1$ such that~\eqref{4.011} metrizes the weak convergence on~$\KK_a$. Since~$G$ is open, we can find a number~$\delta>0$ and an integer~$n\ge1$ such that if $\nu\in\KK_a$ satisfies the inequality
$$
\sum_{k=1}^n2^{-k}|\lag V_k,\nu\rag-\lag V_k,\nu_\e\rag|<\delta,
$$
then $\nu\in G$. Let us endow~$\R^n$ with the norm 
$$
\|x\|_n=\sum_{k=1}^n 2^{-k}|x_k|, \quad x=(x_1, \ldots, x_n).
$$ 
We now define a mapping~$f_n$ by~\eqref{4.09} and set $x_\e=f_n(\nu_\e)$ and $\zeta_\theta^n=f_n(\zeta_\theta)$. The construction implies that the set $f_n^{-1}\bigl(\mathring{B}_{\R^n}(x_\e,\de)\bigl)\cap\KK_a$ is contained in~$G$. It follows that
\begin{align*} 
\pP\{\zeta_\theta\in G\}&\ge\pP\{\zeta_\theta\in G\cap \KK_a\}
\ge \pP\bigl\{\zeta_\theta\in f_n^{-1}\bigl(\mathring{B}_{\R^n}(x_\e,\de)\bigl)\cap\KK_a\bigr\}\\ 
&=\pP\{\zeta_\theta^n\in \mathring{B}_{\R^n}(x_\e,\de)\}
-\IP\{\zeta_\theta\in \KK_a^c\}.
\end{align*}
Furthermore, for $0<v\le u/2$ we have $\log(u-v)\ge \log u-\log 2$. Combining these inequalities with~\eqref{2.8}, \eqref{2.10}, and~\eqref{2.20}, we see that 
\begin{align*}
\liminf_{\theta\in\Theta} \frac1{r(\theta)}\log\pP\{\zeta_\theta\in G\}
&\ge \liminf_{\theta\in\Theta} \frac1{r(\theta)}
\bigl(\log\pP\{\zeta_\theta^n\in \mathring{B}_{\R^n}(x_\e,\de)\}-\log 2\bigr)\\
&\ge -J_n(\mathring{B}_{\R^n}(x_\e,\de))\ge -I_n(x_\e)\\
&\ge -I(\nu_\e) \ge -I(G)-\e.
\end{align*} 
Since $\e>0$ is arbitrary, we obtain~\eqref{2.7}. This completes the proof of the theorem.
\end{proof}

\begin{proof}[Proof of Proposition~\ref{L:2.4}]
The upper bound~\eqref{2.9} follows immediately from~\eqref{2.6}. To prove the lower bound~\eqref{2.10}, we follow the scheme used in~\cite{kifer-1990} (see Lemma~2.1). The additional difficulty in our case  comes from the fact that the image~$L_n:=f_n(\PP(X))$ is not  necessarily a compact subset of~$\R^n$.

\smallskip
{\it Step~1: Construction of a shifted measure\/}. 
Let us fix an open set $U\subset\PP(X)$ and, given $\beta=(\beta_1, \ldots, \beta_n)\in \R^n$, define $V_\beta=\sum_{j=1}^n\beta_jV_j$ and~$Q_n(\beta)=Q(V_\beta)$. We set
$$
I_n(\alpha)=
\inf_{\sigma \in f_n^{-1}(\alpha)} I(\sigma),
$$
where the infimum over an empty set equals~$+\infty$. Direct verification shows that
\begin{align}
Q_n(\beta)&=\sup_{\alpha\in\R^n} (\lag\beta,\alpha\rag_n-I_n(\alpha)),\label{2.11}\\
J_n(U)&=\inf_{\alpha\in  U}I_n(\alpha),
\label{2.12} 
\end{align}  
where~$\lag  \cdot, \cdot\rag_n$ denotes the scalar product in~$\R^n$.
Thus, for any~$\e>0$, there is~$\alpha_\e\in U$ such that 
$I_n(\alpha_\e)<J_n(U)+\e$. Without loss of generality, we can assume that $I_n(\alpha_\e)<\infty$. In view of Proposition~\ref{p9.2}, we can choose~$\alpha_\e$ in  such a  way that, for some~$\beta_\e\in \R^n$, the following equality holds: 
\begin{equation}\label{2.13}
Q_n(\beta_\e)=\lag \beta_\e,\alpha _\e\rag_n-I_n(\alpha_\e).
\end{equation}
We now define a new measure~$\IP_\theta^{(\beta_\e)}$ on~$\Omega_\theta$ whose density with respect to~$\IP_\theta$ is equal to
$$
\frac{\dd\IP_\theta^{(\beta_\e)}}{\dd\IP_\theta}
=e^{-r(\theta)Q_n^\theta(\beta_\e)}
\exp\bigl(r(\theta)\langle V_{\beta_\e},\zeta_\theta\rangle\bigr),
$$
where we set 
$$
Q_n^\theta(\beta_\e)=\frac{1}{r(\theta)}
\log\E_\theta\exp\bigl(r(\theta)\langle V_{\beta_\e},\zeta_\theta\rangle\bigr). 
$$
Using the relation 
$\langle V_{\beta},\zeta_\theta\rangle=\langle \beta,\zeta_\theta^n\rangle_n$, 
it is straightforward to check that
\begin{equation} \label{4.017}
\E_\theta^{(\beta_\e)} g(\zeta_\theta^n)
=e^{-r(\theta)Q_n^\theta(\beta_\e)}
\E_\theta
\bigl\{\exp\bigl(r(\theta)\langle\beta_\e,\zeta_\theta^n\rangle_n\bigr)
g(\zeta_\theta^n)\bigr\}, 
\end{equation}
where $\E_\theta^{(\beta_\e)}$ denotes the expectation defined by $\IP_\theta^{(\beta_\e)}$ and $g:\R^n\to\R$ is any non-negative Borel function. 

\smallskip
{\it Step 2: Exponential tightness of~$\zeta_\theta^n$ for the shifted measure\/}. 
We claim that for any $a>0$ there is $\rho_a>0$  such that 
\begin{equation}\label{2.15}
\limsup_{\theta\in \Theta}\frac{1}{r(\theta)}\log \pP_\theta^{(\beta_\e)}
\{\zeta_\theta^n\notin {B}_{\R^n}(0, \rho_a)\}\le -a. 
\end{equation}
Indeed, taking $g(x)=I_\Gamma(x)$ with $\Gamma={B}_{\R^n}(0, \rho)^c$ in~\eqref{4.017} and using~\eqref{2.1} and the Cauchy--Schwarz inequality, we obtain
\begin{align}
&\limsup_{\theta\in\Theta}\frac{1}{r(\theta)}
\log\pP_\theta^{(\beta_\e)}\{\zeta_\theta^n\notin{B}_{\R^n}(0,\rho)^c\}\notag\\
&\quad=-\lim_{\theta\in\Theta}Q_n^\theta(\beta_\e)
+\limsup_{\theta\in\Theta}\frac{1}{r(\theta)}
\log\E_\theta
\bigl\{\exp\bigl(r(\theta)\langle\beta_\e,\zeta_\theta^n\rangle_n\bigr)
I_{{B}_{\R^n}(0, \rho)^c}(\zeta_\theta^n)\bigr\}\notag\\
&\quad \le -Q(V_{\beta_\e})+\frac12Q(2V_{\beta_\e})+\limsup_{\theta\in\Theta}\frac{1}{2r(\theta)}
\log\IP_\theta\{\zeta_\theta^n\notin{B}_{\R^n}(0,\rho)^c\}. \label{4.019}
\end{align}
Since $\{\zeta_\theta\}$ is exponentially tight for~$\IP_\theta$ and the mapping $f_n$ is continuous, it follows that~$\{\zeta_\theta^n\}$ is exponentially tight in~$\R^n$. Hence, the right-hand side of~\eqref{4.019} can be made smaller than~$-a$  by choosing a sufficiently large~$\rho>0$.

\smallskip
{\it Step 3: LD upper bound for~$\zeta_\theta^n$ for the shifted measure\/}. 
Let us define
\begin{equation}\label{2.18}
I_n^{(\beta_\e)}(\alpha)
=I_n(\alpha)-\bigl(\lag\beta_\e,\alpha\rag-Q_n(\beta_\e)\bigr). 
\end{equation} 
We claim that, for any closed subset $M\subset\PP(\R^n)$, relation~\eqref{2.9} holds  with~$\IP$ and~$J_n$ replaced by~$\IP_\theta^{(\beta_\e)}$ and~$I_n^{(\beta_\e)}$, respectively. Indeed, in view of Step~2 of the proof of Theorem~\ref{T:2.3}, the LD upper bound will be established if for any $\beta\in\R^n$ we prove the existence of the limit (cf.~\eqref{2.1})
\begin{equation} \label{4.021}
R_\e(\beta):=
\lim_{\theta\in\Theta}\frac{1}{r(\theta)}\log\E_\theta^{(\beta_\e)}
\exp\bigl(r(\theta)\langle\beta,\zeta_\theta^n\rangle_n\bigr). 
\end{equation}
Relation~\eqref{4.017} implies that the right-hand side of~\eqref{4.021} is equal to
$$
-\lim_{\theta\in\Theta}Q_n^\theta(\beta_\e)
+\lim_{\theta\in\Theta}\frac{1}{r(\theta)}\log\E_\theta
\exp\bigl(r(\theta)\langle V_\beta+V_{\beta_\e},\zeta_\theta\rangle\bigr). 
$$
Combining this with~\eqref{2.1}, we see that
\begin{equation} \label{4.022}
R_\e(\beta)=-Q(V_{\beta_\e})+Q(V_{\beta}+V_{\beta_\e})
=-Q_n(\beta_\e)+Q_n(\beta+\beta_\e). 
\end{equation}
Thus, the LD upper bound holds for~$\{\zeta_\theta^n\}$ with the rate function
$$
I_n^{(\beta_\e)}(\alpha)=\sup_{\beta\in\R^n}\bigl(\langle\alpha,\beta\rangle_n+Q_n(\beta_\e)-Q_n(\beta+\beta_\e)\bigr). 
$$
Since $I_n$ is the Legendre transform of~$Q_n$, this expression coincides with~\eqref{2.18}. 

\smallskip
{\it Step~4: Completion of the proof}. 
Since~$U$ is open and~$\alpha_\e\in U$, there is $\de>0$ such that~$B_{\R^n}(\alpha_\e,\de)\subset U$. Setting $B=\mathring{B}_{\R^n}(\alpha_\e,\de)$, taking 
$$
g(x)=\exp(-r(\theta)\langle\beta_\e,x-\alpha_\e\rangle_n)I_{B}(x)
$$ 
in~\eqref{4.017}, and using~\eqref{2.13}, we obtain
\begin{align*}
&\liminf_{\theta\in\Theta}\frac{1}{r(\theta)}\log\IP_\theta\{\zeta_\theta^n\in U\}
\ge\liminf_{\theta\in\Theta}\frac{1}{r(\theta)}
\log\IP_\theta\{\zeta_\theta^n\in B\}\\
&\quad=\liminf_{\theta\in\Theta}\frac{1}{r(\theta)}
\log\bigl(e^{r(\theta)(Q_n^\theta(\beta_\e)-\langle \alpha_\e,\beta_\e\rangle_n)}
\E_\theta^{(\beta_\e)}\{e^{-r(\theta)\langle\beta_\e,\zeta_\theta^n-\alpha_\e\rangle_n}I_B(\zeta_\theta^n)\}\bigr)\\
&\quad\ge -I_n(\alpha_\e)-|\beta_\e|\delta
+\liminf_{\theta\in\Theta}\frac{1}{r(\theta)}
\log\IP_\theta^{(\beta_\e)}\{\zeta_\theta^n\in B\}. 
\end{align*}
Since $-I_n(\alpha_\e)>-J_n(U)-\e$ and the positive numbers~$\e$ and~$\delta$ can be chosen to be arbitrarily small, the lower bound~\eqref{2.10} will be established once we prove that 
$$
\liminf_{\theta\in\Theta}\frac{1}{r(\theta)}
\log\IP_\theta^{(\beta_\e)}\{\zeta_\theta^n\in B\}=0. 
$$
To this end, it suffices to show that (recall $L_n=f_n(\PP(X))$)
$$
\lim_{\theta\in\Theta}\IP_\theta^{(\beta_\e)}\{\zeta_\theta^n\in L_n\setminus B\}=0. 
$$
Furthermore, in view of inequality~\eqref{2.15}, we only need to check that, for any   $a>0$, 
\begin{equation} \label{4.023}
\lim_{\theta\in\Theta}
\IP_\theta^{(\beta_\e)}\{\zeta_\theta^n\in B_{\R^n}(0,\rho_a)\setminus B\}=0. 
\end{equation}
Since $B_{\R^n}(0,\rho_a)\setminus B$ is closed, the LD upper bound established in Step~3 shows that
\begin{equation} \label{2.19}
\limsup_{\theta\in\Theta}\frac{1}{r(\theta)}\log
\IP_\theta^{(\beta_\e)}\{\zeta_\theta^n\in B_{\R^n}(0,\rho_a)\setminus B\}
\le -\inf_{\alpha\in B_{\R^n}(0,\rho_a)\setminus B}I_n^{(\beta_\e)}(\alpha). 
\end{equation}
The required assertion will be proved if we show that the infimum on the right-hand side is positive. To this end, it suffices to check that $I_n^{(\beta_\e)}(\alpha)\ne0$ for $\alpha\ne\alpha_\e$. 

It follows from~\eqref{2.13} and~\eqref{2.18} that~$I_n^{(\beta_\e)}(\alpha_\e)=0$.
Suppose, by contradiction,  that $I_n^{(\beta_\e)}(\alpha')=0$ for some~$\alpha'\in \R^n$, $\alpha'\neq \alpha_\e$. Since the level sets of~$I$
are compact, there exists $\sigma_\e\in f_n^{-1}(\alpha_\e)$ and
$\sigma'\in f_n^{-1}(\alpha')$ such that $I_n(\alpha_\e)=I(\sigma_\e)$
and $I_n(\alpha')=I(\sigma')$. Using~\eqref{2.11},~\eqref{2.18}, we get that  $\sigma_\e$ and $\sigma'$ are distinct equilibrium measures
 corresponding to the function~$V_{\beta_\e}\in\VV$. This contradiction proves the required property and completes the proof of Lemma~\ref{L:2.4}. 
\end{proof}

\section{Large-time asymptotics for generalised Markov semigroups}
\label{s5}
This section is devoted to the study of large-time behaviour of trajectories for a class of dual semigroups. The main result is a  generalisation of Theorem~2.1 in~\cite{JNPS-2012} to the unbounded phase space; see also~\cite{szarek-1997,KS-mpag2001,LS-2006} for some related results in the case of Markov semigroups.  

\medskip
Let $X$ be a Polish space, let~$\MM_+(X)$ be the cone of non-negative Borel measures on~$X$, and let $\{P(u,\cdot),u\in X\}\subset\MM_+(X)$ be a generalised Markov kernel. The latter means that the function $u\mapsto P(u,\cdot)$ is continuous and non-vanishing  from~$X$ to the space~$\MM_+(X)$ (with the topology of weak convergence). It follows, in particular, that for any compact subset $K\subset X$ there is $C_K>1$ such that
\begin{equation} \label{3.1}
C_K^{-1}\le P(u,X)\le C_K\quad\mbox{for $u\in K$}. 
\end{equation}
We denote by $P_k(u,\Gamma)$ the iterations of~$P(u,\Gamma)$ and by~$\PPPP_k:C_b(X)\to C_b(X)$ and $\PPPP_k^*:\MM_+(X)\to \MM_+(X)$ the corresponding semigroups; for the exact definitions, see Section~2 in~\cite{JNPS-2012}. In the case~$k=1$, we write~$\PPPP$ and~$\PPPP^*$, respectively. 

In what follows, we shall deal with a more restrictive class of kernels concentrated on the union of a countable family of compact subsets. Namely, let~$\{X_R\}_{R=1}^\infty$ be an increasing  sequence of compact subsets of~$X$ and let~$X_\infty$ be the union of~$\{X_R\}$. We shall say that ${\wwww}:X\to[1,+\infty]$ is a {\it weight function on~$X$\/} if it is measurable and its restriction to~$X_R$ is continuous for any $R\ge1$. We denote by $C_\wwww(X)$ (respectively, $L_\wwww^\infty(X)$) the space of continuous (measurable) functions $f:X\to\R$ such that $|f(u)|\le C\wwww(u)$ for all $u\in X$. Let us endow~$C_\wwww(X)$ (respectively, $L_\wwww^\infty(X)$) with the seminorm
$$
\|f\|_{L_\wwww^\infty}=\sup_{u\in X}\frac{|f(u)|}{\wwww(u)}.
$$The spaces $C_\wwww(X_\ty)$ and  $L_\wwww^\infty(X_\ty)$ are defined in a similar way. Let~$\MM_\wwww(X)$ be  the space of measures $\mu\in\MM_+(X)$ such that $\langle \wwww,\mu\rangle<\infty$ and let $\PP_\wwww(X)=\MM_\wwww(X)\cap \PP(X)$.
Note that the integral $\lag f,\mu\rag$ is well defined for any $f\in L_\wwww^\infty(X)$ and $\mu\in\MM_\wwww(X)$. 

\smallskip
We shall assume that the kernels~$P_k(u,\Gamma)$ satisfy the following hypothesis.
\begin{description}
\item[Growth condition.]
The subset~$X_\infty$ is dense in~$X$, the measures $P_{k_0}(u,\cdot)$ are concentrated on~$X_\infty$ for any $u\in X$ and an integer $k_0\ge1$, and there exists a weight function $\wwww:X\to[1,+\infty]$ and an integer~$R_0\ge1$ such that\,\footnote{The expression $(\PPPP_k\wwww)(u)$ is understood as the integral of the positive function~$\wwww$ against the positive measure $P_k(u,\cdot)$.}
\begin{equation}\label{3.2}
M:=\sup_{k\ge0}
\frac{\|\PPPP_k\wwww\|_{L_\wwww^\infty}}{\|\PPPP_k{\mathbf1}\|_{R_0}}<\infty,
\end{equation}
where $\mathbf 1$ is the function on~$X$ identically equal to~$1$, $\|\cdot\|_R$ denotes the~$L^\infty$ norm on~$X_R$, and we set $\infty/\infty=0$.  
\end{description}
Note that if $P_k$ satisfy this condition, then the operators~$\PPPP_k$ are well defined on the space~$L_\wwww^\infty(X)$ and map it  into itself.  Furthermore, inequality~\eqref{3.2} implies that 
$$
\beta_k(R):= {\|\PPPP_k\wwww\|_R} <\infty 
\quad\quad\mbox{for all $k\ge0$ and $R\ge1$}.
$$

\smallskip
To formulate the main result of this section, we need some additional definitions. Given any family $\CC\subset C_b(X)$, we denote by~$\CC^\wwww$ the vector space of those functions $f\in L_\wwww^\infty(X)$ that can be approximated, within any accuracy with respect to the norm~$\|\cdot\|_{L_\wwww^\infty}$,  by finite linear combinations of functions from~$\CC$. Note that the restriction of any function $f\in\CC^\wwww$ to~$X_R$ is continuous. A family $\CC\subset C_b(X)$ is said to be {\it determining\/} if any two measures~$\mu,\nu\in\MM_+(X)$ satisfying the relation $\lag f,\mu\rag=\lag f,\nu\rag$ for all~$f\in\CC$ coincide. A sequence of functions~$f_k: X\to\R$ is said to be {\it uniformly equicontinuous on\/}~$X_R$ if for any~$\e>0$ there is~$\de>0$ such that $|f_k(u)-f_k(v)|<\e$ for any~$u\in X_R$, $v\in B_{X_R}(u,\de)$, and~$k\ge1$. Recall that a nonzero~$\mu\in\MM_+(X)$ is called an {\it eigenvector\/} for~$\PPPP^*$ if there is $\lambda\in\R$ such that
\begin{equation} \label{3.3}
\PPPP^*\mu=\lambda \mu.
\end{equation}

\begin{theorem} \label{T:3.1}
Let~$P(u,\Gamma)$ be a generalised Markov kernel satisfying the growth condition formulated above and possessing the following properties. 
\begin{description}
\item[Uniform Feller property.]
There is a determining family~$\CC\subset C_b(X)$ and an integer~$R_0\ge1$ such that~${\mathbf1}\in\CC$ and the sequence~$\{\|\PPPP_k{\mathbf1}\|_R^{-1}\PPPP_kf,k\ge0\}$ is uniformly equicontinuous on~$X_R$ for any~$f\in \CC$ and~$R\ge R_0$. 
\item[Uniform irreducibility.] 
For sufficiently large $\rho\ge1$, any integer $R\ge 1$, and any $r>0$, there is an integer $l=l(\rho,r,R)\ge1$ and a positive number $p=p(\rho,r)$, not depending on~$R$, such that
\begin{equation} \label{3.4}
P_l(u,B_{X_\rho}(\hat u,r))\ge p\quad\mbox{for all~$u\in X_R ,\hat u\in X_\rho$}. 
\end{equation}
\end{description}
Then~$\PPPP^*$ has at most one eigenvector $\mu\in\PP_\wwww(X)$ satisfying the following condition for any~$k$: 
\begin{equation} \label{3.5}
\beta_k(R)\int_{X\setminus X_R}\wwww\,\dd\mu\to0
\quad\mbox{as $R\to\infty$}.
\end{equation}
Moreover, if such a measure~$\mu$ exists, then the corresponding eigenvalue~$\lambda$ is positive, the support of~$\mu$ coincides with~$X$, and     there is a  non-negative function $h\in L^\ty_\wwww(X)$  such that $\lag h,\mu\rag=1$, the restriction of~$h$ to~$X_R$ belongs to $C_+(X_R)$,  
\begin{equation} \label{3.6}
(\PPPP h)(u)=\lambda h(u)\quad\mbox{for $u\in X$},
\end{equation}
  and for any $f\in \CC^\wwww$ and $R\ge1$, we have 
\begin{equation}
\lambda^{-k}\PPPP_k f\to\lag f,\mu\rag h
\quad\mbox{in~$C(X_R)\cap L^1(X,\mu)$ as~$k\to\infty$}. \label{3.7}
\end{equation}
Finally, if a Borel set $B\subset X$ is such that
\begin{align} 
 \sup_{u\in B}\biggl(\int_{X\setminus X_R}\wwww(v)\,P_m(u,\dd v)\biggr) &\to0
\quad\mbox{as $R\to\infty$} \label{3.8}
\end{align} 
for some integer $m\ge1$, then for any $f\in \CC^\wwww$ we have
\begin{equation}
\lambda^{-k}\PPPP_k f\to\lag f,\mu\rag h
\quad\mbox{in~$L^\ty(B) $ as~$k\to\infty$}. \label{3.9}
\end{equation}
\end{theorem}

\begin{proof} 
We begin with a number of simple remarks. Let~$\lambda\in\R$ be an eigenvalue for~$\PPPP^*$ corresponding to an eigenvector $\mu\in \PP_\wwww(X)$. Since $P_{k_0}(u,\cdot)$ is concentrated on~$X_\infty$, it follows from~\eqref{3.3} that 
\begin{align}
\lambda^{k_0}&=\int_XP_{k_0}(u,X)\mu(du), \label{3.10}\\
\lambda^{k_0}\mu(X\setminus X_\infty)&=\int_XP_{k_0}(u,X\setminus X_\infty)\mu(\dd u)=0.
\label{3.010}
\end{align}
The lower bound in~\eqref{3.1} implies that the right-hand side of~\eqref{3.10} is positive, and it follows from~\eqref{3.010} that~$\mu$ is concentrated on~$X_\infty$. Furthermore, for any $\hat u\in X_\ty$ and $r>0$, the relation $\PPPP_l^*\mu=\lambda^l\mu$ and the uniform irreducibility imply that 
\begin{align*}
\mu\bigl(B(\hat u,r)\bigr)
&=\lambda^{-l}\int_XP_l\bigl(u,B(\hat u,r)\bigr)\mu(\dd u)\\
&\ge \lambda^{-l}\int_{X_{R}}P_l\bigl(u,B_{X_R}(\hat u,r)\bigr)\mu(\dd u)
\ge \lambda^{-l} p\,\mu(X_{R})>0,
\end{align*}
where $R\ge1$ is such that $\mu(X_R)>0$ and $\hat u \in X_R$, and~$l=l(R,r,R)$, $p=p(\rho,r)$ are the constants in~\eqref{3.4}. Thus, 
\begin{equation} \label{3.11}
\mu(B(\hat u,r))>0\quad\mbox{for any $\hat u\in X_\ty$ and $r>0$}.
\end{equation}
Since~$X_\infty$ is dense in~$X$, the support of~$\mu$ must coincide with~$X$. Finally, let us show that if the existence of~$h$ is established, then the uniqueness of~$\mu$   follows immediately from the normalisation conditions and convergence~\eqref{3.7}. Indeed, suppose that~$\tilde\mu\in\PP_\wwww(X)$ is an eigenvector of~$\PPPP^*$ with an eigenvalue~$\tilde\lambda>0$. The above argument { concerning} the support of an eigenvector for~$\PPPP^*$ applies also to~$\tilde\mu$, so that $\supp\tilde\mu=X$. Using convergence~\eqref{3.7} in~$C(X_R)$ and inequality~\eqref{3.15} established below, we easily prove that
$$
\delta_k:=
\bbar\lambda^{-k}\PPPP_kf-\langle f,\mu\rangle h\bbar_{\tilde\mu}\to0
\quad\mbox{as $k\to\infty$},
$$
where $f\in\CC$ is arbitrary. On the other hand, we can write
$$
\delta_k
\ge \biggl|\int_X\bigl(\lambda^{-k}\PPPP_kf-\langle f,\mu\rangle h\bigr)\,\dd\tilde\mu\biggr|
=\bigl|\bigl(\tfrac{\tilde\lambda}{\lambda}\bigr)^k\langle f,\tilde\mu\rangle-\langle f,\mu\rangle \lag h, \tilde\mu \rag\bigr|.
$$
Comparing the two relations above in which $f=\mathbf1$ and using the fact that $\lag h,\tilde \mu\rag\neq0$, we see that $\tilde\lambda=\lambda$ and $\lag h, \tilde\mu \rag=1$. It follows that $\langle f,\tilde\mu\rangle=\langle f,\mu\rangle$ for any $f\in\CC$. Since~$\CC$ is a determining family, we conclude that $\mu=\tilde\mu$. 

\smallskip
We thus assume that~$\PPPP^*$ has an eigenvector $\mu\in\PP_\wwww(X)$ satisfying~\eqref{3.5} and prove that there is~$h\in L^\infty_\wwww(X)$ such that the restriction of~$h$ to~$X_\infty$ is positive, $\langle h,\mu\rangle=1$, and that~\eqref{3.7} holds. The proof of these facts is split into several steps. In what follows, replacing~$P(u,\Gamma)$ by~$\lambda^{-1}P(u,\Gamma)$ if necessary, we may assume  that~$\lambda=1$. 

\smallskip
{\it Step~1}. 
Let us prove that  for any~$f\in\CC$ and~$R\ge 1$, we have
\begin{equation} \label{3.12}
 \|\PPPP_{k}f\|_R\le C_{f,R} \quad\mbox{for all~$k\ge1$},
\end{equation}
where~$C_{f,R}$ is a constant not depending on~$k$. Clearly, it suffices to prove that 
\begin{equation} \label{3.13}
 \|\PPPP_{k} \mathbf1    \|_R\le C_{\mathbf1,R} \quad\mbox{for all~$k\ge1$}.
\end{equation}
 Let us suppose that there is a sequence~$k_j\to\infty$ and  an integer $R\ge1$ such that 
\begin{equation} \label{3.14}
\|\PPPP_{k_j}\mathbf1\|_R \to+\infty\quad\mbox{as~$j\to\infty$}.
\end{equation}
In view of the uniform Feller property, we can assume that 
$$
\|\PPPP_{k_j}\mathbf1\|_{R}^{-1}\PPPP_{k_j}\mathbf1\to g
\quad\mbox{in~$C(X_R)$ as~$j\to\infty$},
$$
where~$g\in C(X_R)$ is a non-negative function whose norm is equal to~$1$. We now write
$$
\int_{X_R} g(u)\mu(\dd u)
=\lim_{j\to\infty} \|\PPPP_{k_j}\mathbf1\|_{R}^{-1}
\int_{X_R}\PPPP_{k_j}\mathbf1(u)\mu(\dd u)
\le \lim_{j\to\infty} \|\PPPP_{k_j}\mathbf1\|_{R}^{-1}\langle \mathbf1,\mu\rangle=0. 
$$
On the other hand, since $\mu(X\setminus X_\infty)=0$ and $g\ge0$ is a non-zero continuous function on~$X_R$, the left-hand side of this relation is positive for sufficiently large~$R$. Since the validity of convergence~\eqref{3.14} for some integer~$R=R_0\ge1$ implies that it is true for any $R>R_0$, we arrive at a contradiction. We have thus established~\eqref{3.12}. 

\smallskip
{\it Step~2}. 
Let us prove the existence of a non-negative function $h\in L_\wwww^\infty(X)$  satisfying~\eqref{3.6} with $\lambda=1$. As was mentioned after the formulation of the growth condition, the operators~$\PPPP_k:C_\wwww(X)\to C_\wwww(X)$ are well defined. Moreover, the fact that~$P(u,\cdot)$ is concentrated on~$X_\infty$ for any~$u\in X$ and inequalities~\eqref{3.2} and~\eqref{3.12} with $f=\mathbf1$ imply that 
\begin{equation} \label{3.15}
\sup_{k\ge0}\|\PPPP_kf\|_{L_\wwww^\infty(X)}\le M_1\,\|f\|_{L_\wwww^\infty(X_\ty)}
\quad\mbox{for any $f\in L_\wwww^\infty(X_\ty)$},
\end{equation}
where $M_1=MC_{{\mathbf1},R_0}$.   The uniform Feller property and inequality~\eqref{3.12} imply that the sequence~$\{\PPPP_k\mathbf1\}$ is uniformly equicontinuous on~$X_R$ for any~$R\ge1$. It follows that so is the sequence 
$$
h_k:=\frac1k\sum_{l=0}^{k-1}\PPPP_l\mathbf1. 
$$
Applying the diagonal process, we construct a function~$h:X_\infty\to\R_+$ and a sequence $k_j\to\infty$ such that 
\begin{equation}  \label{3.16}
\|h_{k_j}-h\|_R\to0\quad\mbox{as $j\to\infty$ for any $R\ge1$}. 
\end{equation} 
This yields that  the restriction of~$h$  to~$X_R$ is continuous for any~$R\ge1$. Furthermore, it follows from~\eqref{3.15} that \begin{equation} \label{3.17}
\Big\|\frac h\wwww\Big\|_R\le \sup_{k\ge0} \|h_k\|_{L_\wwww^\infty}\le M_1\|\mathbf1\|_{L_\wwww^\infty}\quad
\mbox{for all $R\ge1$}, 
\end{equation}
and therefore  $h\in L_\wwww^\infty(X_\ty)$. We claim that the mean value of~$h$ with respect to~$\mu$ is equal to~$1$, and that  equality~\eqref{3.6} holds for any~$u\in X_\ty$. Indeed, relations~\eqref{3.3} and~\eqref{3.17} and the Lebesgue theorem on dominated convergence imply that 
$$
\langle h,\mu\rangle=\lim_{j\to\infty}\langle h_{k_j},\mu\rangle 
=\langle \mathbf1,\mu\rangle=1.
$$
In particular, $h$ is a non-zero function. To prove~\eqref{3.6} for~$u\in X_\ty$, note that
$$
(\PPPP h_{k_j})(u)=\int_XP(u,\dd v)h_{k_j}(v)
=h_{k_j}(u)+\frac{1}{k_j}\bigl(\PPPP_{k_j}\mathbf1(u)-\mathbf1(u)\bigr). 
$$
The Lebesgue theorem  combined with~\eqref{3.17} implies that, for $u\in X_\infty$,  the left-hand side of this relation converges to~$(\PPPP h)(u)$, while~\eqref{3.15} and~\eqref{3.16} show that its right-hand side converges to~$h(u)$. 

For any $u\in X$, the function  $(\PPPP  h)(u)$ is well defined and, by~\eqref{3.15}, satisfies the inequality
$$ 
\| \PPPP   h\|_{L^\ty_\wwww} \le M_1 \|h\|_{L^\ty_\wwww(X_\ty)}.
$$ 
Thus, defining $h(u):= \PPPP   h (u)$  for any $u\in X\backslash X_\ty$,  we obtain a non-negative function $h\in L^\ty_\wwww(X)$ satisfying~\eqref{3.6} for any~$u\in X$.

\smallskip
{\it Step~3}. 
Let us prove that~$h(u)>0$ for all~$u\in X_\infty$. Indeed, let~$R\ge1$ be such that $u\in X_R$. Since $\langle h,\mu\rangle=1$, there is an integer $\rho\ge1$ and a point $\hat u\in X_\rho$ such that $h(\hat u)>0$. By continuity, there is~$r>0$ such that $h(v)\ge r$ for any~$v\in B_{X_\rho}(\hat u,r)$. Combining this with~\eqref{3.6}, for any~$u\in X_R$ we derive
\begin{align}\label{3.18}
h(u)=\PPPP_lh(u)
&=\int_XP_l(u,dv)h(v)\ge\int_{B_{X_\rho}(\hat u,r)}P_l(u,dv)h(v)\nonumber\\
&\ge rP_l\bigl(u,B_{X_\rho}(\hat u,r)\bigr)\ge rp>0,
\end{align}
where we assumed with no loss of generality that~$\rho$ is so large that inequality~\eqref{3.4} holds, and~$l=l(\rho,r,R)\ge1$ and~$p=p(\rho,r)$ are the constants appearing in this inequality. 

\smallskip
{\it Step~4.}
To prove convergence~\eqref{3.7}, we first show that it suffices to establish it for functions in~$\CC$. Indeed, for any $f\in \CC^\wwww$ and any~$\e>0$ one can find a function $g\in C_b(X)$ which is a finite linear combination of elements of~$\CC$ such that $\|f-g\|_{L_\wwww^\infty}<\e$. 
Combining this with~\eqref{3.15}, we now write
$$
\bigl\|\PPPP_kf-\langle f,\mu\rangle h\bigr\|_R
\le \bigl\|\PPPP_kg-\langle g,\mu\rangle h\bigr\|_R
+\e\bigl(M_1\|\wwww\|_R+\bbar\wwww\bbar_\mu\|h\|_R\bigr). 
$$
Since~$\e>0$ is arbitrary, the second term on the right-hand side of this relation can be made arbitrary small, while the first one goes to zero as $k\to\infty$. A similar argument shows that 
$$
\bbar\PPPP_kf-\langle f,\mu\rangle h\bbar_\mu\to0
\quad\mbox{as $k\to\infty$}.
$$

\smallskip
{\it Step~5.}
We now prove~\eqref{3.7} for $f\in\CC$. Setting $g=f-\lag f,\mu\rag h$ and $g_k=\PPPP_kg$, we need to prove that $g_k\to0$ in~$C(X_R)$ for any~$R\ge1$. Since~$\{g_k,k\ge0\}$ is uniformly equicontinuous on~$X_R$, the required assertion will be established if we prove that 
\begin{equation} \label{3.19}
\bbar g_{k}\bbar_\mu\to0\quad\mbox{as~$k\to\infty$}.
\end{equation}
For any~$\varphi \in L^\infty_\wwww(X)$, we have
$$
\bbar\PPPP \varphi\bbar_\mu=\langle|\PPPP \varphi|,\mu\rangle\le\langle\PPPP |\varphi|,\mu\rangle
=\langle|\varphi|,\mu\rangle=\bbar\varphi\bbar_\mu. 
$$
Thus, the sequence~$\{\bbar g_k\bbar_\mu\}$ is non-increasing, and it suffices to show that there is a sequence of integers~$k_j$ such that~$\bbar g_{k_j}\bbar_\mu\to 0$  as~$j\to\ty$. 

Let us assume that for any integer~$\rho\ge1$ there is a sequence~$k_j=k_j(\rho)\to\ty$ such that 
\begin{equation}\label{3.20}
\|g_{k_j}^+\|_\rho\to0 \quad\mbox{as~$j\to\infty$}.
\end{equation}  
Passing to a subsequence, we can assume that~\eqref{3.20} holds for any $\rho\ge1$ and a universal sequence~$\{k_j\}$. Then, in view of~\eqref{3.15}, we have
$$
\bbar g_{k_j}^+\bbar_\mu=
\biggl(\int_{X\setminus X_\rho}+\int_{X_\rho}\biggr) g_{k_j}^+\,\dd\mu
\le M_1\|g\|_{L_\wwww^\infty}\int_{X\setminus X_\rho}\wwww\,\dd\mu
+\|g_{k_j}^+\|_\rho. 
$$
Combining this with~\eqref{3.20} and the inequality $\langle\wwww,\mu\rangle<\infty$ (which follows immediately from~\eqref{3.5}), we see that $\bbar g_{k_j}^+\bbar_\mu\to0$ as $j\to\infty$. 
Using the relation $\langle g_k,\mu\rangle=\langle g_k^+,\mu\rangle-\langle g_k^-,\mu\rangle=0$, we derive
\begin{equation*}
\bbar g_{k_j}\bbar_\mu=\langle g_k^+,\mu\rangle+\langle g_k^-,\mu\rangle
=2\,\bbar g_{k_j}^+\bbar_\mu\to0\quad\mbox{as~$j\to\infty$}.
\end{equation*}
A similar argument shows that~\eqref{3.19} holds as soon as for any $\rho\ge1$ there is a sequence~$k_j(\rho)\to\ty$ such that 
\begin{equation}\label{3.21}
\|g_{k_j}^-\|_\rho\to0 \quad\mbox{as~$j\to\infty$}.
\end{equation}
In the next step, we prove that~\eqref{3.20} and~\eqref{3.21}   are necessarily satisfied for all~$\rho\ge1$. 

\smallskip
{\it Step~6.}
Assume, by contradiction, that for an integer~$\rho\ge1$ there is no subsequence~$k_j\to\ty$ satisfying~\eqref{3.20} and~\eqref{3.21}. Then one can find sequences~$\{u_k^\pm\}\subset X_\rho$ and a number~$\alpha>0$ such that
$$
g_k^+(u_k^+)=\max_{u\in X_\rho}  g_k^+(u)\ge\alpha, \quad
g_k^-(u_k^-)=\max_{u\in X_\rho}  g_k^-(u)\ge\alpha.
$$
Let us show that, for any $R\ge1$, we have
\begin{equation}\label{3.22}
\PPPP_l  g_k^\pm(u)-A(R)^{-1} \bbar g_k^\pm\bbar_\mu
+A(R)^{-1} \gamma(R)\ge0 \quad\mbox{for all~$u\in X_R$},
\end{equation}
where $l\ge1$ is the integer arising in the uniform irreducibility condition, 
$$
A(R)=2(p\alpha)^{-1}M_1\|g\|_{L_\wwww^\infty}\beta_l(R), \quad
\gamma(R)=M_1\|g\|_{L_\wwww^\infty}
\int_{X\setminus X_R}\wwww\,\dd\mu,
$$
and~$p$ is the constant in~\eqref{3.4}. 
Indeed, since~$g_k^\pm$ are uniformly equicontinuous, we can find~$r>0$ not depending on~$k$ such that
\begin{equation} \label{3.23}
g_k^\pm(u)\ge  \alpha/2   \quad
\mbox{for all~$u\in B_{X_\rho}(u_k^\pm,r)$}. 
\end{equation} 
Using~\eqref{3.15} and~\eqref{3.23}, we obtain
\begin{align*}
\sup_{u\in X_R}\PPPP_l  g_k^\pm(u)
&\le \|g_k^\pm\|_{L_\wwww^\infty}\sup_{u\in X_R}(\PPPP_l\wwww)(u)
\le M_1\|g\|_{L_\wwww^\infty}\beta_l(R),\\
\inf_{u\in X_R}\PPPP_l  g_k^\pm(u)
&\ge \inf_{u\in X_R}\int_{B_{X_\rho}(u_k^\pm,r)}P_l(u,\dd v)  g_k^\pm(v)
\ge \frac{p_l \alpha}{2}\,.
\end{align*}
It follows that 
$$
\sup_{u\in X_R}\PPPP_l  g_k^\pm(u)\le   A(R)\inf_{u\in X_R}\PPPP_l  g_k^\pm(u).
$$
Using again~\eqref{3.15} and the invariance of~$\mu$, we derive
\begin{align*}
\bbar g_k^\pm\bbar_\mu = \bbar \PPPP_l g_k^\pm\bbar_\mu 
&= \int_{X_R} \PPPP_l  g_k^\pm \dd\mu+ \int_{X\backslash X_R} \PPPP_l  g_k^\pm \dd\mu\\
&\le A(R) \inf_{u\in X_R}\PPPP_l  g_k^\pm(u)+ \gamma(R).
\end{align*} 
This inequality implies~\eqref{3.22}. 

We now estimate $\bbar g_{k+l}\bbar_\mu=\bbar\PPPP_lg_k\bbar_\mu$. To this end, we fix an integer~$R\ge1$ and write
\begin{align}
\bbar g_{k+l}\bbar_\mu
&= \int_{X_R}|\PPPP_l  g_k| \mu(\dd u)
+\int_{X\backslash X_R}|\PPPP_l  g_k| \mu(\dd u) \notag\\
&=\int_{X_R}\bigl|(\PPPP_l  g_k^+-A(R)^{-1} \bbar g_k^+\bbar_\mu)-
(\PPPP_l  g_k^--A(R)^{-1} \bbar g_k^-\bbar_\mu)\bigr|\,\dd\mu+\gamma(R) \notag\\
&\le \int_{X_R}\Bigl(\bigl|\PPPP_l  g_k^+-A(R)^{-1} \bbar g_k^+\bbar_\mu\bigr|+\bigl|\PPPP_l  g_k^-
-A(R)^{-1} \bbar g_k^-\bbar_\mu\bigr|\Bigr)\,\dd\mu+\gamma(R).
\label{3.24} 
\end{align}
It follows from~\eqref{3.22} that 
$$
\int_{X_R}\bigl|\PPPP_l  g_k^\pm-A(R)^{-1} \bbar g_k^\pm\bbar_\mu\bigr|\,\dd\mu \le  \int_{X_R}
\bigl(\PPPP_l  g_k^\pm-A(R)^{-1} \bbar g_k^\pm\bbar_\mu\bigr)\dd\mu 
+2A(R)^{-1}\gamma(R). 
$$
Combining this with~\eqref{3.24} and using the invariance of~$\mu$, we obtain
\begin{align}
\bbar g_{k+l}\bbar_\mu
&\le\int_{X_R}\PPPP_l( g_k^++  g_k^-)\,\dd\mu
-A(R)^{-1} \mu(X_R)\bigl(\bbar g_k^+\bbar_\mu+\bbar g_k^-\bbar_\mu\bigr)+\e(R)\notag\\
&\le a(R)\bbar g_k\bbar_\mu+ \e(R),
\label{3.25}
\end{align}
where we set $\e(R)=(1+4A(R)^{-1})\gamma(R)$ and $a(R)=1-A(R)^{-1} \mu(X_R)$. Let~$R$ be so large that $\mu(X_R)\ge1/2$ and $a(R)<1$. Then iteration of~\eqref{3.25} results in 
\begin{align*}
\bbar g_{jl}\bbar_\mu
&\le a(R)^j\bbar g\bbar_\mu
+\e(R)\sum_{n=0}^{j-1} a(R)^n
\le a(R)^j\bbar g\bbar_\mu+\e(R)\,\bigl(1-a(R)\bigr)^{-1}\\
&\le a(R)^j\bbar g\bbar_\mu+ (2A(R)+8)\gamma(R)
\le a(R)^j\bbar g\bbar_\mu
+C(\beta_l(R)+1)\int_{X\setminus X_R}\wwww\,\dd\mu,
\end{align*}
where $C=C(l,M_1,\alpha,g)>0$ does not depend on~$R$. It follows from~\eqref{3.5} that the right-hand side of this inequality can be made arbitrarily small by an appropriate choice of~$R$ and~$j$. 

We have thus proved that convergence~\eqref{3.19} holds along a subsequence $k=k_j$. As was explained in Step~5, this implies that $g_k\to0$ in~$C(X_R)$ for any $R>0$, and we arrive at a contradiction.

\smallskip
{\it Step~7}. It remains to prove   convergence~\eqref{3.9} under condition~\eqref{3.8}. 
 For any  $f\in\CC^\wwww$, we have
\begin{align*}
 |(\PPPP_{k+m}  f  )(u)-\lag f,\mu\rag h(u)|
 &\le   \int_{X}| (\PPPP_k f)(v)-\lag f,\mu\rag h(v)|  P_m(u,\dd v) \\
 &= \int_{X_R} + \int_{X\backslash X_R} =: I_k(R,u)+J_k(R,u).
 \end{align*}
By~\eqref{3.7}, for any $R\ge1$, we have $\sup_{u\in B} I_k(R,u)\to0 $ as $k\to \ty$.   Furthermore, it follows from~\eqref{3.8} and~\eqref{3.15} that
$$
\sup_{u\in B}J_k(R,u)\le C \sup_{u\in B}\biggl(\int_{X\setminus X_R}\wwww(v)\,P_m(u,\dd v)\biggr) \to0\quad\mbox{as $R\to\infty$},
$$
where $C=M_1 \|f\|_{L^\ty_\wwww}+|\lag f,\mu\rag|\|h\|_{L^\ty_\wwww}$. This completes the proof of Theorem~\ref{T:3.1}. 
\end{proof}

\section{Proof of Theorem~\ref{t1.2}}
\label{s8}

In this section, we prove a number of properties of the random dynamical system~\eqref{1.015} and, taking for granted the uniform Feller property, establish the main result of this paper. We shall always assume, often without further stipulation, that the hypotheses of Theorem~\ref{t1.2} are fulfilled. 

\subsection{Hyper-exponential recurrence}
\label{s8.0}
Let~$\ttau_U(R)$ be the first hitting time of the set~$B_U(R)\times\cdots\times B_U(R)\,(\mbox{$\ell$ times})$ for the Markov process~$(\uuu_k,\IP_\uuu)$ associated with~\eqref{1.015}:
$$
\ttau_U(R)=\min\{k\ge0:u_k^1,\dots,u_k^\ell\in B_U(R)\}. 
$$
For any integer $m\ge1$, we set $\varPhi_m(u)=\varPhi(u)^m$, where~$\varPhi$ is the function entering Condition~(A). 

\begin{proposition} \label{p8.1}
Under the hypotheses of Theorem~\ref{t1.2}, for any $\gamma>0$ there are positive numbers $R$, $C$, and~$m$ such that 
\begin{equation} \label{8.1}
\E_\uuu\exp\bigl(\gamma \ttau_U(R)\bigr)\le C\,\varPhi_m(u)
\quad\mbox{for any $\uuu\in H^\ell$}, 
\end{equation}
where $\uuu=[u^1,\dots,u^\ell]\in H^\ell$ and $u=u^\ell$.
\end{proposition}

\begin{proof}
Let us define a stopping time for the Markov process $(u_k,\IP_u)$ associated with~\eqref{1.1}:
$$
\tau_U(R)=\min\{k\ge\ell-1:u_{k-\ell+1},\dots,u_{k}\in B_U(R)\}.
$$
It follows from~\eqref{1.017} that the required inequality~\eqref{8.1} will be established if we prove that
\begin{equation} \label{8.2}
\E_u\exp\bigl(\gamma \tau_U(R)\bigr)\le C\,\varPhi_m(u)
\quad\mbox{for any $u\in H$}.
\end{equation}
To this end,  we use a standard Lyapunov function technique well-known in the theory of Markov processes; see Chapter~3 in~\cite{hasminski1980} or Chapter~8 in~\cite{MT1993}. 

\smallskip
{\it Step~1}.
Given a number $r>0$, we denote by~$\tau(r)$ the first hitting time of the set $\{\varPhi\le r\}=\{v\in H:\varPhi(v)\le r\}$. We claim that $\IP_u\{\tau(r)<\infty\}=1$ for a sufficiently large~$r$ and any $u\in H$, and for any $\gamma>0$ there are positive numbers~$r$, $m$, and~$C$ such that
\begin{equation} \label{8.3}
\E_u\exp\bigl(3\gamma\tau(r)\bigr)\le C\,\varPhi_m(u)
\quad\mbox{for any $u\in H$}.
\end{equation}
Indeed, it follows from~\eqref{1.04} that 
\begin{equation} \label{8.4}
\varPhi_m(S(u)+v)\le 2q^m\varPhi_m(u)+C_m\varPhi_m(v)\quad\mbox{for all $u,v\in H$},
\end{equation}
where $m\ge1$ is an arbitrary integer and~$C_m$ does not depend on~$u$ and~$v$. Combining this with~\eqref{1.8}, for any $\varkappa\in(2q^m,1)$ one can find $r_*>0$ such that
\begin{equation} \label{8.5}
\E_u\varPhi_m(u_1)\le \varkappa\,\bigl(\varPhi_m(u)\vee r_*\bigr)\quad\mbox{for any $u\in H$}.
\end{equation}
Using the Markov property and arguing by induction, we easily prove that if $r\ge r_*$, then  (cf. proof of Lemma~3.6.1 in~\cite{KS-book})
$$
p_k(u):=\E_u\bigl(I_{\{\tau(r)>k\}}\varPhi_m(u_k)\bigr)\le \varkappa^k\varPhi_m(u)\quad\mbox{for all $k\ge0$, $u\in H$}. 
$$
It follows that 
\begin{equation} \label{8.6}
\IP_u\{\tau(r)>k\}\le r^{-m}p_k(u)\le r^{-m}\varkappa^k\varPhi_m(u),
\end{equation}
and therefore $\tau(r)$ is $\IP_u$-almost surely finite for $r\ge r_*$. Furthermore, given $\gamma>0$ we can choose~$m\ge1$ so large that $e^{3\gamma} \varkappa<1$ for some $\varkappa>2q^m$. In this case, inequality~\eqref{8.6} implies that
\begin{align*}
\E_u\exp\bigl(3\gamma \tau(r)\bigr)&\le1+\sum_{k=1}^\infty e^{3\gamma k}\IP_u\{\tau(r)>k-1\}\\
&\le1+r^{-m}\varPhi_m(u)\sum_{k=1}^\infty e^{3\gamma k}\varkappa^{k-1}\le C\,\varPhi_m(u). 
\end{align*}

\smallskip
{\it Step~2}.
Given a positive number~$R$, let us introduce the event 
$$
\Gamma(R)=\{u_{j+1}\in B_U(R)\mbox{ for $j=1,\dots,\ell$}\}. 
$$
Suppose that, for any $r>0$ and $p<1$, we can find~$R$ such that
\begin{equation} \label{8.7}
\IP_u\bigl(\Gamma(R)\bigr)\ge p\quad\mbox{for any $u\in \{\varPhi\le r\}$}.
\end{equation}
In this case, we define a sequence of stopping times~$\sigma_n'$ by the relations
$$
\sigma_0'=\tau(r), \quad \sigma_n'=\min\{k\ge \sigma_{n-1}'+\ell+1:\varPhi(u_k)\le r\}, \quad n\ge1, 
$$
and denote $\sigma_n=\sigma_n'+\ell+1$. Let $\hat n$ be the first integer~$n\ge0$ such that $u_{\sigma_n-j}\in B_U(R)$ for $j=0,\dots,\ell-1$. It follows from~\eqref{8.7} and the strong Markov property that 
\begin{equation} \label{8.8}
\IP_u\{\hat n>k\}\le (1-p)^k\quad\mbox{for any $u\in H$, $k\ge0$}. 
\end{equation}
Furthermore, using~\eqref{8.3}, \eqref{1.04}, and the strong Markov property, it is not difficult to check that
\begin{equation} \label{8.9}
\E_ue^{3\gamma \sigma_k}\le C_1^k\,\varPhi_m(u),
\end{equation}
where $C_1>0$ does not depend on $u\in H$ and $k\ge0$. Combining~\eqref{8.8} and~\eqref{8.9}, for any integers $k,M\ge1$ we now write
\begin{align}
\IP_u\bigl\{\tau_U(R)\ge M\bigr\}&=\IP_u\bigl\{\tau_U(R)\ge M,\sigma_k<M\bigr\}+\IP_u\bigl\{\tau_U(R)\ge M,\sigma_k\ge M\bigr\}
\notag\\
&\le  \IP_u\bigl\{\tau_U(R)>\sigma_k\bigr\}+\IP_u\bigl\{\sigma_k\ge M\bigr\}.\label{8.10}
\end{align}
Since $\{\tau_U(R)>\sigma_k\}\subset\{\hat n>k\}$, the first term on the left-hand side can be estimated by~\eqref{8.8}. Furthermore, it follows from~\eqref{8.9} that
$$
\IP_u\bigl\{\sigma_k\ge M\bigr\}\le C_1^k\,\varPhi_m(u)e^{-3\gamma M}. 
$$
Substituting these inequalities into~\eqref{8.10}, we obtain
$$
\IP_u\bigl\{\tau_U(R)\ge M\bigr\}\le (1-p)^k+C_1^k\,\varPhi_m(u)e^{-3\gamma M}. 
$$
Choosing $k\sim \e M$, with $\e>0$ so small that $\e\log C_1\le \gamma$ and then $R>0$ so large that $\e\log(1-p)^{-1}\ge2\gamma$, we obtain
$$
\IP_u\bigl\{\tau_U(R)\ge M\bigr\}\le 2e^{-2\gamma M}\varPhi_m(u). 
$$
This immediately implies the required inequality~\eqref{8.2}.

\smallskip
{\it Step~3}.
Thus, it remains to establish~\eqref{8.7}. To this end, we introduce the events
$$
\Gamma_1^j(\rho)=\{\pppp(u_j)+\varPhi(u_j)\le\rho\}, \quad \Gamma_2^j(R)=\{u_j\in B_U(R)\}, \quad j\ge1, 
$$
and notice that, for any sequence of positive numbers $\rho_1,\dots,\rho_{\ell+1}$, we have
\begin{equation} \label{8.11}
\Gamma(R)\supset\Gamma_1^1(\rho_1)\cap \bigcap_{j=2}^{\ell+1} \bigl(\Gamma_1^j(\rho_j)\cap \Gamma_2^{j}(R)\bigr).
\end{equation}
It follows from inequalities~\eqref{1.4}, \eqref{1.08}, and~\eqref{1.7} and the inclusion $S(0)\in U$ that, for any $\delta>0$ and $\rho>0$, one can find positive numbers~$\rho'$ and~$R$ such that 
$$
\IP_v\bigl\{\pppp(u_1)+\varPhi(u_1)\le\rho',u_1\in B_U(R)\bigr\}\ge 1-\delta,
$$
where $v\in H$ is any vector satisfying the inequality $\pppp(v)+\varPhi(v)\le\rho$. Using this observation, for any $\delta>0$ one constructs by induction a finite sequence $\rho_1\le\cdots\le\rho_{\ell+1}$ and a number $R>0$ such that 
\begin{align} 
\IP_v\bigl\{\pppp(u_1)+\varPhi(u_1)\le\rho_{1}\bigr\}&\ge 1-\delta\quad
\mbox{for $v\in \{\varPhi\le r\}$},
\label{8.12}\\
\IP_v\bigl\{\pppp(u_1)+\varPhi(u_1)\le\rho_{j+1},u_1\in B_U(R)\bigr\}&\ge 1-\delta\quad
\mbox{for $v\in \Gamma_1(\rho_j)$},
\label{8.13}
\end{align}
where $j=1,\dots,\ell$. Combining this with the Markov property and inclusion~\eqref{8.11}, we obtain
$$
\IP_u\bigl(\Gamma(R)\bigr)\ge (1-\delta)^\ell\,\IP_u\bigl(\Gamma_1(\rho_1)\bigr)\ge (1-\delta)^{\ell+1},
$$
where $u\in\{\varPhi\le r\}$. Choosing $\delta>0$ sufficiently small, we obtain the required inequality~\eqref{8.7}. 
\end{proof}

\subsection{Exponential tightness}
\label{s8.1}
Recall that the occupation measures~$\zzeta_k$ were defined by~\eqref{1.016} and the concept of exponential tightness is introduced in Definition~\ref{d4.1}. Given a subset $\LLambda\subset H^\ell$, we shall say that $\{\zzeta_k\}$ is {\it exponentially tight uniformly in $\lambda\in\LLambda$\/} if for any $a>0$ there is a compact subset $\KK\subset H^\ell$ such that
$$
\limsup_{k\to\infty}\frac{1}{k}\log\sup_{\lambda\in\LLambda}\IP_\lambda\{\zzeta_k\in\KK^c\}\le -a.
$$

\begin{proposition} \label{p6.2}
For any $\delta>0$ and $M>0$, the family $\{\zzeta_k\}$ is exponentially tight uniformly with respect to $\lambda\in\LLambda(\delta,M)$.
\end{proposition}

\begin{proof}
In view of~\eqref{1.017} and Lemma~\ref{L:2.2}, it suffices to construct a function $\varPsi:H\to[0,+\infty]$ with compact level sets and positive numbers~$c$ and~$C$ such that\footnote{We used also Lemma~6.2 in~\cite{JNPS-2012}, according to which, for any random sequence, all the occupation measures corresponding to various initial times are exponentially equivalent, and therefore the exponential tightness for one of them implies the same property for all others.}
\begin{equation} \label{6.14}
\E_\lambda \exp\bigl(\varPsi(u_2)+\cdots+\varPsi(u_k)\bigr)\le Ce^{ck}
\quad\mbox{for $k\ge1$, $\lambda\in\Lambda(\delta,M)$},
\end{equation}
where $\{u_k\}$ stands for the trajectory of~\eqref{1.1}. We claim that inequality~\eqref{6.14} holds for $\varPsi(u)=\gamma\log(1+\|u\|_U)$, with a sufficiently small~$\gamma\in(0,1)$. Indeed, the function~$\varPsi$ is continuous on~$U$, and the embedding $U\subset H$ is compact. Hence, $\varPsi$~has compact level sets. In view of the Cauchy--Schwarz inequality and estimates~\eqref{1.4} (with $N=0$ and $v=0$) and~\eqref{1.08}, we have  
\begin{align}
\E_\lambda \exp\biggl(\,\sum_{n=2}^k\varPsi(u_n)\biggr)&=\E_\lambda\prod_{n=2}^k\bigl(1+\|u_n\|_U\bigr)^\gamma\notag\\
&\le\E_\lambda\prod_{n=1}^{k-1}\bigl(1+\|S(u_{n})\|_U\bigr)^\gamma\bigl(1+\|\eta_{n+1}\|_U\bigr)^\gamma\notag\\
&\le e^{c_1k} \,\E_\lambda\prod_{n=1}^{k-1}e^{2\gamma\pppp(u_{n})}\bigl(1+\|u_{n}\|\bigr)^{2\gamma}.
\label{6.015}
\end{align}
Here and henceforth, we denote by $c_i$ positive constants not  depending on~$k$ and~$\lambda$. Using again the Cauchy--Schwarz inequality, the stabilisability of~$\pppp$, and~\eqref{1.8}, for $4\gamma\le\min\{\alpha,\delta\}$ we obtain
\begin{equation*}
\E_\lambda \exp\biggl(\,\sum_{n=2}^k\varPsi(u_n)\biggr)
\le e^{c_2 k} \,\E_\lambda\prod_{n=1}^{k-1}\varPhi(u_n). 
\end{equation*}
The required inequality~\eqref{6.14} follows now from Lemma~\ref{l3.1}.
\end{proof}

\subsection{Growth condition}
\label{s8.2}
As was mentioned in Section~\ref{s1.5}, the proof of Theorem~\ref{t1.2} is based on an application of Theorem~\ref{T:3.1}. In this section, we introduce some auxiliary objects entering the formulation of that result and check the growth condition; see~\eqref{3.2}. 

For any integer $m\ge1$, we define a function $\wwww_m:H^\ell\to[1,+\infty)$ by the relation
\begin{equation} \label{6.21}
\wwww_m(\uuu)=\varPhi_m(u^1)+\cdots+\varPhi_m(u^\ell), \quad \uuu=[u^1,\dots,u^\ell].
\end{equation}
In what follows, to simplify notation, we shall often write $\wwww$ instead of~$\wwww_m$. Furthermore, for any integer $R\ge1$, we denote by~$X_R$ the direct product of~$\ell$ copies of~$B_U(R)$. Recall that, given $V\in C_b(H^\ell)$, the generalised Markov semigroup~$\PPPP_k^V$ is defined by~\eqref{6.3}. 

\begin{proposition} \label{p6.3}
For any $V\in C_b(H^\ell)$, the measures $P_{\ell+1}^V(\uuu,\cdot)$, $\uuu\in H^\ell$, are concentrated on~$U^\ell$, and there exist two integers $m_0\ge1$ and $R_0\ge1$ such that inequality~\eqref{3.2} holds with $\PPPP_k=\PPPP_k^V$, $\wwww=\wwww_m$, and $m\ge m_0$. 
\end{proposition}

Let us note that if~\eqref{3.2} holds for some integer~$m\ge1$, $\pppp:H\to\R_+$ is a continuous function bounded on any ball, and~$X$ is a Banach space satisfying the inclusions $U\subset X$ and $X\Subset H$, then the conclusion of the proposition remains true if we define~$X_R$ as the direct product of~$\ell$ copies of $\{u\in B_X(R):\pppp(u)\le R\}$. 

\begin{proof}[Proof of Proposition~\ref{p6.3}]
The fact that $P_{\ell+1}^V(\uuu,\cdot)$ is concentrated on~$U^\ell$ follows immediately from~\eqref{1.017} and~\eqref{6.14}. Replacing $V$ by $V-\inf_H V$, we can assume without loss of generality that $V\ge0$ and $\Osc(V)=\|V\|_\infty$. The proof of~\eqref{3.2} is divided into three steps. 

\smallskip
{\it Step 1}. 
We first show that it suffices to prove the inequality
\begin{equation} \label{6.15}
\sup_{k\ge0}
\frac{\|\PPPP_{k\ell}^V\wwww\|_{L_\wwww^\infty}}{\|\PPPP_{k\ell}^V{\mathbf1}\|_{R_0}}<\infty.
\end{equation}
Indeed, suppose that~\eqref{6.15} is established. Since $V\ge0$, we have 
$$
e^{-\gamma(p-j)}\PPPP_p^V{\mathbf1}(\uuu)\le\PPPP_j^V{\mathbf1}(\uuu)\le\PPPP_p^V{\mathbf1}(\uuu)\quad\mbox{for $j\le p$, $\uuu\in H^\ell$},
$$
where $\gamma=\|V\|_\infty$. It follows that 
\begin{equation} \label{6.19}
e^{-\gamma(p-j)}\|\PPPP_p^V{\mathbf1}\|_{R_0}\le\|\PPPP_j^V{\mathbf1}\|_{R_0}\le \|\PPPP_p^V{\mathbf1}\|_{R_0}
\quad\mbox{for $j\le p$}.
\end{equation}
On the other hand, using inequality~\eqref{3.02} with $\lambda=\delta_{u^\ell}$ and~$\varPhi$ replaced by~$\varPhi_m$, it is easy to check that
$$
\PPPP_1^V\wwww(\uuu)\le e^\gamma\wwww(\uuu)+C_1, \quad \uuu=[u^1,\dots,u^\ell]\in H^\ell,
$$
whence it follows that 
$$
\|\PPPP_1^Vf\|_{L^\infty_\wwww}\le (e^\gamma+C_1)\|f\|_{L^\infty_\wwww}\quad\mbox{for any $f\in L_\wwww^\infty$}. 
$$
The semigroup property now implies that
\begin{equation} \label{6.20}
\|\PPPP_{p}^V\wwww\|_{L_\wwww^\infty}\le (e^\gamma+C_1)^{p-j}\|\PPPP_{j}^V\wwww\|_{L_\wwww^\infty}
\quad\mbox{for $j\le p$}. 
\end{equation}
Combining inequalities~\eqref{6.19} and~\eqref{6.20} with $j=k\ell$ and $p=k\ell+n$, where $n\in[1,\ell-1]$ is an integer, we obtain the required inequality~\eqref{3.2}. 

\smallskip
{\it Step 2}. 
We now show that
\begin{equation} \label{6.0015}
\sup_{k\ge0}
\frac{\|\PPPP_{k}^V{\mathbf1}\|_{L_\wwww^\infty}}{\|\PPPP_{k}^V{\mathbf1}\|_{R_0}}<\infty,
\end{equation}
where $\wwww=\wwww_n$, and~$n$ and $R_0$ are some positive integers. 
Indeed, let us  find~$R_0$ and~$n$ such that inequality~\eqref{8.1} holds with $R=R_0$, $m=n$, and a constant~$C>0$. We now write
\begin{equation}
\PPPP_k^V{\mathbf1}(\uuu)=\E_\uuu\Xi_V(k)
=\E_\uuu \bigl(I_{G_k}\Xi_V(k)\bigr)+\E_\uuu\bigl(I_{G_k^c}\Xi_V(k)\bigr)=:I_1+I_2,
\label{6.16}
\end{equation}
where $G_k=\{\ttau_U(R_0)> k\}$ and $\Xi_V(k)=\exp(V(u_1)+\cdots+V(u_k))$. Since $V\ge0$, it follows from~\eqref{8.1} that $\PPPP_k^V{\mathbf1}(\uuu)\ge1$ for any $\uuu\in H^\ell$ and 
$$
I_1\le \E_\uuu \Xi_V\bigl(\ttau_U(R_0)\bigr)\le \E_\uuu\exp\bigl(\gamma\ttau_U(R_0)\bigr)\le C\,\varPhi_{n}(\uuu)
\le C\,\varPhi_{n}(\uuu)\,\|\PPPP_k^V{\mathbf1}\|_{R_0}. 
$$
Furthermore, in view of the strong Markov property, we have
$$
I_2\le \E_\uuu\bigl\{I_{G_k}\Xi_V(\ttau_U(R_0))\,\E_{\uuu(\ttau_U(R_0))}\Xi_V(k)\bigr\}
\le \E_\uuu\bigl(e^{\gamma\tau_U(R_0)}\bigr)\,\|\PPPP_k^V{\mathbf1}\|_{R_0},
$$
where we write $\uuu(\ttau_U(R_0))$ instead of $\uuu_{\ttau_U(R_0)}$ to avoid triple subscript. Using again~\eqref{8.1} and substituting these inequalities into~\eqref{6.16}, we obtain~\eqref{6.0015}.

\smallskip
{\it Step~3}. 
We claim that~\eqref{6.15} holds with $\wwww=\wwww_m$ and sufficiently large $m\ge n$. Indeed, let $\uuu_k=\Ss_k(\uuu;\eta_1,\dots,\eta_k)$ be the random variable defined by~\eqref{1.015}. Using~\eqref{8.4} and arguing by induction, it is straightforward to check that
\begin{align} 
\wwww\bigl(\Ss_\ell(\uuu;\eta_1,\dots,\eta_\ell)\bigr)
&\le \varkappa_{m,\ell}\,\varPhi_m(u^\ell)+C_{m,\ell}\bigl(\varPhi_m(\eta_1)+\cdots+\varPhi_m(\eta_\ell)\bigr)\notag\\
&\le \varkappa_{m,\ell}\,\wwww(\uuu)+C_{m,\ell}\,\wwww\bigl([\eta_1,\dots,\eta_\ell]\bigr), 
\label{6.17}
\end{align}
where $\uuu=[u^1,\dots,u^\ell]$, and we set
$$
\varkappa_{m,\ell}=\sum_{j=1}^{\ell}(2q^m)^j, \quad C_{m,\ell}=C_m\sum_{j=0}^{\ell-1}(2q^m)^j. 
$$
It follows from~\eqref{6.17} that
\begin{align}
\PPPP_{k\ell}^V\wwww(\uuu)
&=\E_\uuu \bigl\{\Xi_V(k\ell)\wwww\bigl(\Ss_\ell(\uuu_{(k-1)\ell};\eta_{(k-1)\ell+1},\dots,\eta_{k\ell})\bigr)\bigr\}
\notag\\
&\le\varkappa_{m,\ell}\,e^{\ell\gamma}\,\E_\uuu\bigl\{\Xi_V((k-1)\ell)\wwww(\uuu_{(k-1)\ell})\bigr\}\notag\\
&\qquad+C_{m,\ell}\,e^{\ell\gamma}\,\E_\uuu\bigl\{\Xi_V((k-1)\ell)\wwww([\eta_{(k-1)\ell+1},\dots,\eta_{k\ell}])\bigr\}\notag\\
&\le \varkappa_{m,\ell}\,e^{\ell\gamma}\,\PPPP_{(k-1)\ell}^V\wwww(\uuu)+C_{m,\ell}'\,\PPPP_{(k-1)\ell}^V{\mathbf1}(\uuu),
\label{6.18}
\end{align}
where we used the independence of~$\Xi_V((k-1)\ell)$ and $\{\eta_j,j>(k-1)\ell\}$ and set 
$$
C_{m,\ell}'=\ell\,C_{m,\ell}\,e^{\ell\gamma}\,\E\,\varPhi_m(\eta_1).
$$
Let us choose $m\ge n$ so large that $\varkappa:=\varkappa_{m,\ell}\,e^{\ell\gamma}<1$. Taking the $L^\infty_{\wwww}$-norm of both sides of~\eqref{6.18} and using inequality~\eqref{6.0015} (with is true for any $m\ge n$), we derive
$$
\|\PPPP_{k\ell}^V\wwww\|_{L^\infty_\wwww}\le \varkappa \|\PPPP_{(k-1)\ell}^V\wwww\|_{L^\infty_\wwww}
+C\,\|\PPPP_{(k-1)\ell}^V{\mathbf1}\|_{R_0},
$$
where $C>0$ does not depend on~$k$.
Iterating this inequality and using the relation $\|\PPPP_0^V\wwww\|_{L^\infty_\wwww}=1$ and the right-hand inequality in~\eqref{6.19}, we obtain 
$$
\|\PPPP_{k\ell}^V\wwww\|_{L^\infty_\wwww}\le\varkappa^{k }
+C(1-\varkappa)^{-1}\|\PPPP_{k\ell}^V{\mathbf1}\|_{R_0}. 
$$
Since $\PPPP_{k\ell}^V{\mathbf1}\ge{\mathbf1}$, this implies the required inequality~\eqref{6.15}. 
\end{proof}

\subsection{Existence of an eigenvector}
\label{s8.3}
In this section, as further preparation of application of Theorem~\ref{T:3.1}, we prove the existence of an eigenvector for the operator $\PPPP_1^{V*}:\MM_+(H^\ell)\to\MM_+(H^\ell)$ defined as the dual of~$\PPPP_1^V$:
\begin{equation} \label{6.101}
\langle f,\PPPP_k^{V*}\mu\rangle=\langle \PPPP_k^{V}f,\mu\rangle, \quad f\in C_b(H^\ell).
\end{equation}
To this end, we first introduce some notation. We define the kernel
\begin{equation} \label{6.024}
P_1^V(\uuu,\dd\vvv)=P_1^\ell(\uuu,\dd\vvv)e^{V(\vvv)}
\end{equation}
and denote by~$P_k^V(\uuu,\dd\vvv)$ its iterations. 
Given a number $s\in[0,1]$, let~$H_s$ be the (complex) interpolation space $[H,U]_s$, so that $H_0=H$, $H_1=U$, and the embedding $H_s\subset H$ is compact for $s\in(0,1]$; see~\cite{lunardi2009}. The norm in~$H_s$ is denoted by~$\|\cdot\|_s$, and we write~$B_s(R)$ for the ball in~$H_s$ of radius~$R$ centred at zero and~$B_s^\ell(R)$ for the direct product of~$\ell$ copies of~$B_s(R)$. Recall that the function $\wwww=\wwww_m$ was defined by~\eqref{6.21}.

\begin{proposition} \label{p6.4}
For any $V\in C_b(H^\ell)$ the operator~$\PPPP_1^{V*}$ has at least one eigenvector $\mu\in\PP(H^\ell)$ with a positive eigenvalue~$\lambda:$
\begin{equation} \label{6.22}
\PPPP_1^{V*}\mu=\lambda\mu.
\end{equation}
Moreover, any such eigenvector is concentrated on~$U^\ell$ and satisfies the following inequality for any integer $n\ge1$ and some positive numbers~$\varkappa$ and~$s=s_n\le1${\rm:}
\begin{equation} \label{6.23}
\int_{H^\ell}\biggl(\|\uuu\|_{s}^n+\sum_{j=1}^{\ell}e^{\varkappa\,\varPhi(u^j)}\biggr)\,\mu(\dd \uuu)<\infty.
\end{equation}
Finally, for any $k\ge0$, we have
\begin{equation} \label{6.24}
\|\PPPP_k^V\wwww\|_{L^\infty(B_s^\ell(R))}\int_{B_s^\ell(R)^c}\wwww(\uuu)\mu(\dd \uuu)\to0\quad\mbox{as $R\to\infty$}.
\end{equation}
\end{proposition}

\begin{proof}
{\it Step~1: A priori estimate\/}. 
We first prove that any eigenvector of~$\PPPP_1^{V*}$ is concentrated on~$U^\ell$ and satisfies~\eqref{6.23} (so that it belongs to~$\PP_{\wwww_m}(H^\ell)$ for any $m\ge1$). Indeed, let $\mu\in\PP(H^\ell)$ be a solution of~\eqref{6.22} with some $\lambda>0$. Then, for any integer $k\ge1$ and any non-negative function $f:H^\ell\to\R$, we have
\begin{align} 
\int_{H^\ell}f(\vvv)\,\mu(\dd\vvv)&=\lambda^{-k}\int_{H^\ell}\mu(\dd\vvv)\int_{H^\ell}P_k^V(\vvv,\dd\uuu)f(\uuu)
\notag\\
&\le \lambda^{-k}e^{k\|V\|_\infty}\int_{H^\ell}\E\,f\bigl(\Ss_k(\vvv;\eta_1,\dots\eta_k)\bigr)\mu(\dd\vvv)
\notag\\
&=\lambda^{-k}e^{k\|V\|_\infty}\int_{H^\ell}\mu(\dd\vvv)\int_{H^\ell}P_k^\ell(\vvv,\dd\uuu)f(\uuu). 
\label{6.32}
\end{align}
Taking $k=1$ and $f(\vvv)=g(v^j)$ with $j=1,\dots,\ell-1$, where $g:\R\to\R$ is any non-negative function, and using~\eqref{1.12}, we see that 
$$
\int_{H^\ell}g(v^j)\,\mu(\dd\vvv)\le \lambda^{-1} e^{\|V\|_\infty}\int_{H^\ell}g(v^{j+1})\,\mu(\dd\vvv).
$$
Iterating this inequality, for $j=1,\dots,\ell-1$ we obtain
$$ 
\int_Hg(v)\mu_j(\dd v)=\int_{H^\ell}g(v^j)\,\mu(\dd\vvv)\le \lambda^{j-\ell}e^{(\ell-j)\|V\|_\infty}\int_Hg(v)\mu_\ell(\dd v),
$$ 
where~$\mu_j$ stands for the~$j^{\mathrm{th}}$ marginal of~$\mu$. 
In particular, to prove~\eqref{6.23}, it suffices to show that
\begin{equation} \label{6.34}
\int_H\bigl(\|z\|_s^n+e^{\varkappa\varPhi(z)}\bigr)\mu_\ell(\dd z)<\infty. 
\end{equation}
Taking again $k=1$ and $f(\vvv)=g(v^\ell)$ in~\eqref{6.32}, we see that
\begin{equation} \label{6.35}
\langle g,\mu_\ell\rangle\le \lambda^{-1}e^{\|V\|_\infty}\int_H\E\,g\bigl(S(z)+\eta_1\bigr)\mu_\ell(\dd z). 
\end{equation}
For $g(z)=e^{\varkappa\varPhi(z)}$, using~\eqref{1.04} and~\eqref{1.8}, we obtain
$$
\langle e^{\varkappa\varPhi},\mu_\ell\rangle 
\le C_1\langle e^{\varkappa\varPhi},\mu_\ell\rangle^q,
$$
where $C_1=\lambda^{-1}e^{\|V\|_\infty} \E\,e^{C\varkappa\varPhi(\eta_1)}$. Hence, taking $\varkappa>0$ so small that $C\varkappa\le\delta$, we obtain\footnote{The derivation of~\eqref{6.36} is formal. To obtain an accurate justification, it suffices to apply the above argument to bounded approximations of~$e^{\varkappa\varPhi}$ and then pass to the limit with the help of the Fatou lemma.}
\begin{equation} \label{6.36}
\langle e^{\varkappa\varPhi},\mu_\ell\rangle=\int_He^{\varkappa\varPhi(z)}\mu_\ell(\dd z)\le C_2. 
\end{equation}
We have thus proved that the integral of the second term in~\eqref{6.34} is finite. To derive a bound for the integral of the first term, we use~\eqref{1.7} and~\eqref{1.09} to write
\begin{equation} \label{6.25}
\E_{\mu}\exp\bigl\{\delta\bigr(\pppp(u_\ell^1)+\cdots+\pppp(u_\ell^\ell)\bigl)\bigr\}
\le e^{c\ell}\int_He^{\rho\,\varPhi(z)}\mu_\ell(\dd z),
\end{equation}
where $\uuu_k=[u_k^1,\dots,u_k^\ell]$ is the trajectory of~\eqref{1.015}. Decreasing, if necessary, the number $\delta>0$, we can assume that~$\rho$ is no larger than the constant~$\varkappa$ in~\eqref{6.36}. We see that the left-hand side of~\eqref{6.25} is finite. On the other hand, in view of the equality in~\eqref{6.32} with $k=\ell$ and $f(\vvv)=\exp\{\delta(\pppp(v^1)+\cdots+\pppp(v^\ell))\}$, we have
\begin{equation} \label{6.26}
C_3:=\int_{H^\ell}\exp\bigl\{\delta\bigl(\pppp(v^1)+\cdots+\pppp(v^\ell)\bigr)\bigr\}\mu(\dd \vvv)<\infty. 
\end{equation}
It follows from~\eqref{1.03} and~\eqref{1.04} that
\begin{equation} \label{6.29}
\|S(v)\|\le C_4(\varPhi(v)^{1/\alpha}+1), \quad v\in H.
\end{equation}
Furthermore, using~\eqref{1.4} and the inclusion $S(0)\in U$, we see that
\begin{equation} \label{6.27}
\|S(v)\|_U\le e^{\pppp(v)}\|v\|+C_5.
\end{equation}
Combining~\eqref{6.27} and~\eqref{6.29}, using the right-hand inequality in~\eqref{1.03}, and interpolating between~$H$ and~$U$, we derive
\begin{equation} \label{6.30}
\|S(v)\|_s^n\le C_6\bigl(\varPhi_m(v)+e^{2sn\pppp(v)}+1\bigr),
\end{equation}
where $m\ge2n/\alpha$ is an integer. Using again the interpolation inequality, we see that
\begin{equation} \label{6.31}
\|S(v)+\eta_1\|_s^n\le C_7\bigl(\varPhi_m(v)+e^{2sn\pppp(v)}+\|\eta_1\|_U^{2sn}+\|\eta_1\|^{2n}+1\bigr).
\end{equation}
Taking the mean value, integrating over $\mu_\ell(\dd v)$, and recalling~\eqref{6.35}, we obtain
\begin{align*}
\int_H\|u\|_s^n\mu_\ell(\dd u)&\le\lambda^{-1}e^{\|V\|_\infty}
\int_H\E\,\|S(v)+\eta_1\|_s^n\mu_\ell(\dd v)\\
&\le C_8\bigl(\langle \varPhi_m+e^{2sn\pppp},\mu_\ell\rangle+\E\,\|\eta_1\|_U^{2sn}+\E\,\|\eta_1\|^{2n}+1\bigr)<\infty,
\end{align*}
where we chose~$s\in(0,1]$ so small that $2sn\le\min\{1,\delta\}$ and used inequalities~\eqref{1.8}, \eqref{1.08}, \eqref{6.36}, and~\eqref{6.26}. This completes the proof of~\eqref{6.34} and that of~\eqref{6.23}. The latter estimate also implies that $\mu\in\PP_{\wwww_m}(H^\ell)$ for any $m\ge1$. 

It remains to prove that $\mu(U^\ell)=1$. In view of~\eqref{1.12} and \eqref{1.08}, it suffices to check that $S(v)\in U$ almost surely for any random variable~$v$ whose law coincides with~$\mu_\ell$. It follows from~\eqref{6.27} and~\eqref{6.26} that, for any $R>0$, 
$$
\E\,\bigl(I_{B_H(R)}\|S(v)\|_U^{\delta}\bigr)\le R^\delta\, \E\,e^{\delta\pppp(v)}+C_5^\delta\le R^\delta C_3+C_5^\delta.
$$
Since $R>0$ is arbitrary, this inequality proves that $S(v)\in U$ almost surely.

\smallskip
{\it Step~2: Decay}. 
Let us show that~\eqref{6.23} implies~\eqref{6.24}. Indeed, it follows from~\eqref{1.017} and~\eqref{3.02} (with~$\varPhi$ replaced by~$\varPhi_m$) that 
\begin{equation} \label{6.036}
\|\PPPP_k^V\wwww\|_{L^\infty(B_s^\ell(R))}\le e^{k\|V\|_\infty}\sup_{\uuu\in B_s^\ell(R)}\E_\uuu\wwww(\uuu_k)
\le C_9(k)\,\sup_{u\in B_H(R)}\varPhi_m(u).
\end{equation}
On the other hand, we have $(B_s^\ell(R))^c\subset\cup_{j=1}^\ell\Gamma_j(R)$, where 
$$
\Gamma_j(R)=\{[u^1,\dots,u^\ell]\in H^\ell:u^j\notin B_s(R)\}.
$$ 
Using the Cauchy--Schwarz inequality, we derive
\begin{align} 
\int_{B_s^\ell(R)^c}\wwww(\uuu)\mu(\dd \uuu)
&\le\sum_{j=1}^\ell\int_{\Gamma_j(R)}\wwww(\uuu)\,\mu(\dd\uuu)\notag\\
&\le \langle\wwww^2,\mu\rangle^{1/2} \sum_{j=1}^\ell \mu\bigl(\Gamma_j(R)\bigr)^{1/2}. 
\label{6.40}
\end{align}
It follows from~\eqref{6.23} and the Chebyshev inequality that $ \langle\wwww^2,\mu\rangle<\infty$ and 
$$
\mu\bigl(\Gamma_j(R)\bigr)=\mu_j\bigl(B_s(R)^c\bigr)\le C_{10}(n)R^{-n}. 
$$
Combining this with~\eqref{6.40}, \eqref{6.036}, and the right-hand inequality in~\eqref{1.03} and choosing $n>m\beta$, we obtain~\eqref{6.24}.

\smallskip
{\it Step~3: Construction of eigenvector}. 
To construct a measure $\mu\in\PP(H^\ell)$ satisfying~\eqref{6.22}, we first remark that it suffices to construct an eigenvector for~$\PPPP_\ell^{V*}$. Indeed, suppose that $\mu'\in\PP_\wwww(H^\ell)$ is an eigenvector for~$\PPPP_\ell^{V*}$ with an eigenvalue $\lambda'>0$. Then
$$
(\PPPP_1^{V*}-\lambda)\mu=0, \quad 
\mu=c\,\bigl(\PPPP_{\ell-1}^{V*}+\lambda\PPPP_{\ell-2}^{V*}+\cdots+\lambda^{\ell-1}I\bigr)\mu',
$$
where $\lambda=(\lambda')^{1/\ell}$, and~$c>0$ is a normalising constant chosen so that $\mu(H^\ell)=1$. 

To construct an eigenvector for~$\PPPP_\ell^{V*}$, we use the Leray--Schauder fixed point theorem; e.g., see Chapter~14 in~\cite{taylor1996}. Let us fix $s\in(0,1]$, $\varkappa>0$, and $n\ge1$, and given~$A>0$, define the convex set 
$$
D_{A,m}=\bigl\{\nu\in\PP(H^\ell): \langle\wwww_m,\nu\rangle\le A\bigr\}.
$$
A simple application of the Fatou lemma shows that~$D_{A,m}$ is closed in~$\PP(H^\ell)$. Let us define a continuous mapping $G:D_{A,m}\to \PP(H^\ell)$ by the relation
$$
G(\nu)=\frac{\PPPP_\ell^{V*}\nu}{\PPPP_\ell^{V*}\nu(H^\ell)}. 
$$
We claim that $G(D_{A,m})\subset D_{A,m}$ for an appropriate choice of~$A$ and~$m$, and that~$G(D_{A,m})$ is relatively compact in~$\PP(H^\ell)$. Once this is proved, the existence of a measure $\mu\in D_{A,m}$ satisfying~\eqref{6.22} will follow from the Leray--Schauder theorem. 

It follows from~\eqref{6.17} and~\eqref{1.8} that 
\begin{align}
\bigl\langle\wwww_m,G(\nu)\bigr\rangle
&\le \exp\{\ell\Osc_H(V)\}\int_{H^\ell}\E\,\wwww_m\bigl(\Ss_\ell(\uuu;\eta_1,\dots,\eta_\ell)\bigr)\nu(\dd\uuu)
\notag\\
&\le \varkappa_{m,\ell}\exp\{\ell\Osc_H(V)\}\langle\wwww_m,\nu\rangle+C_{11}. 
\end{align}
Since $\varkappa_{m,\ell}\to0$ as $m\to\infty$, we conclude that $G(D_{A,m})\subset D_{A,m}$ if $A\ge 2C_{11}$ and $m\ge1$ is sufficiently large. In view of the Prokhorov compactness criterion (see Theorem~11.5.4 in~\cite{dudley2002}), to prove that~$G(D_{A,m})$ is relatively compact it suffices to check that
\begin{equation} \label{6.38}
\int_{H^\ell}\|\uuu\|_s \nu'(\dd\uuu)\le C_{12}\quad
\mbox{for any $\nu'\in G(D_{A,m})$},
\end{equation}
where $s>0$. It follows from~\eqref{1.017} and the boundedness of~$V$ that inequality~\eqref{6.38} will be established if we prove that, for sufficiently large $m\ge1$ and $A>0$, 
\begin{equation} \label{6.39}
\int_{H}\|u\|_s (\PPPP_k^*\nu)(\dd u)\le C_{13}\quad\mbox{for $\nu\in \PP(H)$, $\langle\varPhi_m,\nu\rangle\le A$},
\end{equation}
where $1\le k\le\ell$.  Inequality~\eqref{3.02} with~$\varPhi$ replaced by~$\varPhi_m$ implies that if $\nu\in \PP(H)$ is such that $\langle\varPhi_m,\nu\rangle\le A$ with $A\gg1$, then $\langle\varPhi_m,\PPPP_k^*\nu\rangle\le A$ for $k\ge1$. Thus, it suffices to prove~\eqref{6.39} for $k=1$. To this end, it suffices to take the mean value of both sides of~\eqref{6.31} and to integrate in $\nu(\dd v)$. The proof of Proposition~\ref{p6.4} is complete. 
\end{proof}

\subsection{Multiplicative  ergodic theorem}
\label{s8.5}
In this section, we apply Theorem~\ref{T:3.1} to obtain a detailed description of the large-time behaviour of the semigroups~$\{\PPPP_k^{V}\}$ and~$\{\PPPP_k^{V*}\}$. This results will be used in the next two subsections to establish Properties~1 and~2 of Section~\ref{s1.5} needed to prove Theorem~\ref{t1.2}. 

We first introduce some notation. Recall that the orthonormal basis~$\{e_j\}$ was defined in Condition~(C) (see Section~\ref{s1.1}). We denote by~$\VV$ the set of functions $V\in C_b(H^\ell)$ that can be represented in the form 
\begin{equation} \label{6.51}
V(\uuu)=F({\mathsf P}_N\uuu), \quad {\mathsf P}_N\uuu=[{\mathsf P}_Nu^1,\dots,{\mathsf P}_Nu^\ell],
\end{equation}
where $N\ge1$ is an integer and $F\in C_b^1(H_N^\ell)$. It is straightforward to check that~$\VV$ is a vector space containing constant functions such that the intersection $\CC=\VV\cap C_+(H^\ell)$ is a determining family for~$\PP(H^\ell)$. 

\begin{theorem} \label{t6.5}
Let us assume that the hypotheses of Theorem~\ref{t1.2} are satisfied. Then, for any $V\in\VV$, there is an integer $m_0=m_0(V)\ge1$ such that the following assertions hold for $m\ge m_0$:
\begin{description}
\item[Existence and uniqueness.]
The measure $\mu=\mu_V$ constructed in Proposition~\ref{p6.4} is the only eigenvector of~$\PPPP_1^{V*}$ belonging to~$\PP_{\wwww_m}(H^\ell)$. Moreover, the operator~$\PPPP_1^V$ has a unique eigenvector~$h_V$ in $C_{\wwww_m}(H^\ell)\cap C_+(H^\ell)$ normalised by the condition $\langle h_V,\mu_V\rangle=1$.
\item[Convergence.] 
For any $f\in C_{\wwww_m}(H^\ell)$, $\nu\in \PP_{\wwww_m}(H^\ell)$, and $R>0$, we have 
\begin{align}
\lambda_V^{-k}\PPPP_k^V f&\to\lag f,\mu_V\rag h_V
\quad\mbox{in~$C_b(B^\ell(R))\cap L^1(H^\ell,\mu_V)$ as~$k\to\infty$}, \label{5.24}\\
\lambda^{-k}_V\PPPP_k^{V*}\nu&\to\lag h_V,\nu\rag\mu_V
 \quad\mbox{in~$\MM_+(H^\ell)$ as~$k\to\infty$}. \label{5.25}
\end{align}
\end{description}
\end{theorem}

\begin{proof}
We shall show that the hypotheses of Theorem~\ref{T:3.1} are satisfied for an appropriate choice of compact sets~$X_R\subset H^\ell$ and the function $\wwww=\wwww_m$ defined in~\eqref{6.21}. The conclusions of that theorem, combined with some simple arguments, will imply all required results. 

\medskip
{\it Step~1: Framework and sketch\/}.
Let us fix a function $V\in\VV$. We apply Theorem~\ref{T:3.1} in which
\begin{equation} \label{6.44}
X=H^\ell, \quad X_R=\bigl(B_s(R)\cap\{\pppp\le R\}\bigl)^\ell,  \quad P(\uuu,\dd \vvv)=P_1^\ell(\uuu,\dd\vvv)e^{V(\vvv)},
\end{equation}
 where $s\in(0,1]$ will be chosen later and $\{\pppp\le R\}=\{u\in H:\pppp(u)\le R\}$. In particular, the semigroups~$\PPPP_k^V$ and~$\PPPP_k^{V*}$ are given by~\eqref{6.3} and~\eqref{6.101}, respectively. The subset $X_\infty=\cup_{R\ge 1} X_R$ coincides with~$(H_s)^\ell\supset U^\ell$, which is dense in~$X$. The weight function~$\wwww:X\to\R_+$ is defined by~\eqref{6.21} and is continuous on~$X$.  As is proved in Proposition~\ref{p6.3}, if the integers~$m\ge1$ and~$R_0\ge1$ are sufficiently large, then the growth condition~\eqref{3.2} is satisfied, and Proposition~\ref{p6.4} ensures the existence of an eigenvector $\mu\in\PP_\wwww(H^\ell)$ for~$\PPPP_1^{V*}$, which satisfies~\eqref{6.23} and~\eqref{6.24}. Thus, to apply Theorem~\ref{T:3.1}, we need to check the uniform Feller and irreducibility properties (see Step~2 below). Once it is done, we can conclude that~$\PPPP_1^V$ has an eigenvector $h_V\in L_\wwww^\infty(X)$ whose restriction to~$X_R$ is continuous and strictly positive for any $R\ge1$. We show in Step~3 that~$h_V$ is continuous and positive on~$X$ and, therefore, belongs to~$C_\wwww(X)$. The uniqueness of eigenvectors for~$\PPPP_1^{V}$ and~$\PPPP_1^{V*}$ in the spaces $C_\wwww(X)\cap C_+(X)$ and~$\PP_\wwww(X)$, respectively, follow immediately from convergences~\eqref{5.24} and~\eqref{5.25}, of which the second is a straightforward consequence of the first. The proof of~\eqref{5.24} is carried out in Steps~4 and~5.

\smallskip
{\it Step~2: Uniform Feller and irreducibility properties\/}. 
The uniform Feller property is the crucial step of the proof and will be established in Section~\ref{s7} with the help of a coupling construction. Let us prove the uniform irreducibility. Since~$V$ is bounded, we have
$$
P_k^V(\uuu,\dd\vvv)\ge e^{-k\|V\|_\infty}P_k^\ell(\uuu,\dd\vvv)\quad\mbox{for any $\uuu\in X$}.
$$
Thus, in view of continuity of~$\pppp$, it suffices to prove that, for any integers~$\rho,R\ge1$ and any $r>0$, there exists an integer $l\ge1$ and a  positive number $p=p(\rho,r)$ satisfying the inequalities
\begin{equation} \label{6.46}
\IP_\uuu\bigl\{\|\uuu_l-\hat\uuu\|_s\le r,\uuu_l\in X_\rho\bigr\}\ge p\quad
\mbox{for any $\uuu\in X_R$, $\hat\uuu\in X_\rho$}. 
\end{equation}
This type of result is  well known in the theory of randomly forced PDE's (e.g., see Proposition~4.6 in~\cite{KS-mpag2001}), and we only outline its proof. 

In view of~\eqref{1.017}, inequality~\eqref{6.46} will be established if we show that 
\begin{equation} \label{6.47}
\IP_u\bigl\{\|u_{l+j}-\hat u_j\|_s\le r, u_{l+j}\in B_s(\rho)
\mbox{ for $j=1,\dots,\ell$}\bigr\}\ge p,
\end{equation}
where $\{u_k\}$ is the trajectory of~\eqref{1.1},  $u\in B_H(R)$ and $\hat u_1,\dots\hat u_\ell\in B_s(\rho)$ are arbitrary vectors. It follows from~\eqref{3.02} that, for sufficiently large $\widehat R>0$ and $\widehat m(R)\ge1$, the probability of transition, at time $\widehat m(R)$, from any point $u\in B_H(R)$ to the ball $B_H(\widehat R)$ is no less than~$1/2$. Therefore, in view of the Kolmogorov--Chapman relation, it suffices to prove~\eqref{6.47} for $u\in B_H(\widehat R)$ and $l=1$. The argument used in the derivation of~\eqref{6.14} (see~\eqref{6.015} with $k=\ell+1$) shows that, for a sufficiently large $K>0$, we have
\begin{equation} \label{6.49}
\IP_u\bigl\{\|S(u_{j})\|_U\le K \mbox{ for $j=1,\dots,\ell$}\bigr\}\ge \frac12. 
\end{equation}
Furthermore, since the law  of~$\eta_k$ is concentrated on~$U$, it follows from~\eqref{1.08} that
$$
\IP\{\|w+\eta_1-\hat v\|_s\le r\}\ge \e_r\quad\mbox{for $w\in B_U(K)$, $\hat v\in B_s(\rho)$},
$$
where $s\in(0,1)$ is any number, and $\e_r>0$ depends only on~$s$, $K$, and~$\rho$. Combining this with~\eqref{6.49} and the Markov property, we obtain~\eqref{6.47} with $l=1$ and $p_1=\e_r^\ell$. 

\smallskip
{\it Step~3: Positivity  and continuity of~$h_V$\/}. 
Since~$h_V$ is an eigenvector of~$\PPPP_1^V$ with the eignevalue~$\lambda_V$, we have
\begin{align}
h_V(\uuu)
&=\lambda_V^{-\ell-1}\E_\uuu\bigl\{\exp\bigl(V(\uuu_1)+\cdots+V(\uuu_{\ell+1})\bigr)h_V(\uuu_{\ell+1})\bigr\} \notag\\
&\ge\lambda_V^{-\ell-1}e^{-(\ell+1)\|V\|_\infty}\,\E_\uuu h_V(\uuu_{\ell+1}).
\label{6.50}
\end{align}
In view of~\eqref{6.14}, for any $\uuu\in X$, the vector function $\uuu_{\ell+1}$ belongs to~$U^\ell$ with $\IP_\uuu$-probability~$1$. Since~$h_V$ is positive on~$X_\infty=(H_s)^\ell$, we see that the expectation on the right-hand side of~\eqref{6.50} is positive, whence it follows that $h_V(\uuu)>0$ for any $\uuu\in X$.

To prove the continuity of~$h_V$, let us set $\vec\eta_k=(\eta_1,\dots,\eta_k)$ and denote by $f(\uuu,\vec{\eta}_{\ell+1})$ the expression under the expectation~$\E_\uuu$ in the first line of~\eqref{6.50}. Given two initial points $\uuu,\vvv\in X$ and a number~$R>0$, we introduce the event
$$
G_R(\uuu,\vvv)=\{\pppp(u_\ell^1)+\cdots+\pppp(u_\ell^\ell)+\pppp(v_\ell^1)+\cdots+\pppp(v_\ell^\ell)\le R\},
$$
where $\uuu_k$ and $\vvv_k$ are the trajectories of~\eqref{1.015} corresponding to the initial points~$\uuu$ and~$\vvv$. It follows from~\eqref{1.7} that when~$\uuu$ and~$\vvv$ vary in a bounded set in~$X$, the probability of~$(G_R)^c$ goes to zero as $R\to\infty$. Now note that, in view of the first line in~\eqref{6.50},
\begin{multline} \label{6.051}
|h_V(\uuu)-h_V(\vvv)|
\le \lambda_V^{-\ell-1}\E \bigr\{I_{G_R(\uuu,\vvv)}\bigl|f(\uuu,\vec{\eta}_{\ell+1})-f(\vvv,\vec{\eta}_{\ell+1})\bigr|\bigl\}\\
+C_1\IP(G_R(\uuu,\vvv)^c)^{1/2}\bigl(\PPPP_{\ell+1}h_V^2(\uuu)+\PPPP_{\ell+1}h_V^2(\vvv)\bigr)^{1/2},
\end{multline}
where $C_1$ depends only on~$\|V\|_\infty$. 
Since $h_V^2\in L_{\wwww_{2m}}^\infty$, it follows from inequality~\eqref{6.17} (with~$\ell$ and~$m$ replaced by~$\ell+1$ and $2m$, respectively) that the second term on the right-hand side of~\eqref{6.051} goes to zero as $R\to\infty$, uniformly with respect to~$\uuu$ and~$\vvv$ varying in a bounded set in~$X$. Since~$V$ is bounded and $h_V\in L_\wwww^\infty(X)$, it follows from inequality~\eqref{6.17} with~$\ell$ replaced by~$\ell+1$ that 
$$
|f(\uuu;\vec\eta_{\ell+1})-f(\vvv;\vec\eta_{\ell+1})|
\le C_2\bigl(\wwww(\uuu)+\wwww(\vvv)+\varPhi_m(\eta_1)+\cdots+\varPhi_m(\eta_{\ell+1})\bigr). 
$$
The right-hand side of this inequality an integrable function, and by the Lebesgue theorem on dominated convergence, the continuity of~$h_V$ will be established if we prove that, with probability~$1$, 
$$
I_{G_R(\uuu,\vvv)}|f(\uuu;\vec\eta_{\ell+1})-f(\vvv;\vec\eta_{\ell+1})|\to0\quad\mbox{as $\uuu\to\vvv$ in~$X$}. 
$$

To this end, recall that~$V$ is continuous on~$X$ and that the restriction of~$h_V$ to~$X_\infty$ is also continuous. Thus, it suffices to prove 
$$
I_{G_R(\uuu,\vvv)}\|\Ss_{\ell+1}(\uuu,\vec{\eta}_{\ell+1})-\Ss_{\ell+1}(\vvv,\vec{\eta}_{\ell+1})\|_s\to0
\quad\mbox{as $\uuu\to \vvv$}. 
$$
In view of~\eqref{1.017}, this is equivalent to the almost sure convergence 
\begin{equation} \label{6.52}
I_{G_R(\uuu,\vvv)}\sum_{k=2}^{\ell+1}\|S_k(u^\ell,\vec{\eta}_k)-S_k(v^\ell,\vec{\eta}_k)\|_s\to0
\quad\mbox{as $\uuu\to \vvv$},
\end{equation}
where~$S_k$ stands for the trajectory of~\eqref{1.1} at time~$k$.
To prove~\eqref{6.52}, let us note that, in view of~\eqref{1.4}, on the set $G_R(\uuu,\vvv)$ we have
$$
\|S_k(u^\ell,\vec{\eta}_{\ell+1})-S_k(v^\ell,\vec{\eta}_{\ell+1})\|_U
\le e^R\|S_{k-1}(u^\ell,\vec{\eta}_{k-1})-S_{k-1}(v^\ell,\vec{\eta}_{k-1})\|. 
$$
The continuity of~$S:H\to H$ implies that the mapping $u\mapsto S_k(u,\vec\eta_k)$ is also continuous on~$H$, and the last inequality proves the required convergence~\eqref{6.52}. 

\smallskip
{\it Step~4: Proof of~\eqref{5.24} for bounded continuous functions\/}. 
In view of~\eqref{3.7}, for any $f\in \VV$ we have
\begin{equation} \label{6.53}
\lambda_V^{-k}\PPPP_k^Vf\to\langle f,\mu_V\rangle h_V\quad \mbox{as $k\to\infty$},
\end{equation}
where the convergence holds in $C(X_R)\cap L^1(X,\mu_V)$. 
We claim that it holds also in $C(B^\ell(R))$. Indeed, in view of~\eqref{3.9}, it suffices to check condition~\eqref{3.8} with $B=B^\ell(R)$ and $m=\ell+1$. The latter is a consequence of the boundedness of~$V$, inequality~\eqref{6.20} with $j=0$, and the following lemma, which is established at the end of this subsection.\,\footnote{Recall that the set $\LLambda(\delta,M)$ is defined in Section~\ref{s1.5}.}

\begin{lemma} \label{l6.6}
Under the hypotheses of Theorem~\ref{t1.2}, for any integer $n\ge1$ there is $s=s_n\in(0,1]$ such that
\begin{equation} \label{6.55}
\sup_{k\ge\ell+1} \E_\lambda\,\|\uuu_k\|_s^n\le C(n,\delta,M)\quad\mbox{for $\lambda\in\LLambda(\delta,M)$},
\end{equation}
where $\delta$ and $M$ are arbitrary positive numbers. 
\end{lemma}

We have thus established~\eqref{5.24} for $f\in\VV$. To prove~\eqref{5.24} for any function $f\in C_b(X)$, we apply a simple approximation argument. Namely, the convergence in~$L^1(X,\mu_V)$ is a straightforward consequence of convergence in~$C(B^\ell(R))$ and inequalities~\eqref{6.23} and~\eqref{3.2}. To prove the convergence in the $L^\infty$ norm, we fix a number~$R>0$ and a function $f\in C_b(X)$ and choose a sequence of $C^1$-functions $\tilde f_n:H_n^\ell\to\R$ such that 
$$
\sup_{\vvv\in H_n^\ell,\|\vvv\|\le n}|\tilde f_n(\vvv)-f(\vvv)|\le\frac1n\quad
\mbox{for any $n\ge1$}. 
$$
Then the functions $f_n={\tilde f}_n\circ{\mathsf P}_n$ belong to the space~$\VV$, satisfy the inequality $\|f_n\|_\infty\le \|f\|_\infty$ for any $n\ge1$, and $f_n\to f$ as $n\to\infty$, uniformly on compact subsets of~$X$. Setting 
$$
\Delta_k(g)=\sup_{\uuu\in B^\ell(R)}\bigl|\lambda_V^{-k}\PPPP_k^Vg(\uuu)-\langle g,\mu_V\rangle h_V(\uuu)\bigr|,
\quad \|g\|_{L^\infty_R}=\sup_{\uuu\in B^\ell(R)}|g(\uuu)|,
$$
for any integers $k,n\ge1$, we write
\begin{equation} \label{6.56}
\Delta_k(f)\le \Delta_k(f_n)+\|h_V\|_{L_R^\infty}\,|\langle f-f_n,\mu_V\rangle|+\lambda_V^{-k}\|\PPPP_k^V(f-f_n)\|_{L_R^\infty}. 
\end{equation}
Since $f_n\in\VV$, for any fixed $n\ge1$ the first term on the right-hand side of this inequality goes to zero as $k\to\infty$. The Lebesgue theorem on dominated convergence implies that $|\langle f-f_n,\mu_V\rangle|\to0$ as $n\to\infty$. Thus, the required convergence will be established if we prove that
\begin{equation} \label{6.57}
\sup_{k\ge 1}\lambda_V^{-k}\|\PPPP_k^V(f-f_n)\|_{L_R^\infty}\to0\quad\mbox{as $n\to\infty$}.
\end{equation}
 
To prove~\eqref{6.57}, for any $\rho>0$ we write
\begin{equation} \label{6.58}
\|\PPPP_k^V(f-f_n)\|_{L_R^\infty}\le J_1(k,n,\rho)+J_2(k,n,\rho),
\end{equation}
where 
$$
J_1(k,n,\rho)=\|\PPPP_k^V\bigl((f-f_n)I_{X_\rho}\bigr)\bigr\|_{L_R^\infty}, \quad 
J_2(k,n,\rho)=\|\PPPP_k^V\bigl((f-f_n)I_{X_\rho^c}\bigr)\|_{L_R^\infty}. 
$$
Since $f_n\to f$ uniformly on $X_\rho$, we have 
$$
J_1(k,n,\rho)\le \e(n,\rho)\,\|\PPPP_k^V\mathbf1\|_{L^\infty_R},
$$
where $\e(n,\rho)\to0$ as $n\to\infty$. Convergence~\eqref{5.24} with $f=\mathbf1$ implies that the $L_R^\infty$-norm of the sequence $\lambda_V^{-k}\PPPP_k^V{\mathbf1}$ is bounded by a constant~$C_R$, whence it follows that
\begin{equation} \label{6.59}
\sup_{k\ge 1}\lambda_V^{-k}J_1(k,n,\rho)\le C_R\,\e(n,\rho)\to0\quad\mbox{as $n\to\infty$}. 
\end{equation}
To estimate $J_2$, we use the following result, proved at the end of this subsection. 

\begin{lemma} \label{l6.7} 
Assume that  the hypotheses of Theorem~\ref{t1.2} are satisfied.  Then, for any $s\in(0,1]$, we have
$$
\sup_{k\ge0}
\frac{\|\PPPP_k^V F_s \|_{L_\wwww^\infty}}{\|\PPPP_k^V{\mathbf1}\|_{R_0}}<\infty,
$$ where $F_s(u)=\log(1+\|u\|_s)$.
\end{lemma}

Using this lemma, for any $k,\rho,n \ge1$ we get
\begin{align*}
\la_V^{-k}J_2(k,n,\rho)
&\le 2\|f\|_\infty \bigl(\log(1+\rho)\bigr)^{-1} \la_V^{-k} \|\PPPP_k^V F_s \|_{L_R^\infty}\\
&\le C_R\|f\|_\infty \bigl(\log(1+\rho)\bigr)^{-1} \la_V^{-k}\|\PPPP_k^V{\mathbf1}\|_{R_0}.  
\end{align*}
Since $\la_V^{-k}\|\PPPP_k^V{\mathbf1}\|_{R_0}$ is bounded,   the right-hand side of this inequality goes to zero as $\rho\to+\infty$, uniformly with respect to $k\ge1$. Combining this with~\eqref{6.59}, we see that supremum over $k\ge 1$ of the right-hand side of~\eqref{6.58} can be made arbitrarily small by choosing first $\rho>0$ and then~$n\ge1$ sufficiently large. This proves the required convergence~\eqref{6.57}. 

\smallskip
{\it Step~5: Proof of~\eqref{5.24} for $f\in C_\wwww(X)$\/}. 
We apply an approximation argument similar to the one used in the previous step. Let us fix a function $f\in C_{\wwww}(X)$ and define a sequence~$\{f_n\}$ by the relation $f_n=f^+\wedge n-f^-\wedge n$. Then $f_n\in C_b(X)$ and $|f_n|\le |f|$ for any $n\ge1$ and $f_n\to f$ in $L_{\wwww_p}^\infty(X)$ for any $p>m$. 
Now note that inequality~\eqref{6.56} remains valid. Furthermore, in view of~\eqref{5.24} and the Lebesgue theorem on dominated convergence, we have
\begin{gather*}
\Delta_k(f_n)\to0\quad\mbox{as $k\to\infty$ for any fixed $n\ge1$},\\
|\langle f-f_n,\mu_V\rangle|\to0\quad\mbox{as $n\to\infty$}. 
\end{gather*}
To prove that~\eqref{6.57} holds, we set  $\wwww'=\wwww_{m+1}$ and write
\begin{equation} \label{6.61}
\|\PPPP_k^V(f-f_n)\|_{L_R^\infty}\le \delta_n\,\|\PPPP_k^V\wwww'\|_{L_R^\infty}
\end{equation}
where $\delta_n=\|f-f_n\|_{L_{\wwww'}^\infty}\to0$ as $n\to\infty$. It follows from~\eqref{3.2} with~$\wwww$ replaced by~$\wwww'$ that
$$
\|\PPPP_k^V\wwww'\|_{L_R^\infty}
\le C_4(R)\|\PPPP_k^V{\mathbf1}\|_{R_0}.
$$
Substituting this inequality into~\eqref{6.61} and recalling that~\eqref{5.24} holds with $f=\mathbf1$ (and, hence, the sequence $\lambda_V^{-k}\|\PPPP_k^V{\mathbf1}\|_{R_0}$ is bounded), we obtain~\eqref{6.57}. This completes the proof of Theorem~\ref{t6.5}. 
\end{proof}

The following corollary of Theorem~\ref{t6.5} will be important for proving that the LDP holds uniformly with respect to a class of initial measures. 

\begin{corollary} \label{c6.6}
Under the hypotheses of Theorem~\ref{t6.5}, for any positive numbers~$\delta$ and $M$ and any functions $V\in\VV$ and $f\in C_\wwww^\infty$, the convergence 
\begin{equation} \label{6.62}
\lambda_V^{-k}\,\E_\nu\bigl\{\exp\bigl(V(\uuu_1)+\cdots+V(\uuu_k)\bigr)f(\uuu_k)\bigr\}\to\langle f,\mu_V\rangle\,\langle h_V,\nu\rangle, \quad k\to\infty, 
\end{equation}
holds uniformly in $\nu\in\LLambda(\delta,M)$. 
\end{corollary}

\begin{proof}
By Theorem~\ref{t6.5}, convergence~\eqref{6.62} holds for $\nu=\delta_\uuu$, uniformly with respect to $\uuu\in B^\ell(R)$ for any $R>0$. It is easily seen that the required result with be established if we prove that, uniformly in $\lambda\in\LLambda(\delta,M)$, 
\begin{equation} \label{6.63}
\sup_{k\ge1}\biggl\{\int_XI_{B^\ell(R)^c}\bigl|\lambda_V^{-k}\PPPP_k^Vf-\langle f,\mu_V\rangle h_V\bigr|\,\dd\nu\biggr\}\to0
\quad\mbox{as $R\to\infty$}.
\end{equation}
To see this, let us note that, by~\eqref{3.2} and~\eqref{5.24}, we have
$$
\|\PPPP_k^Vf\|_{L_\wwww^\infty}\le C_1\|\PPPP_k^V{\mathbf1}\|_{R_0}\le C_2\lambda_V^k\quad\mbox{for all $k\ge1$}.
$$
It follows that 
$$
\bigl|\lambda_V^{-k}\PPPP_k^Vf(\uuu)\bigr|\le C_2\wwww(\uuu), \quad\uuu\in X, \quad k\ge0. 
$$
Since $h_V\in C_\wwww(X)$ and 
$$
\sup_{\nu\in\LLambda(\delta,M)}\int_XI_{B^\ell(R)^c}(\uuu)\,\wwww(\uuu)\,\nu(\dd\uuu)\to0\quad\mbox{as $R\to\infty$}, 
$$
we obtain the required convergence~\eqref{6.63}.
\end{proof}

\begin{proof}[Proof of Lemma~\ref{l6.6}]
In view of~\eqref{1.017}, it suffices to prove that
\begin{equation} \label{6.71}
\sup_{k\ge2}\E_\lambda \|u_k\|_s^n\le C(\delta,M)\quad\mbox{for any $\lambda\in\Lambda(\delta,M)$}. 
\end{equation}
It follows from inequality~\eqref{6.31} with $(v,\eta_1)$ replaced by $(u_{k-1},\eta_k)$ that 
\begin{equation} \label{5.064}
\|u_{k}\|_s^n\le C_1\bigl(\varPhi_m(u_{k-1})+e^{2sn\pppp(u_{k-1})}+\|\eta_k\|_U^{2sn}+\|\eta_{k}\|^{2n}+1\bigr). 
\end{equation}
Taking the mean value with respect to~$\E_\lambda$ and using the Markov property and inequalities~\eqref{1.8}, \eqref{1.08}, \eqref{1.7}, and~\eqref{1.09}, for sufficiently small $s>0$ we derive
\begin{equation} \label{6.60}
\E_\lambda \|u_{k}\|_s^n\le C_2\bigl(\E_\lambda\varPhi_m(u_{k-1})+e^c\,\E_\lambda e^{\varkappa\varPhi(u_{k-2})}+1\bigr). 
\end{equation}
Recalling~\eqref{3.03}, we arrive at the required inequality~\eqref{6.71}.
\end{proof}

\begin{proof}[Proof of Lemma~\ref{l6.7}]
It follows from inequality~\eqref{5.064} with $n=1$ that
$$
F_s(u_k)\le \varPhi(u_{k-1})+2s\pppp(u_{k-1})+\|\eta_k\|_U+C_1. 
$$
Multiplying both sides of this inequality by $\exp(V(\uuu_1)+\cdots+V(\uuu_k))$, applying the mean value~$\E_\uuu$, and using the Markov property and inequalities~\eqref{1.08} and~\eqref{3.02}, we derive
\begin{equation} \label{5.066}
(\PPPP_k^VF_s)(\uuu)
\le e^{2\|V\|_\infty}\bigl(\PPPP_{k-2}^V\varPhi+2\,\PPPP_{k-2}^V(\PPPP_1\pppp)+C_2\PPPP_{k-2}^V{\mathbf1}\bigr)(\uuu). 
\end{equation}
In view of~\eqref{1.7}, \eqref{1.09}, and the Jensen inequality, we have
$$
\PPPP_1\pppp(u)=\E_u\pppp(u_1)\le\frac1\delta\log \E e^{\delta\pppp(u_1)}\le\frac1\delta(\rho\,\varPhi(u)+c). 
$$
Substituting this into~\eqref{5.066} and using~\eqref{3.2}, we see that the $L_\wwww^\infty$-norm of the right-hand side can be estimated by~$\|\PPPP_k{\mathbf1}\|_{R_0}$. 
\end{proof}

\subsection{Pressure function}
\label{s8.6}
We now prove that, for any $V\in C_b(H^\ell)$, $\delta>0$ and~$M>0$, the limit~\eqref{1.21} exists uniformly with respect to $\lambda\in\LLambda(\delta,M)$. Indeed, when $V\in\VV$, this property follows from Corollary~\ref{c6.6} applied to the function~$f=\mathbf1$. To deal with the general case, we need the concept of buc-convergence. Following~\cite{FK2006}, we say that a sequence $\{V_n\}\subset C_b(X)$ {\it buc-converges to $V\in C_b(X)$\/} if $\sup_n\|V_n\|_\infty<\infty$ and $\|V-V_n\|_{L^\ty(K)}\to0$ as $n\to \ty$ for any compact set $K\subset X$.\footnote{Note that the concept of buc-convergence can be defined on any Polish space.}

We claim that the existence of a limit~\eqref{1.21},  uniformly with respect to $\lambda\in\LLambda(\delta,M)$, follows immediately from the following two properties:
\begin{itemize}
\item[\bf(a)]
In the setting of Section~\ref{s4}, let $\{\zeta_\theta\}$ be an exponentially tight family of random probability measures. Then the set of functions $V\in C_b(X)$ for which limit~\eqref{2.1} exists is buc-closed.
\item[\bf(b)]
Let $X=H^\ell$ and let $\VV\subset C_b(X)$ be the subspace defined above. Then any function $V\in C_b(X)$ can be buc-approximated by a sequence $\{V_n\}\subset\VV$. 
\end{itemize}
Indeed, taking these properties for granted, let us define the set $\Theta=\N\times\LLambda(\delta,M)$  of elements $\theta=(k,\lambda)$ with an order relation~$\prec$ such that $(k_1,\lambda_1)\prec (k_2,\lambda_2)$ if and only if $k_1\le k_2$. Denote by~$\zzeta_\theta$ the random measure $\zzeta_k$ (see~\eqref{1.016}) considered on the probability space $(\Omega,\FF,\IP_\lambda)$. It is straightforward to check that the existence of limit~\eqref{1.21} uniformly with respect to~$\lambda\in\LLambda(\delta,M)$ is equivalent to the existence of the limit
\begin{equation} \label{6.70}
\lim_{(k,\lambda)\in\Theta}\frac{1}{k}\log\int_\Omega \exp\bigl(k\langle V,\zzeta_k\rangle\bigr)\dd\IP_\lambda. 
\end{equation}
As was mentioned above, this limit exists for $V\in\VV$. By property~(a), the set of functions for which~\eqref{6.70} exists is buc-closed in~$C_b(H^\ell)$, and by~(b), any function $V\in C_b(H^\ell)$ can be buc-approximated by a sequence from~$\VV$. Hence, limit~\eqref{1.21} exists uniformly in $\lambda\in\LLambda(\delta,M)$ (with arbitrary positive~$\delta$ and~$M$) for any $V\in C_b(H^\ell)$. Thus, it remains to establish properties~(a) and~(b).

\smallskip
{\it Proof of~(a)}. 
To any function $V\in C_b(X)$ there corresponds a bounded continuous function $\widetilde V:\PP(X)\to\R$ defined by $\widetilde V(\nu)=\langle V,\nu\rangle$. By Proposition~3.17 in~\cite{FK2006}, the set of bounded continuous functions $F:\PP(X)\to\R$ for which the limit 
$$
\lim_{\theta\in\Theta} \frac{1}{r(\theta)}\log\int_{\Omega_\theta} \exp\bigl(r(\theta)F(\zeta_\theta)\bigr)\dd \pP_\theta
$$
exists is buc-closed\footnote{Proposition~3.17 in~\cite{FK2006} deals with the case when $\Theta=\N$. However, the proof presented there remains valid for exponentially tight families of random measures indexed by a directed set.} in $C_b(\PP(X))$. Thus, it suffices to  prove that if a sequence $\{V_n\}\subset C_b(X)$ buc-converges to~$V$, then $\{\widetilde V_n\}$ buc-converges to~$\widetilde V$ in $C_b(\PP(X))$. 

To see this, we first note that~$\{\widetilde V_n\}$ is bounded in $C_b(\PP(X))$. Now fix a compact set $\KK\subset\PP(X)$ and a number~$\e>0$, and use the Prokhorov compactness criterion to find a compact subset $K_\e\subset X$ such that $\nu(K_\e^c)<\e$ for any $\nu\in \KK$. We have
\begin{align*}
|\lag V-V_n,\nu\rag|&\le \int_{K_\e} |V(u)-V_n(u)|\nu(\dd u)+\int_{K_\e^c} |V(u)-V_n(u)|\nu(\dd u)\\
&\le \|V-V_n\|_{L^\ty(K_\e)}+\e\Bigl(\|V\|_\ty+\sup_{n\ge1}\|V_n\|_\ty\Bigr),
\end{align*}
where $\nu\in \KK$. The right-hand side of this inequality can be made arbitrarily small by choosing~$n\gg1$ and~$\e\ll 1$. 

\smallskip
{\it Proof of~(b)}. 
Given $V\in C_b(H^\ell)$, we define $V_n(\uuu)=V({\mathsf P}_n\uuu)$. Then $V_n\in\VV$  and $\|V_n\|_\infty\le\|V\|_\infty$ for any $n\ge1$. Let $K\subset H^\ell$ be a compact set. Since~$\{{\mathsf P}_n\}$ converges to identity in the strong operator topology, and the strong convergence is uniform on compact subsets, we see that $\|V_n(\uuu)-V(\uuu)\|_{L^\infty(K)}\to0$ as $n\to\infty$. This completes the proof of Property~1. 

\subsection{Uniqueness of equilibrium state}
\label{s8.7}
In this section, we   show that,  for any~$V\in \VV$, there is a unique equilibrium state~$\sigma_V\in \PP(H^\ell)$ for~$Q_\ell(V)$. Recall that the pressure function $Q_\ell:C_b(H^\ell)\to\R$ is $1$-Lipschitz continuous and convex and that $I_\ell:\MM(H^\ell)\to \R$ stands for its Legendre transform. It follows from~\eqref{5.24} and the positivity of~$h_V$ that
\begin{equation} \label{6.66}
Q_\ell(V)=\log\lambda_V=\lim_{k\to\infty}\frac1k\log(\PPPP_k^Vf)(\uuu),
\end{equation}
where $f\in C_w(H^\ell)\cap C_+(H^\ell)$ and $\uuu\in H^\ell$ are arbitrary. 

\smallskip
We define a semigroup~$\SSS_k^V$ by~\eqref{1.27} and denote by~$\SSS^{V*}_k:\PP(H^\ell)\to\PP(H^\ell)$ its  dual semigroup. As in the case of~$\PPPP_k$, we can consider the corresponding generalised Markov semigroup (cf.~\eqref{6.3}):
$$
(\SSSS_1^F f)(\uuu)=\SSS_1^V(e^Ff), \quad \SSSS_k^F=(\SSSS_1^F)^k,
$$
where $F\in C_b(H^\ell)$ is a fixed function. It is straightforward to check that $\SSSS_1^Ff=\lambda_V^{-1}h_V^{-1}\PPPP_1^{V+F}(h_Vf)$, whence it follows that
\begin{equation} \label{6.67}
\SSSS_k^Ff=\lambda_V^{-k}h_V^{-1}\PPPP_k^{V+F}(h_V f), \quad k\ge0. 
\end{equation}
Combining this relation with Theorem~\ref{t6.5}, we see that the pressure function~$Q_\ell^V$ is well defined for~$\SSS_k^V$:
$$
Q_\ell^V(F)=\lim_{k\to\infty}\frac1k\log(\SSSS_k^F{\mathbf1})(\uuu), \quad F\in C_b(H^\ell), \quad \uuu\in H^\ell.
$$ Denote by~$I ^V_\ell: \MM(H^\ell)\to\R $   its Legendre transform. We shall  use the following   well-known   characterisation of   a stationary measure; see \   Lemma~2.5 in~\cite{DV-1975}. 
\begin{lemma}\label{L:5.8}
We have $I_\ell^V(\sigma)=0$ for some~$\sigma\in\PP(H^\ell)$  if and only if $\sigma$ is a  stationary measure for $\SSS_1^{V*}$.
\end{lemma}   

Convergence~\eqref{5.24} implies the uniqueness of a stationary measure for $\SSS_1^{V*}$. More precisely, we have the following result.

\begin{lemma}\label{L:5.9}
The semigroup $\SSS^{V*}_k$ has a unique stationary measure, which is given by $\nu_V=h_V\mu_V$. 
\end{lemma}

\begin{proof} 
The relation~$\nu_V(H^\ell)=\lag h_V,\mu_V\rag =1$ implies that~$\nu_V \in {\cal P}({H^\ell})$.
For any $g\in C_b(H^\ell)$, we have
$$
\lag \SSS^V_1 g, \nu_V \rag
= \la_V^{-1}\lag    \PPPP_1^V (gh_V), \mu_V \rag
=  \lag     gh_V , \mu_V \rag
=\lag   g , \nu_V \rag,
$$
whence it follows that~$\nu_V$ is a  stationary measure for~$\SSS^{V*}_1$. Furthermore, applying~\eqref{5.24} 
to $gh_V\in C_{\wwww_m}(H^\ell)$, for any $R>0$ we get 
$$
\SSS^V_k g \to\lag g ,\nu_V\rag \quad\mbox{in~$C_b(B_{H^\ell}(R)) $ as~$k\to\infty$}. 
$$
Since~$g$ is arbitrary, we see that $\SSS^{V*}_k \sigma \to   \nu_V$ in $\MM_+(H^\ell)$ as $k\to\ty$
for any $\sigma\in\PP(H^\ell)$, whence we conclude that~$\nu_V$ is the unique stationary measure for~$\SSS^{V*}_k$.
\end{proof}
We are now ready to prove the existence and the uniqueness of equilibrium state. 
Relations~\eqref{6.66} and~\eqref{6.67}  imply that
\begin{equation*} %\label{6.68}
Q_\ell^V(F)=Q_\ell(V+F)-Q_\ell(V). 
\end{equation*}
It follows that   
$$
I_\ell^V(\sigma)=\sup_{F\in C_b(H^\ell)}\bigl(\lag F, \sigma\rag-Q_\ell^V(F)\bigr)=I_\ell(\sigma)+Q_\ell(V)-\langle V,\sigma\rangle
$$ 
for any $\sigma\in \PP(H^\ell)$. Using   Lemmas~\ref{L:5.8} and~\ref{L:5.9}, we conclude that the relation
$$
I_\ell(\sigma)=\langle V,\sigma\rangle-Q_\ell(V) 
$$
holds if and only if $\sigma=\nu_V$. Thus, $\nu_V$ is the unique equilibrium state.

\subsection{Completion of the proof}
\label{s8.8}
Let us fix positive numbers~$\delta$ and~$M$. As was explained in Section~\ref{s1.5}, the LDP for the occupation measures~\eqref{1.18} (which is uniform with respect to the initial measure $\lambda\in\Lambda(\delta,M)$) will be established if we prove the uniform LDP for the measures~$\zzeta_k$ given by~\eqref{1.016}. Recall that the directed set~$\Theta$ of pairs $\theta=(k,\lambda)$, $k\in\mathbb N$, $\lambda\in\LLambda(\delta,M)$ was defined in Section~\ref{s8.6} and that convergence in $\theta\in\Theta$ is equivalent to the convergence as $k\to\infty$, which is uniform in $\lambda\in\LLambda(\delta,M)$. By Theorem~\ref{T:2.3}, it suffices to prove the following three properties:
\begin{itemize}
\item[\bf(a)]
The family~$\{\zeta_k\}$ satisfies the hypotheses of Lemma~\ref{L:2.2}.
\item[\bf(b)]
Limit~\eqref{1.21} exists uniformly in $\lambda\in\LLambda(\delta,M)$.
\item[\bf(c)] 
The equilibrium state is unique for any function~$V$ that belongs to a vector space $\VV\subset C_b(H^\ell)$ whose restriction to any compact set $K\subset H^\ell$ is dense in~$C(K)$. 
\end{itemize}
The validity of~(a) was proved in Section~\ref{s8.1}, and the existence of a uniform limit~\eqref{1.21} was established in Section~\ref{s8.6}. As was shown in Section~\ref{s8.7}, the equilibrium state is unique for the functions~$V$ of the form~\eqref{6.51}. Since these functions form a vector space which is determining for~$\PP(H^\ell)$, their restrictions to any compact subset $K\subset H^\ell$ must be dense in~$C(K)$. Thus, properties (a)--(c) are established, and to complete the proof of Theorem~\ref{t1.2} we need to establish the uniform Feller property used in Section~\ref{s8.5}. This is done in the next section.

\section{Uniform Feller property}
\label{s7}
\subsection{Coupling}
\label{s6.1}
In this section, we recall a coupling construction for trajectories of system~\eqref{1.1} (see Section~3.2.2 in~\cite{KS-book}). We shall always assume that the hypotheses of Theorem~\ref{t1.2} are satisfied. 

\smallskip
As was mentioned after the formulation of Condition~(B), the sequence of finite-dimensional subspaces~$\{H_N\}$ can be chosen arbitrarily. From now on, we assume that~$H_N$ is the vector span of~$e_1,\dots,e_N$, where~$\{e_j\}$ is the orthonormal basis entering the decomposition of the random variables~$\eta_k$; see~\eqref{1.5}. In this case, the law of~${\mathsf P}_N \eta_1$ has a density~$D$ against the Lebesgue measure:
$$
D(v)=\prod_{i=1}^N b^{-1}_i p_i(b^{-1}_ix_i), 
\quad v=(x_1,\ldots,x_N)\in H_N.
$$
For any $u\in H$, let us denote by~$\nu_u$ the law of~${\mathsf P}_N (S(u)+\eta_1)$. The following result is Lemma 3.2.6 in  \cite{KS-book}.

\begin{lemma}\label{L:6.1}
For any integer $N\ge1$, there is a probability space $(\Omega, \FF,\pP)$ and two families of $H$-valued random variables $\zeta=\zeta(v,v',\omega)$ and $\zeta'=\zeta'(v,v',\omega)$ with $v,v'\in H$ such that the following properties hold. 
\begin{itemize}
\item[\bf(i)] The laws of   $\zeta$ and $\zeta'$ coincide with that of $\eta_1$.
\item[\bf(ii)] The random variables $({\mathsf P}_N\zeta,{\mathsf P}_N\zeta')$ and $({\mathsf Q}_N\zeta, {\mathsf Q}_N\zeta')$ are independent. Furthermore, the random variables ${\mathsf Q}_N\zeta$ and ${\mathsf Q}_N\zeta'$ are equal for all~$\omega\in\Omega$ and do not depend on $(v,v').$ 
\item[\bf(iii)] The pair
$$
V= {\mathsf P}_N (S(v)+\zeta), \quad V'= {\mathsf P}_N (S(v')+\zeta')
$$is a maximal coupling for $(\nu_v,\nu_{v'})$ and
\begin{equation}\label{E:6.1}
\pP\{V\neq V'\}\le C_N \|S(v)-S(v')\| \quad\text{for any $v,v'\in H$}.
\end{equation}

\item[\bf(iv)]  
The random variables $\zeta$ and $\zeta'$ are measurable functions of $(v,v',\omega)\in H\times H\times \Omega.$
\end{itemize}
\end{lemma} 

Using this result, we now define coupling operators by the formulas
$$
\RR(v,v',\omega)=S (v)+\zeta(v,v',\omega),\quad \RR'(v,v',\omega)=S(v')+\zeta'(v,v',\omega),
$$ 
where  $v, v'\in H$ and $\omega\in\Omega$. Let $(\Omega^k,\FF^k,\pP^k)$, $k\ge1$ be independent copies of the
probability space constructed in Lemma~\ref{L:6.1} and $(\Omega,\FF,\pP)$ be their direct product. For any~$u,u'\in H$, we set   $u_0=u$, $u_0'=u'$,  and 
\begin{align}
u_k(\omega)&=\RR(u_{k-1}(\omega),u'_{k-1}(\omega),\omega^k), &\quad
u'_k(\omega)&=\RR'(u_{k-1}(\omega),u'_{k-1}(\omega),\omega^k),\nonumber\\
\zeta_k(\omega)&=\zeta(u_{k-1}(\omega),u'_{k-1}(\omega),\omega^k), &\quad 
\zeta'_k(\omega)&=\zeta'(u_{k-1}(\omega),u'_{k-1}(\omega),\omega^k), \nonumber
\end{align} 
where $\omega=(\omega^1,\omega^2,\ldots)\in\Omega$ and $k\ge1$.  Then, by construction, $\{\zeta_k\}$ and~$\{\zeta_k'\}$ are sequences of i.i.d.  random variables, while~$\{u_k\}$ and~$\{u_k'\}$ are the corresponding trajectories of~\eqref{1.1}. We shall say that $(u_k,u_k')$ is a {\it coupled trajectory at level~$N$\/} for~\eqref{1.1} issued from~$(u,u')$.

\subsection{The result and its proof}
\label{s6.2}
Let us recall that, given a function $V\in C_b(H^\ell)$, the generalised Markov semigroup~$\PPPP_k^V$ is given by~\eqref{6.3}, the subspace $\VV\subset C_b(H^\ell)$ was introduced in Section~\ref{s8.5}, and given $s\in(0,1]$, the compact subsets $X_R\subset H^\ell$ are defined in~\eqref{6.44}. The aim of this section is to prove the following theorem on the validity of the uniform Feller property (see  Theorem~\ref{T:3.1}) for any $V\in\VV$. 

\begin{theorem} \label{t7.1}
Under the hypotheses of Theorem~\ref{t1.2}, for any $s\in(0,1]$ and $V\in\VV$ there is  an integer $R_0\ge1$ such that the sequence $\{\|\PPPP_k^V{\mathbf1}\|_R^{-1}\PPPP_k^Vf,k\ge0\}$ is uniformly equicontinuous on~$X_R$ for any $f\in \VV$ and any integer~$R\ge R_0$. 
\end{theorem}

\begin{proof} 
We invoke some ideas from~\cite[Chapter 3]{KS-book} which were used to establish exponential mixing for system~\eqref{1.1}.   The proof is divided into five steps. 

\smallskip
{\it Step~1: Reduction}.
Let us fix two functions $V,f\in\VV$. With a slight abuse of notation, we shall use the same letters to denote the function~$F$ entering representation~\eqref{6.51} for~$V$ and~$f$. Without loss of gene\-rality, we can assume that    \eqref{6.51} holds for~$V$ and~$f$ with the same integer~$N_0$.
  Furthermore, in view of~\eqref{1.017}, for $k\ge1$ and $f\in C_b(H^\ell)$, we have 
$$
\PPPP_k^Vf(\uuu)=\E_{u^\ell}\bigl\{(\Xi_Vf)(\uuu_k, k)\bigr\},
$$
where $\uuu=[u^1,\dots,u^\ell]\in H^\ell$,  $\uuu_k=[\uuu,u_{1},\dots,u_k]$, 
\begin{equation}\label{E:7.2}
(\Xi_Vf)(\uuu^k, k):=\exp\biggl(\sum_{j=1}^kV(u_{j-\ell+1},\dots,u_j)\biggr)f(u_{k-\ell+1},\dots,u_k),
\end{equation}
and $u_{i}=u^{\ell-i}$ for $i\in[2-\ell,0]$. 
We need to prove the uniform equicontinuity of the sequence $\{g_k,k\ge \ell\}$ on~$X_R$, where 
$$
g_k(\uuu)= \|\PPPP_k^V{\mathbf1}\|_R^{-1}\PPPP_k^Vf(\uuu). 
$$
There is no loss of generality in assuming that $0\le f\le1$ and $\inf_H V=0$, so that $\Osc_H(f)\le1$ and $\Osc_H(V)=\|V\|_\infty$.

\smallskip
{\it Step~2: Stratification}. Let us take two points $\zzz_i=[z_i^{1},\dots,z_i^\ell]\in X_R$, $i=1,2$ and  an integer $N\ge N_0$. 
In what follows, we denote by $(\Omega,\FF,\pP)$  the probability space constructed in Section~\ref{s6.1} and by $(u_k,u_k'):=(u_k(z^\ell_1), u_k'(z^\ell_2))$ a coupled trajectory at level~$N$ issued from $(z^\ell_1,z^\ell_2)\in H\times H$. Let us set
\begin{gather*}
\bar G(r)=\bigcap_{j=1}^rG(j), \quad  G(j)=\{{\mathsf P}_Nu_j ={\mathsf P}_Nu_j'\},\\
F(r,\rho)=\biggl\{\sum_{j=0}^{r}  \left(\pppp(u_j  )+\pppp(u_j' )\right) \le  \rho\biggr\}, \quad F(r,0)=\varnothing, 
\end{gather*}
where $r,\rho\ge1$ are integers. We also define the pairwise disjoint events\,\footnote{\,Of course, the events~$A_{r,\rho}$ depend also on~$\zzz_1,\zzz_2\in X_R$ and $N$. However, to simplify the notation, we do not indicate that dependence.}
$$
A_{r,\rho}:=\bigl(\bar G(r-1)\cap G(r)^c\cap F(r-1,\rho)\bigr)\setminus F(r-1,\rho-1),\quad r,\rho\ge1,
$$
where $\bar G(0)=\Omega$.
Using the fact that~$f(\uuu)$ and~$V(\uuu)$ depend only on~${\mathsf P}_N\uuu$ and setting 
$$
\uuu_k=[\zzz_1,u_1,\dots,u_k], \quad \uuu_k'=[\zzz_2,u_1',\dots,u_k'], \quad
\vvv_k=[\zzz_2,u_1,\dots,u_k],
$$
for $k\ge\ell$ we write
%\,\footnote{\,With a slight abuse of notation,  we denote by $(\Xi_Vf)(\zzz_i,k)$ (with $i=1,2$) the random variable defined by relation~\eqref{E:7.2} in which~$u_k$ is replaced by~$u_k(z^\ell_i)$.} 
 \begin{align} 
\PPPP_k^Vf(\zzz_1)-\PPPP_k^Vf(\zzz_2)
&=I^k(\zzz_1,\zzz_2)+\sum_{r=1}^k\sum_{\rho=1}^{\ty} I_{r,\rho}^k(\zzz_1,\zzz_2)
\label{E:7.3}
\end{align}
where we set
\begin{align*}
I^k(\zzz_1,\zzz_2)&= \E\bigl\{(\Xi_Vf)(\uuu_k, k)-(\Xi_Vf)(\vvv_k, k)\bigr\},\\
I_{r,\rho}^k(\zzz_1,\zzz_2)&=\E \bigl\{I_{A_{r,\rho}}\bigl[(\Xi_Vf)(\vvv_k,k ) -(\Xi_Vf)(\uuu_k',k)\bigr] \bigr\}.  
\end{align*}
To prove the uniform equicontinuity of~$\{g_k, k\ge\ell\}$, we first estimate these two quantities.

\smallskip
{\it Step~3: Estimates for $I^k$ and $I_{r,\rho}^k$}. Since~$V$ is bounded and globally Lipschitz continuous and~$f$ satisfies the inequality $0\le f\le 1$, for $\zzz_1,\zzz_2\in X_R$ we have
\begin{equation} \label{6.4}
|I^k(\zzz_1,\zzz_2)|\le C_1d\,\|\PPPP_k^V{\bf1}\|_R,
\end{equation}
where $d=\|\zzz_1-\zzz_2\|$. Furthermore, using the positivity of~$\Xi_Vf$, the inequalities $0<f\le 1$ and $V\ge0$, and the Markov property, we derive
\begin{align*}
I_{r,\rho}^k(\zzz_1,\zzz_2)
&\le\E\bigl\{I_{A_{r,\rho}}(\Xi_Vf)(\vvv_k,k)\bigr\}
\le \E\bigl\{I_{A_{r,\rho}}(\Xi_{V}{\mathbf1})(\vvv_k,k)\bigr\}\\
&= \E\bigl\{I_{A_{r,\rho}}\E\bigl[(\Xi_{V}{\mathbf1})(\vvv_k,k)\,\big|\,\FF_r\bigr]\bigr\}
\le e^{(r+\ell)\|V\|_\infty}\E\bigl\{I_{A_{r,\rho}}(\PPPP_{k-r}^V{\mathbf1})(u_r)\bigr\},
\end{align*}
where $\{\FF_k\}$ stands for the filtration generated by~$(u_k,u_k')$. 
Applying the inequality 
$$
\PPPP_{k-r}^{V}{\mathbf1}(\zzz)\le M\|\PPPP_{k-r}^{V}{\mathbf1}\|_{R_0}\wwww_m(\zzz),
$$
which follows from~\eqref{3.2}, we obtain
\begin{align*}
I_{r,\rho}^k(\zzz_1,\zzz_2)
&\le M e^{(r+\ell)\|V\|_\infty}\|\PPPP_{k-r}^V{\mathbf1}\|_{R_0}\E\bigl\{I_{A_{r,\rho}}\wwww_m(u_r)\bigr\}\\
&\le M e^{(r+\ell)\|V\|_\infty}\|\PPPP_{k-r}^V{\mathbf1}\|_{R_0}\bigl\{\pP(A_{r,\rho})\,\E\,\wwww_m^2(u_r)\bigr\}^{1/2}. 
\end{align*}
Combining this with~\eqref{3.03} and using the symmetry, for $\zzz_1,\zzz_2\in X_R$ and $R\ge R_0$, we derive
\begin{equation} \label{6.5}
|I_{r,\rho}^k(\zzz_1,\zzz_2)|\le C_2(R,V)e^{r\|V\|_\infty}\|\PPPP_k^V{\mathbf1}\|_R\,\pP(A_{r,\rho})^{1/2}. 
\end{equation}

\smallskip
{\it Step~4: An estimate for $\pP(A_{r,\rho})$}. 
We now  prove that
\begin{equation}\label{E:7.7}
\pP(A_{r,\rho}) \le C_3(R,N)\bigl(\{\gamma_N^{-r}e^{2\rho}d\}\wedge e^{cr- \sigma\rho}\bigr),
\end{equation}
where $\sigma>0$ does not depend on the other parameters. To this  end, note that, in view of~\eqref{E:8.6}, on the event $\bar G(r-1)\cap F(r-1,\rho)$, we have
\begin{align*}
\|u_{r-1}-u_{r-1}'\|
&\le \gamma_N^{1-r}\exp\biggl(\,\sum_{k=0}^{r-2}\bigl(\pppp(u_k)+\pppp(u_k')\bigr)\biggr)\|z_1^\ell-z_2^\ell\| \\
&\le \gamma_N^{1-r}\exp\bigl\{2\rho-\pppp(u_{r-1})-\pppp(u_{r-1}')\bigr\}\,d. 
\end{align*}
Combining this with~\eqref{E:6.1} and~\eqref{1.4} (for $N=0$) and using the Markov property, we derive
\begin{align} 
\IP\bigl\{\bar G(r-1)\cap G(r)^c \cap F(r-1,\rho)\bigr\}
&=\E\bigl\{I_{\bar G(r-1)\cap F(r-1,\rho)}\E\bigl(I_{G(r)^c}\,\big|\,\FF_{r-1}\bigr)\bigr\}
\notag\\
&\le C_N\,\gamma_0^{-1}\gamma_N^{1-r}e^{2\rho}d. 
\label{6.7}
\end{align}
On the other hand, we have $A_{r,\rho}\subset F(r-1,\rho-1)^c$. Recalling~\eqref{1.7} and~\eqref{1.09} and using the Chebyshev inequality and the boundedness op~$\pppp$ on any ball of~$H$, we obtain
\begin{align*} 
\pP(A_{r,\rho})
&\le e^{-\frac{\de}{2}(\rho-1)} \E\exp\biggl(\frac{\de}{2} \sum_{j=0}^{r-1}\bigl(\pppp(u_j )+\pppp(u_j')\bigr)\biggr)\\
&\le e^{-\frac{\de}{2}(\rho-1)} \biggl\{\E\exp\biggl(\de \sum_{j=0}^{r-1}\pppp(u_j )\biggr)\,
 \E\exp\bigg(\de\sum_{j=0}^{r-1}\pppp(u_j')\biggr)\biggr\}^{\frac12}\\
 &\le C_4(R)\,e^{c r-\sigma\rho},
\end{align*}
where $\sigma=\delta/2$. Combining this with~\eqref{6.7}, we get the required inequality~\eqref{E:7.7}. 

\smallskip
{\it Step~5: Completion of the proof}. 
Inequalities~\eqref{E:7.3}--\eqref{E:7.7} imply that
$$
\bigl|g_k(\zzz_1)-g_k(\zzz_2)\bigr|
\le C_1d+C_5(R,V,N)\sum_{r,\rho=1}^\infty 
e^{r\|V\|_\infty}\bigl(\{\gamma_N^{-r}e^{2\rho}d\}\wedge e^{cr- \sigma\rho}\bigr)^{\frac12}
$$
for $\zzz_1,\zzz_2\in X_R$, $k\ge\ell$, and $R\ge R_0$. Since the sum on the right-hand side  vanishes for $d=0$, the uniform equicontinuity of~$\{g_k\}$ will be established if we prove that the series converges uniformly in~$d\in[0,1]$. In view of the positivity and monotonicity of its terms, it suffices to show that
\begin{equation} \label{E:7.5}
\sum_{r,\rho=1}^\infty e^{r\|V\|_\infty}\bigl(\{\gamma_N^{-r}e^{2\rho}\}\wedge e^{cr- \sigma\rho}\bigr)^{\frac12}<\infty.
\end{equation}

To this end, let us fix a number $D>1$ and  choose $N\ge1$ so large that $\log\gamma_N \ge2(\|V\|_\infty+D)$. Then  
\begin{equation} \label{6.10}
\gamma_N^{-r/2}e^{r\|V\|_\infty}\le e^{-Dr}.
\end{equation}
Define the sets
$$
S_1=\{(r,\rho)\in\N^2: \rho-Dr\le-\rho-r\}, \quad S_2=\N^2\setminus S_1.
$$
Using~\eqref{6.10}, it is straightforward to check that
$$
\sum_{(r,\rho)\in S_1}e^{r\|V\|_\infty} \bigl(\gamma_N^{-r}e^{2\rho}\bigr)^{1/2}\le\frac{1}{(e-1)^2}.
$$
Furthermore, if $D> 1+\sigma^{-1}(8\|V\|_\ty+4c)$, then
$$
\sum_{(r,\rho)\in S_2} e^{r\|V\|_\ty}e^{(cr-\sigma\rho)/2}\le C_6 \sum_{\rho=1}^\infty   e^{-\sigma\rho/4}<\infty.
$$
Combining the last two inequalities, we see that~\eqref{E:7.5} holds. The proof of Theorem~\ref{t7.1} is complete. 
\end{proof}

\section{Appendix}
\label{s9}

\subsection{A priori estimates}
\label{s9.0}
In this section, we establish some simple estimates for trajectories of the Markov process~\eqref{1.1}. Recall that Conditions~(A)--(C) are formulated in Section~\ref{s1.1}. 

\begin{lemma} \label{l3.1}
Suppose that Condition {\rm(A)}  holds, and $\{\eta_k\}$ is a sequence of i.i.d.\ random variables in~$H$ such that $\E e^{\delta\varPhi(\eta_1)}<\infty$ for some $\delta>0$. Then,  for any $k\ge0$ and any initial measure $\lambda\in\PP(H)$, we have
\begin{equation} \label{3.02}
\E_\lambda\varPhi(u_k)\le q^k\int_H\varPhi(v)\lambda(\dd v)+C(1-q)^{-1}\E\,\varPhi(\eta_1).
\end{equation}
If, in addition,  $\varkappa>0$ is so small that $C\varkappa(1-q)^{-1}\le \delta$, then
\begin{equation} \label{3.01}
\E_\lambda\exp\biggl(\varkappa\sum_{n=0}^k\varPhi(u_n)\biggr)
\le \bigl(\E e^{\delta\,\varPhi(\eta_1)}\bigr)^k\int_H\exp\bigl\{\varkappa(1-q)^{-1}\varPhi(v)\bigr\}\lambda(\dd v).
\end{equation}
\end{lemma}

\begin{proof}
It follows from~\eqref{1.04} that 
$$
\E_\lambda\varPhi(u_k)\le q \,\E_\lambda\varPhi(u_{k-1})+C\,\E\,\varPhi(\eta_k).
$$
Iterating this inequality, we obtain~\eqref{3.02}. Furthermore, using again~\eqref{1.04}, we write  
$$
\E_v\exp(\varkappa\,\varPhi(u_1))\le e^{q\varkappa\varPhi(v)}\E e^{C\varkappa\varPhi(\eta_1)},
$$
where $\varkappa>0$ satisfies the inequality given in the statement. Combining this with the Markov property and arguing by induction, we obtain~\eqref{3.01}.
\end{proof}

\begin{remark} \label{r3.2}
A similar argument shows that if $\varkappa>0$ is so small that $C\varkappa\le\delta$, then
\begin{equation} \label{3.03}
\E_\lambda e^{\varkappa\varPhi(u_k)}\le \bigl(\E\,e^{\delta\varPhi(\eta_1)}\bigr)^{1/(1-q)}\int_He^{q^k\delta\varPhi(v)}\lambda(\dd v).
\end{equation}
A simple application of the Fatou lemma implies now that, for any stationary measure $\mu\in\PP(H)$ of~\eqref{1.1}, we have
\begin{equation} \label{3.04}
\int_He^{\varkappa\varPhi(v)}\mu(\dd v)\le \bigl(\E\,e^{\delta\varPhi(\eta_1)}\bigr)^{1/(1-q)}. 
\end{equation}
\end{remark}

\subsection{Foia\c{s}--Prodi type  estimate}
\label{s9.1}
Let $H_N$ be a sequence of finite-dimensional spaces satisfying Condition~(B) with some    $\gamma_N$, let 
${\mathsf P}_N:H\to H$ be the orthogonal projection onto~$H_N$, and let ${\mathsf Q}_N:=I-{\mathsf P}_N$. We consider some sequences  $u_k, u_k', \zeta_k,\zeta_k'\in H$  such that 
$$
u_k=S(u_{k-1})+\zeta_k,\quad u_k'=S(u_{k-1}')+\zeta_k'.
$$
\begin{lemma}  
Suppose that
\begin{align}
&{\mathsf P}_N u_j={\mathsf P}_Nu_j',\quad {\mathsf Q}_N\zeta_j={\mathsf Q}_N\zeta_j'\quad \mbox{for $1\le j\le n$},
\label{E:8.5}
\end{align}
where  $N\ge1$ is an integer. Then
\begin{align}
\label{E:8.6}
\|u_n-u_n'\|&\le\gamma_N^{- n}  \exp\biggl(\,\sum_{j=0}^{n-1}( \pppp(u_j )+\pppp(u_j'))\biggr)\|u_0-u_0'\|.
\end{align}  
\end{lemma}

\begin{proof} 
From~\eqref{E:8.5} and~\eqref{1.4} it follows that  
\begin{align*}
\|u_n-u_n'\|&= \|Q_N(u_n-u_n')\|=\|Q_N(S(u_{n-1})-S(u_{n-1}'))\|\nonumber\\
&\le  \gamma_N^{-1}\exp\bigl\{\pppp(u_{n-1})+\pppp(u_{n-1}')\bigr\}\|u_{n-1}-u_{n-1}'\|. 
\end{align*}
Iteration of this inequality results in~\eqref{E:8.6}.
\end{proof}
 
\subsection{A property of convex functions}
\label{s9.2}
Given a convex function $f:\R^n\to\R\cup\{+\infty\}$, let $D(f)=\{x\in\R^n:f(x)<\infty\}$. It is clear that either $D(f)=\varnothing$ or~$D(f)$ is a convex subset of~$\R^n$. We define the {\it relative interior\/} of~$D(f)$, denoted by $\RInt D(f)$, as the interior of the set~$D(f)$ considered as a subset of its affine hull. The {\it conjugate function\/} of~$f$ is defined by
$$
f^*(y)=\sup_{x\in\R^n}\bigl(\langle x,y\rangle_n-f(x)\bigr). 
$$
The {\it subdifferential\/} of~$f$ at a point $x\in D(f)$, denoted by~$\p f(x)$, is the (convex) set of vectors $y\in\R^n$ such that
$$ 
f(z)-f(x)\ge \langle z-x,y\rangle_n\quad\mbox{for any $z\in\R^n$}. 
$$
The following proposition used in Section~\ref{s4} is a particular case of more general results established in Section~23 of~\cite{rockafellar1997}.

\begin{proposition} \label{p9.2}
Let $f:\R^n\to\R\cup\{+\infty\}$ be a lower semicontinuous convex  function and let $x\in D(f)$. Then for any $\delta>0$ there are $x_\delta\in D(f)$ and $y_\delta\in \R^n$ such that 
\begin{equation} \label{9.11}
f(x_\delta)<f(x)+\delta, \quad 
f^*(y_\delta)=\langle x_\delta,y_\delta\rangle_n-f(x_\delta). 
\end{equation}
\end{proposition}

\begin{proof}
We first assume that $D(f)=\{x\}$. In this case, we have 
$$
f^*(y)=\sup_{z\in D(f)}\bigl(\langle z,y\rangle_n-f(z)\bigr)=\langle x,y\rangle_n-f(x),
$$
so that one can take $x_\delta=x$ for any $\delta>0$ and an arbitrary $y_\delta\in\R^n$. 

\smallskip
We now assume that $D(f)$ contains more than one point. Using the continuity of a convex function of one variable on an open set, we can find $x_\delta\in \RInt D(f)$ such that the inequality in~\eqref{9.11} holds. 
Let us denote by~$E$ the vector span of $D(f)-x_\delta$ and define the directional  derivative
$$
f'(x_\delta;\xi)=\lim_{s\to0^+}\frac{f(x_\delta+s\xi)-f(x_\delta)}{s}, \quad \xi\in E.
$$ 
Then $\xi\mapsto f'(x_\delta;\xi)$ is a homogeneous lower semicontinuous  convex  function on~$E$ which is finite everywhere. It follows that 
$$
f'(x_\delta;\xi)=\sup_{z\in\p f(x_\delta)}\langle z,\xi\rangle_n.
$$
In particular, $\p f(x_\delta)$ is not empty. Any vector $y\in\p f(x_\delta)$ satisfies the equality in~\eqref{9.11}. 
\end{proof}

\addcontentsline{toc}{section}{Bibliography}
%\bibliography{references}

\begin{thebibliography}{JNPS14}

\bibitem[BDT95]{BDT-1995}
E.~Bolthausen, J.-D. Deuschel, and Y.~Tamura, \emph{Laplace approximations for
  large deviations of nonreversible {M}arkov processes. {T}he nondegenerate
  case}, Ann. Probab. \textbf{23} (1995), no.~1, 236--267.

\bibitem[BKL02]{BKL-2002}
J.~Bricmont, A.~Kupiainen, and R.~Lefevere, \emph{Exponential mixing of the
  2{D} stochastic {N}avier--{S}tokes dynamics}, Comm. Math. Phys. \textbf{230}
  (2002), no.~1, 87--132.

\bibitem[BV92]{BV1992}
A.~V. Babin and M.~I. Vishik, \emph{Attractors of {E}volution {E}quations},
  North-Holland Publishing, Amsterdam, 1992.

\bibitem[Caz03]{cazenave2003}
T.~Cazenave, \emph{Semilinear {S}chr\"odinger {E}quations}, Courant Lecture
  Notes in Mathematics, vol.~10, New York University Courant Institute of
  Mathematical Sciences, New York, 2003.

\bibitem[dA85]{Ac85}
A.~de~Acosta, \emph{{Upper bounds for large deviations of dependent random
  vectors}}, Z. Wahrsch. Verw. Gebiete \textbf{69} (1985), no.~4, 551--565.

\bibitem[dA90]{deacosta-1990}
\bysame, \emph{Large deviations for empirical measures of {M}arkov chains}, J.
  Theoret. Probab. \textbf{3} (1990), no.~3, 395--431.

\bibitem[Deb13]{debussche-2013}
A.~Debussche, \emph{Ergodicity results for the stochastic {N}avier-{S}tokes
  equations: an introduction}, Topics in mathematical fluid mechanics, Lecture
  Notes in Math., vol. 2073, Springer, Heidelberg, 2013, pp.~23--108.

\bibitem[DS89]{DS1989}
J.-D. Deuschel and D.~W. Stroock, \emph{{Large Deviations}}, Academic Press,
  Boston, 1989.

\bibitem[Dud02]{dudley2002}
R.~M. Dudley, \emph{Real {A}nalysis and {P}robability}, Cambridge University
  Press, Cambridge, 2002.

\bibitem[DV75]{DV-1975}
M.~D. Donsker and S.~R.~S. Varadhan, \emph{{Asymptotic evaluation of certain
  Markov process expectations for large time, I--II}}, Comm. Pure Appl. Math.
  \textbf{28} (1975), 1--47, 279--301.

\bibitem[DV76]{DV-1976}
\bysame, \emph{{Asymptotic evaluation of certain Markov process expectations
  for large time, III}}, Comm. Pure Appl. Math. \textbf{29} (1976), 389--461.

\bibitem[DZ00]{ADOZ00}
A.~Dembo and O.~Zeitouni, \emph{{Large Deviations Techniques and
  Applications}}, Springer--Verlag, Berlin, 2000.

\bibitem[EMS01]{EMS-2001}
W.~E, J.~C. Mattingly, and Ya. Sinai, \emph{Gibbsian dynamics and ergodicity
  for the stochastically forced {N}avier--{S}tokes equation}, Comm. Math. Phys.
  \textbf{224} (2001), no.~1, 83--106.

\bibitem[FK06]{FK2006}
J.~Feng and T.~G. Kurtz, \emph{Large {D}eviations for {S}tochastic
  {P}rocesses}, Mathematical Surveys and Monographs, vol. 131, American
  Mathematical Society, Providence, RI, 2006.

\bibitem[FM95]{FM-1995}
F.~Flandoli and B.~Maslowski, \emph{Ergodicity of the 2{D} {N}avier--{S}tokes
  equation under random perturbations}, Comm. Math. Phys. \textbf{172} (1995),
  no.~1, 119--141.

\bibitem[Gou07a]{gourcy-2007b}
M.~Gourcy, \emph{{A large deviation principle for 2D stochastic Navier--Stokes
  equation}}, Stochastic Process. Appl. \textbf{117} (2007), no.~7, 904--927.

\bibitem[Gou07b]{gourcy-2007a}
\bysame, \emph{{Large deviation principle of occupation measure for a
  stochastic Burgers equation}}, Ann. Inst. H. Poincar\'e Probab. Statist.
  \textbf{43} (2007), no.~4, 375--408.

\bibitem[Has80]{hasminski1980}
R.~Z. Has{\cprime}minski\u{\i}, \emph{Stochastic {S}tability of {D}ifferential
  {E}quations}, Sijthoff \& Noordhoff, Alphen aan den Rijn, 1980.

\bibitem[HM06]{HM-2006}
M.~Hairer and J.~C. Mattingly, \emph{Ergodicity of the 2{D} {N}avier--{S}tokes
  equations with degenerate stochastic forcing}, Ann. of Math. (2) \textbf{164}
  (2006), no.~3, 993--1032.

\bibitem[Jai90]{jain-1990}
N.~C. Jain, \emph{Large deviation lower bounds for additive functionals of
  {M}arkov processes}, Ann. Probab. \textbf{18} (1990), no.~3, 1071--1098.

\bibitem[JNPS12]{JNPS-2012}
V.~{Jak\v si\'c}, V.~Nersesyan, C.-A. Pillet, and A.~Shirikyan, \emph{{Large
  deviations from a stationary measure for a class of dissipative PDE's with
  random kicks}}, Preprint (2012), arXiv:1212.0527.

\bibitem[JNPS14]{JNPS-2013}
\bysame, \emph{{Large deviations and Gallavotti--Cohen principle for
  dissipative PDE's with rough noise}}, Comm. Math. Phys. (2014), accepted.

\bibitem[Kif90]{kifer-1990}
Y.~Kifer, \emph{Large deviations in dynamical systems and stochastic
  processes}, Trans. Amer. Math. Soc. \textbf{321} (1990), no.~2, 505--524.

\bibitem[KN13]{KN-2013}
S.~Kuksin and V.~Nersesyan, \emph{Stochastic {CGL} equations without linear
  dispersion in any space dimension}, Stochastic PDE: Anal. Comp. \textbf{1}
  (2013), no.~3, 389--423.

\bibitem[KS00]{KS-cmp2000}
S.~Kuksin and A.~Shirikyan, \emph{Stochastic dissipative {PDE}s and {G}ibbs
  measures}, Comm. Math. Phys. \textbf{213} (2000), no.~2, 291--330.

\bibitem[KS01]{KS-mpag2001}
\bysame, \emph{Ergodicity for the randomly forced 2{D} {N}avier--{S}tokes
  equations}, Math. Phys. Anal. Geom. \textbf{4} (2001), no.~2, 147--195.

\bibitem[KS12]{KS-book}
\bysame, \emph{Mathematics of {T}wo-{D}imensional {T}urbulence}, Cambridge
  University Press, Cambridge, 2012.

\bibitem[Lio69]{lions1969}
J.-L. Lions, \emph{Quelques {M}\'ethodes de {R}\'esolution des {P}robl\`emes
  aux {L}imites {N}on {L}in\'eaires}, Dunod, 1969.

\bibitem[LS06]{LS-2006}
A.~Lasota and T.~Szarek, \emph{Lower bound technique in the theory of a
  stochastic differential equation}, J. Differential Equations \textbf{231}
  (2006), no.~2, 513--533.

\bibitem[Lun09]{lunardi2009}
A.~Lunardi, \emph{{Interpolation Theory}}, Edizioni della Normale, Pisa, 2009.

\bibitem[MT93]{MT1993}
S.~P. Meyn and R.~L. Tweedie, \emph{Markov {C}hains and {S}tochastic
  {S}tability}, Springer-Verlag London, London, 1993.

\bibitem[Roc97]{rockafellar1997}
R.~T. Rockafellar, \emph{Convex {A}nalysis}, Princeton University Press,
  Princeton, NJ, 1997.

\bibitem[Shi04]{shirikyan-jmfm2004}
A.~Shirikyan, \emph{Exponential mixing for 2{D} {N}avier--{S}tokes equations
  perturbed by an unbounded noise}, J. Math. Fluid Mech. \textbf{6} (2004),
  no.~2, 169--193.

\bibitem[Sza97]{szarek-1997}
T.~Szarek, \emph{Markov operators acting on {P}olish spaces}, Ann. Polon. Math.
  \textbf{67} (1997), no.~3, 247--257.

\bibitem[Tay97]{taylor1996}
M.~E. Taylor, \emph{Partial {D}ifferential {E}quations. {I}--{III}},
  Springer-Verlag, New York, 1996-97.

\bibitem[Wu00]{wu-2000}
L.~Wu, \emph{{Uniformly integrable operators and large deviations for Markov
  processes}}, J. Funct. Anal. \textbf{172} (2000), 301--376.

\end{thebibliography}
%\bibliographystyle{amsalpha}
\def\cprime{$'$} \def\cprime{$'$}
  \def\polhk#1{\setbox0=\hbox{#1}{\ooalign{\hidewidth
  \lower1.5ex\hbox{`}\hidewidth\crcr\unhbox0}}}
  \def\polhk#1{\setbox0=\hbox{#1}{\ooalign{\hidewidth
  \lower1.5ex\hbox{`}\hidewidth\crcr\unhbox0}}}
  \def\polhk#1{\setbox0=\hbox{#1}{\ooalign{\hidewidth
  \lower1.5ex\hbox{`}\hidewidth\crcr\unhbox0}}} \def\cprime{$'$}
  \def\polhk#1{\setbox0=\hbox{#1}{\ooalign{\hidewidth
  \lower1.5ex\hbox{`}\hidewidth\crcr\unhbox0}}} \def\cprime{$'$}
  \def\cprime{$'$} \def\cprime{$'$} \def\cprime{$'$}
\providecommand{\bysame}{\leavevmode\hbox to3em{\hrulefill}\thinspace}
\providecommand{\MR}{\relax\ifhmode\unskip\space\fi MR }
% \MRhref is called by the amsart/book/proc definition of \MR.
\providecommand{\MRhref}[2]{%
  \href{http://www.ams.org/mathscinet-getitem?mr=#1}{#2}
}
\providecommand{\href}[2]{#2}

\end{document}